\documentclass[12pt]{article}

\usepackage{sectsty}
\sectionfont{\centering}

\usepackage{graphicx} 
\usepackage{amsthm}
\usepackage{amsmath}
\usepackage{amsfonts}
\usepackage{amssymb}
\usepackage{enumerate}
\usepackage{esint}
\usepackage{xcolor}
\usepackage{float} 
\usepackage{geometry}
\usepackage{hyperref}
\usepackage{mdwlist}
\numberwithin{equation}{section} 

\newtheorem{theorem}{Theorem}[section]

\newtheorem{proposition}{Proposition}[section]

\newtheorem{lemma}{Lemma}[section]

\newtheorem{corollary}{Corollary}[section]

\theoremstyle{definition}
\newtheorem{definition}{Definition}[section]

\theoremstyle{remark}
\newtheorem{remark}{Remark}[]

\newcommand{\dist}{\mathrm{dist}}
\newcommand{\mres}{\mathbin{\vrule height 1.6ex depth 0pt width
0.13ex\vrule height 0.13ex depth 0pt width 1.3ex}}     

\newcommand{\R}{\mathbb{R}}
\newcommand{\lmin}{\lambda_{\mathrm{min}}}
\newcommand{\lmax}{\lambda_{\mathrm{max}}}
\newcommand{\de}{\mathrm{d}}

\title{On the geometry of measures with density bounds in a Hölder anisotropic setting} 
\author{Ignacio Tejeda}
\date{}

\begin{document}

\maketitle
\begin{abstract}
    We study the regularity of the support of a Radon measure $\mu$ on $\mathbb R^{n+1}$ for which anisotropic versions of its $n$-dimensional density ratio and its doubling character are assumed to converge with H\"older rate. We show that in either case, if the support of $\mu$ is flat enough, then it is a $C^{1,\gamma}$ $n$-dimensional submanifold of $\mathbb R^{n+1}$, for some $\gamma\in (0,1)$. If the flatness assumption is dropped, then the support of $\mu$ is the union of a $C^{1,\gamma}$ $n$-dimensional submanifold of $\mathbb R^{n+1}$ and a closed singular set that is either empty if $n\leq 2$, or has Hausdorff dimension at most $n-3$ if $n\geq 3$.
\end{abstract}

\section{Introduction}\label{introduction}
    Let $\mu$ be a Radon measure on $\R^{n+1}$. We consider the problem of characterizing geometric properties of $\mu$ with the behavior of its $m$-dimensional density. Traditionally, this quantity is defined as
    \begin{equation}\Theta^m(\mu,X)=\lim_{r\searrow 0}\frac{\mu(B(X,r))}{\omega_mr^m},\label{introduction - round density definition}\end{equation}
    provided that the limit exists, where $\omega_m$ denotes the $m$-dimensional Lebesgue measure of the unit ball in $\R^m$, and $B(X,r)$ is an Euclidean open ball of radius $r$ and center $X$ in $\R^{n+1}$. If the limit does not exist, one can consider the lower and upper densities of $\mu$, $\Theta^{m}_*(\mu,\cdot)$ and $\Theta^{*m}(\mu,\cdot)$, obtained by replacing the limit in \eqref{introduction - round density definition} with $\liminf$ or $\limsup$ as $r\searrow 0$, respectively, both of which always exist.
    
    In the context of this work, much of the geometric information about a measure $\mu$ is contained in its \textit{support}, the set
    $$\mathrm{spt}(\mu)=\{X\in\mathbb R^{n+1}: \mu(B(X,r))>0,\text{ for all }r>0\}.$$
    Intuitively, if the ratio $\frac{\mu(B(X,r))}{\omega_mr^m}$ behaves well, one can expect $\mathrm{spt}(\mu)$ to behave as a set of Hausdorff dimension $m$ near $X$, possibly with good regularity properties depending on the asymptotic behavior of $\frac{\mu(B(X,r))}{\omega_mr^m}$ as $r\searrow 0$.
    
    Results in this direction originated with the seminal work of Besicovitch in \cite{Bes28}, \cite{Bes38}, \cite{Bes39}, where he showed that if $m=1$, $n+1=2$ and $\mu = \mathcal{H}^1\mres\Sigma$ with $0<\mathcal{H}^1(\Sigma)<\infty$, then the existence, positivity and finiteness $\mathcal{H}^1-$almost everywhere of $\Theta^1(\mu,\cdot)$  on $\Sigma$ is equivalent to the $1-$rectifiability of $\mu$. After several decades, work of various authors including Marstrand \cite{Mar61}, Mattila \cite{Mat75} and Preiss \cite{Pre87} culminated in a deep result of Preiss, stating that given any integer $1\leq m\leq n+1$ and any Radon measure $\mu$ on $\R^{n+1}$, the $\mu-$almost everywhere existence, positivity and finiteness of $\Theta^m(\mu,\cdot)$ is equivalent to the $m-$rectifiability of $\mu$ (see also notes by De Lellis in \cite{De08}). This completed the picture in the qualitative setting of rectifiability.

    More recently, work has been done in connection with densities and other analytic quantities in quantitative settings. Tolsa showed in \cite{Tol15} that the so-called weak density condition implies uniform rectifiability for Ahlfors-David regular measures, extending a result of David and Semmes (\cite{DS91}, \cite{DS92}) to arbitrary dimensions. In a different direction, higher order rectifiability and parametrization results have been obtained by David, Kenig and Toro \cite{DKT01},  Ghinassi \cite{Ghi20}, Del Nin and Idu \cite{DI22}, Hoffman \cite{Hof24} and Lewis \cite{Le13}.
    
    In \cite{DKT01}, the authors showed that if there exists $\alpha\in(0,1)$ such that $\mu$ locally satisfies
    \begin{equation}\label{introduction - dkt assumption}
        \left|\frac{\mu(B(X,r))}{\omega_nr^n}-1\right| \leq Cr^\alpha, \quad X\in\Sigma=\mathrm{spt}(\mu),
    \end{equation}
    for small $r>0$, then under a suitable flatness assumption, $\Sigma$ is a $C^{1,\gamma}-$submanifold of $\R^{n+1}$ of dimension $n$, where $\gamma\in(0,1)$ depends on $\alpha$. Notice that \eqref{introduction - dkt assumption} implies that $\Theta^n(\mu,\cdot)=1$ everywhere on $\Sigma$, and it gives additional information on the rate at which this limit is attained. The flatness assumption needed in \cite{DKT01} is that $\Sigma$ is Reifenberg flat of dimension $n$, with a constant\footnote{Although their results are stated with the assumption that $\Sigma$ is Reifenberg-flat with vanishing constant, one can check that the vanishing condition is not necessary in their proofs.} that is small enough depending on $n$ (see Section \ref{prelim} or Reifenberg's work in \cite{Rei60}). This assumption helps ensure that $\Sigma$ does not have many holes \cite{Rei60}, as well as ruling out cone singularities \cite{KP87}.

    More generally, it is shown in \cite{DKT01} that the same conclusion about $\Sigma$ holds if $\mu$ obeys a quantitative form of \textit{asymptotic optimal doubling}.
    \begin{definition}
        A Radon measure $\mu$ on $\mathbb R^{n+1}$ is \emph{asymptotically optimally doubling} of dimension $n$ if for every compact set $K\subset\mathbb R^{n+1}$,
        $$\limsup_{r\searrow 0}\left\lbrace\left|\frac{\mu(B(X,tr))}{\mu(B(X,r))}-t^n\right|:X\in \Sigma\cap K,\ \frac{1}{2}\leq t\leq 1\right\rbrace=0.$$
        Additionally, given $\alpha\in(0,1)$, $\mu$ is \emph{$\alpha$-H\"older asymptotically optimally doubling} of dimension $n$ if for every compact set $K\subset\mathbb R^{n+1}$ there exist constants $C_K>0$ and $r_K>0$ such that for every $r\in (0,r_K]$,
        
        \begin{equation}
            \label{introduction - holder asod definition}
            \sup\left\lbrace\left|\frac{\mu(B(X,tr))}{\mu(B(X,r))}-t^n\right|:X\in \Sigma\cap K,\ \frac{1}{2}\leq t\leq 1\right\rbrace\leq C_Kr^\alpha.
        \end{equation}
    \end{definition}

    In this work we consider conditions that are analogous versions of \eqref{introduction - dkt assumption} and \eqref{introduction - holder asod definition} in an anisotropic setting, where the balls used in both conditions are replaced with ellipses whose shape depends on their center. More precisely, we consider a matrix valued function $X\mapsto\Lambda(X)$, $X\in \R^{n+1}$, such that $\Lambda(X)$ is symetric, positive definite for every $X$. The ellipses are given by
    \begin{equation}
        B_{\Lambda}(X,r)=X+\Lambda(X)B(0,r),\quad r>0.
        \label{introduction - ellipses definition}
    \end{equation}
    The corresponding $m-$density is
    \begin{equation}
        \label{introduction - elliptic density definition}\Theta^m_{\Lambda}(\mu,X)=\lim_{r\searrow 0}\frac{\mu(B_{\Lambda}(X,r))}{\omega_mr^m},\quad X\in\Sigma,
    \end{equation}
   whenever the limit exists; otherwise, one could consider the corresponding lower and upper densities as in the Euclidean case. This type of density has been considered by Casey, Goering, Toro and Wilson in \cite{CGTW25}, where the authors showed that $m-$rectifiability can be characterized by the $\mu-$almost everywhere existence, positivity and finiteness of  $\Theta^m_{\Lambda}(\mu,\cdot)$. For our purposes, we will restrict our attention to the case $m=n$.
   
   Our arguments will also rely on the notion of Reifenberg flatness, defined in terms of the quantity
   \begin{equation*}
b\beta_\Sigma(X,r) = \inf_P\left\lbrace \frac{1}{r}D[\Sigma\cap B(X,r); P\cap B(X,r)]\right\rbrace.
\end{equation*}
Here $D[\cdot,\cdot]$ denotes Hausdorff distance and the infimum is taken over all $n-$planes containing $X$. Given a compact set $K\subset\R^{n+1}$ and a radius $r_0>0$, we denote
$$b\beta_\Sigma(K,r_0)=\sup_{r\in (0,r_0]}\sup_{X\in\Sigma\cap K}b\beta_\Sigma(X,r).$$
Some of our main results make reference to the following geometric condition:
\begin{equation}
    \begin{split}
    \label{introduction - flatness condition}
        &\text{For every compact set $K\subset \mathbb R^{n+1}$, there exists $r_K>0$ depending on $K$ and $\Lambda$,}\\ 
        &\text{such that $b\beta_\Sigma(K,r_K)<\delta_K$, where $\delta_K>0$ is a number determined by $K$ and $\Lambda$.}    \end{split}
\end{equation}
Note that any set $\Sigma$ satisfies \ref{introduction - flatness condition} with $\delta_K\geq 1$. On the other hand, if $\delta_K<1$, then \eqref{introduction - flatness condition} gives information on the flatness of $\Sigma$ at points in $\Sigma\cap K$.

\begin{theorem}\label{theorem 2}
       Suppose the mapping $X\mapsto\Lambda(X)$ is locally H\"older continuous with exponent $\beta\in (0,1)$. Assume that that there exists $\alpha\in(0,1)$ such that the following holds: for every compact set $K\subset\R^{n+1}$ there exists a constant $C_K>0$ such that for every $X\in\Sigma\cap K$, $t\in [\frac{1}{2},1]$ and $r\in(0,1]$,
       \begin{equation}
        \label{introduction - lambda holder asod}
        \left|\frac{\mu(B_\Lambda(X,tr))}{\mu(B_\Lambda(X,r))}-t^n\right|\leq C_Kr^\alpha.
    \end{equation}
    If $n\geq 3$, suppose additionally that $\Sigma$ satisfies \eqref{introduction - flatness condition} with $\delta_K$ small enough depending on $K$ and $\Lambda$. Then $\Sigma$ is a $C^{1,\gamma}$ $n$-dimensional submanifold of $\R^{n+1}$, for some $\gamma\in(0,1)$ depending on $\alpha$ and $\beta$.
   \end{theorem}

   \begin{theorem}\label{theorem 1}
       Suppose the mapping $X\mapsto\Lambda(X)$ is locally H\"older continuous with exponent $\beta\in (0,1)$. Assume that that there exists $\alpha\in(0,1)$ such that the following holds: for every compact set $K\subset\R^{n+1}$ there exists a constant $C_K>0$ such that for every $X\in\Sigma\cap  K$ and $r\in(0,1]$,
       \begin{equation}
        \label{introduction - main density estimate}\left|\frac{\mu(B_{\Lambda}(X,r))}{\omega_nr^n}-1\right| \leq C_Kr^\alpha.
    \end{equation}
    If $n\geq 3$, suppose additionally that $\Sigma$ satisfies \eqref{introduction - flatness condition} with $\delta_K$ small enough depending on $K$ and $\Lambda$. Then $\Sigma$ is a $C^{1,\gamma}$ $n$-dimensional submanifold of $\R^{n+1}$, for some $\gamma\in(0,1)$ depending on $\alpha$ and $\beta$.
   \end{theorem}

   \begin{remark}
       The H\"older continuity condition above and \eqref{introduction - main density estimate} will be often referred to as the continuity and density assumptions of Theorem \ref{theorem 1}.
   \end{remark}

   These are analogues of the corresponding results in \cite{DKT01}. The main novelty here is the ability to replace round balls with ellipses that change from point to point. This type of question lies in the framework of studying densities or other related analytic quantities determined by norms other than the Euclidean one. In our case, the associated norm depends on the point, and is given by $\|Z\|_X=|\Lambda(X)^{-1}Z|$, so that
   $$B_{\Lambda}(X,r)=\{Y\in\R^{n+1}:\|Y-X\|_X < r\}.$$

    As we will see, the proof of Theorem \ref{theorem 2} relies on Theorem \ref{theorem 1}. On the other hand, the proof of Theorem \ref{theorem 1} uses the following result of \cite{DKT01}.

    \begin{proposition}[\cite{DKT01} - Proposition 9.1]
\label{introduction - DKT beta to C1gamma regularity}
Let $\gamma\in (0,1]$. Suppose $\Sigma$ is a Reifenberg-flat set with vanishing constant of dimension $m$ in $\R^{n+1}$, $m\leq n+1$, and that for each compact set $K\subset\R^{n+1}$ there exist constants $C_K,r_K>0$ such that
\begin{equation}
    \label{introduction - beta estimate}
    \beta_\Sigma(X,r)\leq C_Kr^\gamma,
\end{equation}
for all $X\in K\cap\Sigma$ and $r\in (0,r_K]$. Then $\Sigma$ is a $C^{1,\gamma}$ submanifold of dimension $m$ of $\R^{n+1}$.
\end{proposition}

\begin{remark}
    It can be seen from the proof of this result that $\Sigma$ only needs to be Reifenberg flat with a constant that is small enough depending on the dimension $n$.
\end{remark}
    Thus, Theorem \ref{theorem 1} will be proven once we complete the following steps:
    \begin{enumerate}
        \item[Step 1.] Prove that \eqref{introduction - beta estimate} holds under the assumptions of Theorem \ref{theorem 1}.
        \item[Step 2.] Show that under the assumptions of Theorem \ref{theorem 1}, condition \eqref{introduction - flatness condition} implies that $\Sigma$ is Reifenberg flat with vanishing constant.
    \end{enumerate}
  Finally, we also prove a result that describes the case in which \eqref{introduction - lambda holder asod} is satisfied but no flatness assumption is made on $\Sigma$. This result contains anisotropic analogues of work of Preiss, Tolsa and Toro \cite{PTT08} and Nimer \cite{Ni18} when the codimension is $1$.

  \begin{theorem}
      \label{theorem 3}
      Suppose the mapping $X\mapsto\Lambda(X)$ is H\"older continuous with exponent $\beta\in(0,1)$. Assume that there exists $\alpha\in (0,1)$ such that the following holds: for every compact set $K\subset \mathbb R^{n+1}$ there exists a constant $C_K>0$ such that for every $X\in\Sigma\cap K$, $t\in [\frac{1}{2},1]$ and $r\in (0,1]$,
      $$\left|\frac{\mu(B_\Lambda(X,tr))}{\mu(B_\Lambda(X,r))}-t^n\right|\leq C_Kr^\alpha.$$
      Then $\Sigma=\mathcal{R}\cup\mathcal{S}$, where $\mathcal{S}$ is a closed set of Hausdorff dimension at most $n-3$ if $n\geq 3$, or $\mathcal{S}=\varnothing$ if $n\leq 2$, and $\mathcal{R}$ is a $C^{1,\gamma}$-submanifold of $\mathbb R^{n+1}$ of dimension $n$, for some $\gamma\in (0,1)$ depending on $\alpha$ and $\beta$.
  \end{theorem}

   The structure of the paper is as follows. Section \ref{prelim} contains technical lemmas that are needed later on, as well as definitions of relevant notions of flatness. Sections \ref{moments} and \ref{beta numbers} provide a proof of the fact that \eqref{introduction - beta estimate} holds under the assumptions of Theorem \ref{theorem 1}.  In sections \ref{pseudo tangents} and \ref{flatness} we show that under the assumptions of Theorem \ref{theorem 1}, condition \eqref{introduction - flatness condition} implies that $\Sigma$ is  Reifenberg flat with vanishing constant, and we prove Theorem \ref{theorem 1}. Section \ref{asod} shows how to derive Theorem \ref{theorem 2} from Theorem \ref{theorem 1}, and Sections \ref{nonflat case} and \ref{singular set} contain a proof of Theorem \ref{theorem 3}.
    
\section{Preliminaries}\label{prelim}

We will adopt the convention that any local constants depending on a compact set $K\subset\mathbb R^{n+1}$ may be denoted by $C_K$. Moreover, we may allow $C_K$ to depend on the matrix-valued function $\Lambda$, and any updates to the value of $C_K$ may be incorporated without changing notation.

\subsection{Measure theoretic background}

Recall that a Radon measure $\nu$ on $\R^{n+1}$ is called $m$-\textit{uniform}, for $0<m\leq n+1$, if there exists a constant $C>0$ such that for all $X\in\mathrm{spt}(\nu)$ and $r>0$,
\begin{equation}
\label{pseudo tangents - uniform measure}
    \nu(B(X,r)) = Cr^m.
\end{equation}
 Uniform measures resemble, in particular, the asymptotic behavior as $r\to 0$ of a measure $\mu$ that satisfies \eqref{introduction - dkt assumption}. A well known result of Marstrand \cite{Mar64} implies that if $\nu$ is $m$-uniform, then $m$ is an integer.
An important example of $m$-uniform measure is when $\nu$ is a \textit{flat} measure (of dimension $m$), i.e. there exists an $m$-plane $P$ and a constant $c>0$ such that $\nu=c\mathcal{H}^m\mres P$, where $\mathcal{H}^m$ denotes $m$-dimensional Hausdorff measure. Flat measures play a key role in the study of rectifiability (see for example \cite{Ma95} for a good summary). However, not all uniform measures are flat, as shown first by Preiss \cite{Pre87} and Kowalski and Preiss \cite{KP87}. Even though a complete classification of uniform measures is still an open problem, the following results summarize some of the main facts we need to know about uniform measures in this work.

\begin{theorem}[Preiss \cite{Pre87}]
\label{prelim - uniform measures of dimension 2}
    Let $\nu$ be an $m$-uniform measure on $\R^{n+1}$, $0\leq m\leq n+1$. If $m\leq 2$, then $\nu$ is flat of dimension $m$.
\end{theorem}

\begin{theorem}[Kowalski, Preiss \cite{KP87}]
\label{flat measures - kowalski preiss}
    Let $\nu$ be an $n$-uniform measure in $\R^{n+1}$.
    Then after a translation, rotation and multiplication by a constant, either
    \begin{equation}
    \label{flat measures - plane}
        \nu = \mathcal{H}^n\mres \{(x_1,\ldots, x_{n+1})\in\R^{n+1}: x_{n+1}=0\},
    \end{equation}
    or $n\geq 3$ and
    \begin{equation}
    \label{flat measures - light cone}
        \nu = \mathcal{H}^n\mres \{(x_1,\ldots, x_{n+1})\in\R^{n+1}: x_4^2= x_1^2 + x_2^2 + x_3^2\}.
    \end{equation}
\end{theorem}
 Theorem \ref{flat measures - kowalski preiss} characterizes $n$-uniform measures in codimension 1. The cone in \eqref{flat measures - light cone} is often referred to as the Kowalski-Preiss cone, or light cone. Other families of examples of $3$-uniform measures in arbitrary codimension have been found by A. Dali Nimer \cite{Ni22}, but our arguments will not rely on those directly.

The way in which uniform measures arise in our context is as \textit{tangent measures} of a measure that satisfies the assumptions of Theorems \ref{theorem 2}, \ref{theorem 1} or \ref{theorem 3}. Given a Radon measure $\mu$ on $\R^{n+1}$, $X\in\Sigma=\mathrm{spt}(\mu)$ and $r>0$, consider the map $T_{X,r}$ and the measure $\mu_{X,r}$ given by
\begin{equation}
    \label{nonflat case - round blow up map}
    T_{X,r}(Z)=\frac{Z-X}{r},\quad \mu_{X,r}=\frac{1}{\mu(B(X,r))}T_{X,r}[\mu],
\end{equation}
where $Z\in\mathbb R^{n+1}$, and $T_{X,r}[\mu]$ is the push-forward measure of $\mu$ via $T_{X,r}$. We say that $\nu$ is a tangent measure of $\mu$ at $X$ if $\mu_{X,r_k}\rightharpoonup\nu$ for some sequence $r_k>0$ with $r_k\to 0$, where ``$\rightharpoonup$" denotes weak convergence of Radon measures (see \cite{Ma95} for basic properties). It is convenient to note that if $\mu$ is $m$-uniform and $\nu$ is a tangent measure of $\mu$ as defined above, then $\nu(B(0,1))=1$.

The following is a standard notion of distance between Radon measures that we will need later in connection with tangent measures. For $r>0$, let $\mathcal{L}(r)$ denote the set of non-negative Lipschitz functions $\varphi$ on $\R^{n+1}$ with $\mathrm{spt}(\varphi)\subset B(0,r)$ and $\mathrm{Lip}(\varphi)\leq 1$. Given Radon measures $\alpha,\beta$ on $\R^{n+1}$, let
\begin{equation}
    \label{prelim - distance between measures}
\mathcal{F}_r(\nu,\tilde\nu)=\sup_{\varphi\in\mathcal{L
}(r)}\left|\int \varphi\de\nu-\int \varphi\de\tilde\nu\right|, \quad \mathcal{F}(\nu,\tilde\nu)=\sum_{k\geq 1}2^{-k}\mathcal{F}_{2^k}(\nu,\tilde\nu).
\end{equation}

As the following result shows (see \cite[Proposition 1.12]{Pre87}), these functionals metrize the weak convergence of Radon measures.

\begin{proposition}[Preiss \cite{Pre87}]\label{singular set - metric for weak convergence}

    Suppose $\nu,\nu_k$ are radon measures on $\R^{n+1}$, $k\in\mathbb N$. Then the following conditions are equivalent:
   
      $$\mathrm{(i)  }\  \nu_k\rightharpoonup\nu,\quad
       \mathrm{(ii) } \lim_k\mathcal{F}(\nu_k,\nu)=0,\quad \mathrm{(iii)}\ \text{For all }r>0,\ \lim_k \mathcal{F}_r(\nu_k,\nu)=0.$$
   
\end{proposition}

We will also need a result of Preiss based on the notion of \textit{tangent measure at infinity}. Recall that if $\mu,\nu$ are nonzero Radon measures on $\R^{n+1}$, then $\nu$ is a tangent measure of $\mu$ at infinity (of dimension $m$) if for every $X\in\R^{n+1}$,
$$\frac{1}{r^m}T_{X,r}[\mu]\rightharpoonup\nu,$$
as $r\to\infty$. Note that \textit{a priori}, different sequences of radii may lead to different limits. However, Preiss showed that this is not the case if $\mu$ is $m$-uniform. Recall that an $m$-uniform measure $\nu$ is \textit{conical} if for every $r>0$,
\begin{equation}
    \label{nonflat case - conical measure}
\frac{1}{r^m}T_{0,r}[\nu]=\nu.
\end{equation}

\begin{theorem}[Preiss \cite{Pre87}]
    \label{nonflat case - uniqueness of tangents for uniform measures}
    Let $\nu$ be an $m$-uniform measure in $\R^{n+1}$. Then $\nu$ has a unique tangent measure $\tilde\nu$ at $\infty$, where $\tilde\nu$ is $m$-uniform and conical. Moreover, for every $X\in\mathrm{spt}(\nu)$, $\nu$ has a unique tangent measure at $X$, which is also $m$-uniform and conical.
\end{theorem}

The question of whether an $m$-uniform measure is flat, is only relevant when $m\geq 3$, by Theorem \ref{prelim - uniform measures of dimension 2}. A key result of Preiss in this direction says that for an $m$-uniform measure $\nu$, $m\geq 3$, if $\nu$ is close enough to being flat at infinity, then it is flat.

\begin{theorem}[Preiss \cite{Pre87}]
\label{nonflat case - theorem of preiss}
     If $3\leq m\leq n+1$, there exists a constant $\varepsilon_0>0$ depending only on $n$ and $m$ such that if $\nu$ is an $m$-uniform measure on $\mathbb{R}^{n+1}$ with $\nu(B(X,1))=1$ for $X\in\mathrm{spt}(\nu)$, for which its tangent measure $\tilde\nu$ at $\infty$ satisfies
    \begin{equation}
        \label{nonflat case - preiss closeness to plane}
        \tilde F(\tilde\nu):=\min_{P}\int_{B(0,1)}\mathrm{dist}(X,P)^2\de\tilde\nu(X)\leq\varepsilon_0^2,
    \end{equation}
    then $\nu$ is flat. Here, the minimum is taken over all $m$-planes $P$ in $\mathbb R^{n+1}$ with $0\in P$.
 \end{theorem}

De Lellis introduced in \cite[Lemma 6.20]{De08} a smooth version of $\tilde F$ that is convenient to work with. Let $\varphi\in C^\infty_c(\mathbb R^{n+1})$ be a radially non-increasing function such that $\chi_{B(0,1)}\leq\varphi\leq\chi_{B(0,2)}$. Given a Radon measure $\nu$ with $0\in\mathrm{spt}(\nu)$, let
\begin{equation}
    \label{singular set - functional}
    F(\nu)=\min_{P}\frac{1}{\nu(B(0,1))}\int\varphi(Z)\mathrm{dist}(Z,P)^2\de\nu(Z),
\end{equation}
where the minimum is taken over all $m$-planes $P$ in $\mathbb R^{n+1}$ with $0\in P$. In contrast with $\tilde F$, $F$ has the property of being continuous with respect to sequential weak convergence of Radon measures. Moreover, $\nu(B(0,1))F(\nu)\geq \tilde F(\nu)$ for any Radon measure $\nu$, so Theorem \ref{nonflat case - theorem of preiss} can be reformulated as follows.

\begin{corollary}[De Lellis \cite{De08}]
    Let $\nu$ be an $m$-uniform measure on $\R^{n+1}$ with $\nu(B(0,1))=1$. If $3\leq m\leq n+1$, there exists $\varepsilon_0>0$, depending only on $m$ and $n$, such that if 
    \begin{equation}
        \label{singular set - Preiss reformulation}
        \limsup_{R\to\infty}F(\nu_{0,R})\leq\varepsilon^2_0,
    \end{equation}
    then $\nu$ is flat.
\end{corollary}

\subsection{The matrix-valued function $\Lambda$}
Let $\mathrm{GL}(n+1,\mathbb R)$ denote the space of $(n+1)\times(n+1)$ real invertible matrices, endowed with the operator norm
$$\|A\|=\sup_{\substack{V\in\mathbb R^{n+1}\\ |V|\leq 1}}|AV|,\quad A\in\mathrm{GL}(n+1,\R),$$
where $|\cdot|$ denotes Euclidean norm in $\R^{n+1}$. We consider a mapping $\Lambda:\mathbb R^{n+1}\to \mathrm{GL}(n+1, \R)$ with the property that $\Lambda(X)$ is a symmetric positive definite matrix for each $X\in\R^{n+1}$. In particular, all the eigenvalues of $\Lambda(X)$ are real and positive. We also assume that $\Lambda$ is locally H\"older continuous with exponent $\beta\in (0,1)$, in the sense that for each compact set $K\subset\R^{n+1}$ there exists a constant $H_K>0$ such that for all $X,Y\in K$,

\begin{equation}
    \label{prelim - holder continuity of lambda}
    \|\Lambda(X) - \Lambda(Y)\|\leq H_K|X-Y|^\beta.
\end{equation}

Important properties of $\Lambda$ will be encoded in the smallest and largest eigenvalues of $\Lambda(X)$ at a given point $X$, which we will denote by $\lambda_{\text{min}}(X)$ and  $\lambda_{\text{max}}(X)$, respectively. \\

\begin{lemma}[Regularity of eigenvalues]
\label{prelim - regularity of eigenvalues}
For all $X,Y\in\R^{n+1}$ we have
$$|\lmin(X)- \lmin(Y)|\leq \|\Lambda(X) - \Lambda(Y)\|,$$
$$|\lmax(X)- \lmax(Y)|\leq \|\Lambda(X) - \Lambda(Y)\|.$$
\end{lemma}

These estimates and the continuity assumption \eqref{prelim - holder continuity of lambda} imply that the functions $\lmin(\cdot)$ and $\lmax(\cdot)$ are locally H\"older continuous with exponent $\beta$. From this and from the fact that $\Lambda(X)$ is an invertible matrix for every $X\in\mathbb R^{n+1}$, it follows that $\lmin(\cdot),\ \lmin(\cdot)^{-1},\ \lmax(\cdot)$ and $ \lmax(\cdot)^{-1}$ are locally bounded from above and below by positive constants, and $\lmin(\cdot)^{-1}$ and $\lmax(\cdot)^{-1}$ are also locally H\"older continuous with exponent $\beta$. These considerations will be used in many of our estimates.

\begin{proof}[Proof of Lemma \ref{prelim - regularity of eigenvalues}]
    For $\lmax$ we can write $\lmax(X) = \|\Lambda(X)\|$, so the second estimate in the statement follows from triangle inequality. As for $\lmin$, notice that $1/\lmin(X)$ is the largest eigenvalue of $\Lambda(X)^{-1}$, so $1/\lmin(X)=\|\Lambda(X)^{-1}\|$. Therefore
    \begin{align*}
        |\lmin(X)-\lmin(Y)| &= \left\lvert\frac{1}{\|\Lambda(X)^{-1}\|} - \frac{1}{\|\Lambda(Y)^{-1}\|}\right\rvert= \frac{|\|\Lambda(X)^{-1}\| - \|\Lambda(Y)^{-1}\||}{\|\Lambda(X)^{-1}\|\|\Lambda(Y)^{-1}\|}\\
        &\leq \frac{\|\Lambda(X)^{-1} - \Lambda(Y)^{-1}\|}{\|\Lambda(X)^{-1}\|\|\Lambda(Y)^{-1}\|}= \frac{|\Lambda(X)^{-1}(I-\Lambda(X)\Lambda(Y)^{-1})\|}{\|\Lambda(X)^{-1}\|\|\Lambda(Y)^{-1}\|}\\
        &=\frac{\|\Lambda(X)^{-1}(\Lambda(Y) - \Lambda(X))\Lambda(Y)^{-1}\|}{\|\Lambda(X)^{-1}\|\|\Lambda(Y)^{-1}\|}\leq \|\Lambda(X) - \Lambda(Y)\|.
    \end{align*}
\end{proof}
The main role of the mapping $\Lambda$ in our context is to determine the ellipses $B_\Lambda(X,r)$ in \eqref{introduction - ellipses definition}. In particular, the regularity of $\Lambda$ ensures that these ellipses enjoy some compatibility, as the next lemma shows.

\begin{lemma}[Nested nonconcentric ellipses]\label{prelim - lemma on nonconcentric ellipses}
Suppose $\Lambda$ is H\"older continuous as in \eqref{prelim - holder continuity of lambda}. Let $K\subset\R^{n+1}$ be compact. If $X,Y\in K$, $r>0$ and $|X-Y|<C_Kr$ for some constant $C_K>0$ depending on $K$, then
\begin{equation}
    \label{prelim - nested ellipses, larger radius}
    B_\Lambda(X,r)\subset B_\Lambda(Y, r + \lmin(X)^{-1}|X-Y| + C_K r^{1+\beta}).
\end{equation}
If in addition $|X-Y|\leq \lmin(X)r/2$ and $r$ is small enough  depending on $K$ and $\Lambda$, then 
$$r - \lmin(X)^{-1}(X)|X-Y| - C_Kr^{1+\beta}>0$$
and
\begin{equation}
    \label{prelim - nested ellipses, smaller radius} B_\Lambda(X,r)\supset B_\Lambda(Y, r - \lmin(X)^{-1}|X-Y| - C_Kr^{1+\beta}).
\end{equation}
\end{lemma}

\begin{proof}
    Let $Z\in B_\Lambda(X,r)$. Write $Z = X + \Lambda(X) W$, where $W\in B(0,r)$. Then
    \begin{equation*}
        Z = Y + X - Y + \Lambda(X)W= Y + \Lambda(Y)[\Lambda(Y)
^{-1}(X - Y + \Lambda(X)W)].
\end{equation*}
Estimating the term in the brackets and keeping in mind the continuity of $\Lambda$, we get
\begin{align*}
    |\Lambda(Y)^{-1}(X-Y + \Lambda(X)W)| &\leq |\Lambda(Y)^{-1}(X-Y)| + |\Lambda(Y)^{-1}\Lambda(X)W|\\
    &\leq \lmin(Y)^{-1}|X - Y| + |(\Lambda(Y)^{-1}(\Lambda(X) - \Lambda(Y)) + I)W |\\
    &\leq \lmin(Y)^{-1}|X-Y| + H_K\lmin(Y)^{-1}|X-Y|^\beta|W| + |W|\\
    &\leq (\lmin(X)^{-1} + H_K|X-Y|^\beta)|X-Y|\\
    &\ \ \ \ \ \ \ \ \ \ \ \  \ \ \ \ \ \ \ \ \ \ \ +H_K\lmin(Y)^{-1}|X-Y|^\beta|W| + |W|\\
    &\leq \lmin(X)^{-1}|X-Y| + C_Kr^{1+\beta} + |W|\\
    &\leq \lmin(X)^{-1}|X-Y| + C_Kr^{1+\beta} + r.
\end{align*}
This implies that
$$Z\in Y + \Lambda(Y)B(0, r + \lmin(X)^{-1}|X-Y| + C_Kr^{1+\beta}),$$
which proves \eqref{prelim - nested ellipses, larger radius}. To prove \eqref{prelim - nested ellipses, smaller radius}, let $Z\in B_\Lambda(Y,\rho)$, with $\rho>0$ to be determined. Write 
\begin{equation*}Z=Y + \Lambda(Y)W = X + \Lambda(X)[\Lambda(X)^{-1}(Y-X+\Lambda(Y)W)],
\end{equation*}
where $W\in B(0,\rho)$, and estimate similarly as before
\begin{equation}
\label{prelim - upper estimate 1}
\begin{split}
    |\Lambda(X)^{-1}(Y - X +\Lambda(Y)W)|&\leq |\Lambda(X)^{-1}(Y - X)| + |\Lambda(X)^{-1}\Lambda(Y)W|\\
    &\leq |\Lambda(X)^{-1}(Y - X)| + |\Lambda(X)^{-1}(\Lambda(Y)-\Lambda(X))W| + |W|\\
    &\leq \lmin(X)^{-1}|X-Y| + H_K\lmin(X)^{-1}|X-Y|^\beta|W| + |W|\\
    &< \lmin(X)^{-1}|X-Y| + C_Kr^{1+\beta} + \rho.
\end{split}
\end{equation}
We would like this upper bound not to exceed $r$, which can be achieved by choosing 
$$\rho = r - \lmin(X)^{-1}|X-Y| - C_Kr^{1+\beta}.$$
Notice that by our assumptions, if $r$ is small enough depending on $K$ and $\Lambda$, we have

$$\rho \geq \frac{r}{2} - C_Kr^{1+\beta}>0.$$
With this choice of $\rho$, it follows from \eqref{prelim - upper estimate 1} that $Z\in X+\Lambda(X)B(0,r)$,
which completes the proof of the lemma.
\end{proof}

\subsection{Flatness notions}
\label{prelim - flatness notions}
To conclude this section we collect some necessary definitions and basic facts about flatness conditions. Given a closed set $\Sigma\subset\R^{n+1}$, for each $X\in\Sigma$ and $R>0$ let
\begin{equation}
\label{prelim - definition of theta numbers}
b\beta_\Sigma(X,r) = \inf_P\left\lbrace \frac{1}{r}D[\Sigma\cap B(X,r); P\cap B(X,r)]\right\rbrace,
\end{equation}
where the infimum is taken over all $n$-planes $P$ through $X$. Here $D$ denotes Hausdorff distance between two closed sets $A$ and $B$, given by 

$$D[A,B] = \max\left\lbrace \sup_{X\in A}\mathrm{dist}(X,B),\sup_{Y\in B}\mathrm{dist}(Y,A)\right\rbrace,$$
where $\mathrm{dist}(X,B) = \inf_{Y\in B}|X-Y|$ and similarly for $\mathrm{dist}(Y,A)$. We will also denote the closed $\varepsilon$-neighborhood of a set $E\subset\R^{n+1}$ by
\begin{equation}\label{prelim - notation for neighborhood}(E;\varepsilon)=\{Z\in \R^{n+1}:\mathrm{dist}(Z,E)\leq\varepsilon\}.\end{equation}

The quantity $b\beta_\Sigma(X,r)$ measures bilateral flatness of $\Sigma$ in Euclidean balls, and it is the main ingredient in the notions of \textit{$\delta$-Reifenberg flatness} or \textit{vanishing Reifenberg flatness} (see \cite{Rei60}). We will also need the following anisotropic version of $b\beta_\Sigma$,
$$b\beta_{\Sigma,\Lambda}(X,r)= \inf_P\left\lbrace \frac{1}{r}D[\Sigma\cap B_\Lambda(X,r); P\cap B_\Lambda(X,r)]\right\rbrace,$$
where the infimum is again taken over all $n$-planes $P$ through $X$. The only difference between this quantity and $b\beta_\Sigma(X,r)$ is that $B(X,r)$ is now replaced by $B_\Lambda(X,r)$.

The following lemma provides a way to compare $b\beta_\Sigma$ with $b\beta_{\Sigma,\Lambda}$. Its statement and many estimates later on make reference to the following eigenvalue bounds, associated with any compact set $K\subset\R^{n+1}$,
\begin{equation}
    \label{prelim - lmin and lmax of K}\lmin(K)=\min_{X\in(\Sigma\cap K;1)}\lmin(X),\quad \lmax(K)=\sup_{X\in(\Sigma\cap K;1)}\lmax(X),
\end{equation}
as well as a local notion of eccentricity of $\Lambda$,
\begin{equation}
\label{prelim - eccentricity}
    e_\Lambda(K)=\frac{\lmax(K)}{\lmin(K)}.
\end{equation}
The fact that some of these quantities consider a neighborhood of $\Sigma\cap K$ as opposed to just $\Sigma\cap K$ will become relevant in later sections.

\begin{lemma}[Euclidean and anisotropic flatness]
\label{prelim - lambda and normal reif flat}
    Let $\Sigma\subset\R^{n+1}$ be closed and let $K\subset\R^{n+1}$ be compact. Then there exists a constant $\delta_K>0$ depending on $K$ and $\Lambda$ with the following property. Let $\delta\in (0,\delta_K)$, $X,\overline{X}\in\Sigma\cap K$, $r>0$, $r'=\lmax(K)r$, $r''=\lmin(K)r$, and let $P$ be an $n$-plane through $X$. Let us denote $B_\Lambda(X,\overline{X},r)=X+\Lambda(\overline{X})B(0,r)$.
    \begin{enumerate}
        \item If
        \begin{equation}
    \label{prelim - round flatness assumption}
        D[\Sigma\cap B(X,r');P\cap B(X,r')]\leq \delta r',
    \end{equation}
    then
    \begin{equation}
    \label{prelim - elliptic flatness conclusion}
        D[\Sigma\cap B_\Lambda(X,\overline{X},r);P\cap B_\Lambda(X,\overline{X},r))]\leq (2+e_\Lambda(K))\delta r'.
    \end{equation}
    \item If 
    \begin{equation}
    \label{prelim - elliptic flatness assumption}
        D[\Sigma\cap B_\Lambda(X,\overline{X},r);P\cap B_\Lambda(X,\overline{X},r))]\leq \delta r,
    \end{equation}
    then
    \begin{equation}
    \label{prelim - round flatness conclusion}
        D[\Sigma\cap B(X,r'');P\cap B(X,r'')]\leq 2\delta r.
    \end{equation}
    \end{enumerate}
    Moreover, $\delta_K$ can be taken to be
    \begin{equation}
        \label{prelim - value of delta_K}\delta_K=\min\{\lmin(K),e_\Lambda(K)^{-1}\}.
    \end{equation}
\end{lemma}
The following corollary is a direct consequence of this lemma in the case $X=\overline{X}$.

\begin{corollary}
\label{prelim - lambda and normal reif flat corollary}
    Let $\Sigma\subset\R^{n+1}$ be closed and let $K\subset\R^{n+1}$ be compact. Then there exists a constant $\delta_K>0$ depending on $K$, $\Lambda$ and $n$ with the following property. Let $\delta\in (0,\delta_K)$, $X\in\Sigma\cap K$, $r>0$, $r'=\lmax(K)r$, $r''=\lmin(K)r$, and let $P$ be an $n$-plane through $X$.
    \begin{enumerate}
        \item If
        \begin{equation}
    \label{prelim - round flatness assumption}
        D[\Sigma\cap B(X,r');P\cap B(X,r')]\leq \delta r',
    \end{equation}
    then
    \begin{equation}
    \label{prelim - elliptic flatness conclusion}
        D[\Sigma\cap B_\Lambda(X,r);P\cap B_\Lambda(X,r))]\leq (2+e_\Lambda(K))\delta r'.
    \end{equation}
    \item If 
    \begin{equation}
    \label{prelim - elliptic flatness assumption}
        D[\Sigma\cap B_\Lambda(X,r);P\cap B_\Lambda(X,r))]\leq \delta r,
    \end{equation}
    then
    \begin{equation}
    \label{prelim - round flatness conclusion}
        D[\Sigma\cap B(X,r'');P\cap B(X,r'')]\leq 2\delta r.
    \end{equation}
    \end{enumerate}   
\end{corollary}

\begin{proof}[Proof of Lemma \ref{prelim - lambda and normal reif flat}]
    The proof will make repeated use of the fact that with $X,\overline{X},r,r'$ and $r''$ as in the statement, we have
    \begin{equation}
        \label{prelim - basic inclusion}
        B(X,r'')\subset B_\Lambda(X,\overline{X},r)\subset B(X,r').
    \end{equation}
    Let us first prove \eqref{prelim - elliptic flatness conclusion} under the assumption that \eqref{prelim - round flatness assumption} holds. We proceed in two steps.\\
    \\
    1. First we show that
    \begin{equation}
        \label{prelim - round and elliptic flatness, first proof conclusion}
        \Sigma\cap B_\Lambda(X,\overline{X},r)\subset (P\cap B_\Lambda(X,\overline{X},r);(1+e_\Lambda(K))\delta r').
    \end{equation}
    Let $Y\in \Sigma\cap B_\Lambda(X,\overline{X},r)$, and write $Y= X+Y_{||} + Y_\perp$,
    where $Y_{||}$ and $Y_{\perp}$ are parallel and orthogonal, respectively, to $P$. We consider two cases.\\
    \\
    Case (i): $X+Y_{||}\in P\cap B_\Lambda(X,\overline{X},r)$. In this situation, using that $B_\Lambda(X,\overline{X},r)\subset B(X,r')$ and \eqref{prelim - round flatness assumption}, we see that
    \begin{equation*}
            \dist(Y,P\cap B_\Lambda(X,\overline{X},r)) = |Y_\perp|= \dist(Y,P\cap B(X,r'))
            \leq \delta r',
    \end{equation*}
    which implies \eqref{prelim - elliptic flatness conclusion}.\\
    \\
    Case (ii): $X+Y_{||}\notin P\cap B_\Lambda(X,\overline{X},r)$, or equivalently, $X+Y_{||}\notin B_\Lambda(X,\overline{X},r)$. Now in addition to $|Y_\perp|$, we also need to control the distance from $X+Y_{||}$ to $P\cap B_\Lambda(X,\overline{X},r)$. Write 
    $$X+Y_{||} = X + \Lambda(X)W,$$
    and notice that $X+Y_{||}\notin B_\Lambda(X,\overline{X},r)$ implies $|W|=|\Lambda(\overline{X})^{-1}(Y_{||})|>r$. Let  $Y' = X+\Lambda(\overline{X})W'$, where $W'=rW/|W|$. Note that $|\Lambda(\overline{X})^{-1}(Y'-X)|=|W'|=r$, so $Y'\in \partial B_\Lambda(X,\overline{X},r)$. Moreover, by construction $Y'$ belongs to the line through $X$ and $X+Y_{||}$, so in particular $Y'\in P$. Therefore
    \begin{equation}
        \label{prelim - round and elliptic flatness, bound 1}
        \dist(X+Y_{||},P\cap B_\Lambda(X,\overline{X},r))\leq |X+Y_{||} - Y'|.
    \end{equation}
    
    Denote $\rho = |X+Y_{||} - Y'|$. Then
    \begin{equation*}
        \begin{split}
            |W-W'|&=|\Lambda(\overline{X})^{-1}(X+Y_{||}-Y')|\geq \lmax(K)^{-1}\rho.
        \end{split}
    \end{equation*}
    Therefore, taking into account that $W-W'=\left(1-\frac{r}{|W|}\right)W$ and 
    $W'=\frac{r}{|W|}W$, so that both $W-W'$ and $W'$ are colinear and point in the same direction, we get
    \begin{equation*}
        \begin{split}
            |W| &= |W-W'|+|W'|\geq r + \lmax(K)^{-1}\rho.
        \end{split}
    \end{equation*}
    In particular, we have $B(W,\lmax(K)^{-1}\rho)\subset\R^{n+1}\backslash B(0,r)$, and applying $X+\Lambda(\overline{X})(\cdot)$ we obtain
    \begin{equation}
        \label{prelim - round and elliptic flatness, helper inclusion 1}X+\Lambda(\overline{X})B(W,\lmax(K)^{-1}\rho)\subset\R^{n+1}\backslash B_\Lambda(X,\overline{X},r).
    \end{equation}
    Now, notice that
    \begin{equation}
        \begin{split}
            \label{prelim - label requested by Tatiana 1}X+\Lambda(\overline{X})B(W,\lmax(K)^{-1}\rho)&= X+\Lambda(\overline{X})W + \Lambda(\overline{X})B(0, \lmax(K)^{-1}\rho)\\
            &\supset X+\Lambda(\overline{X})W + B\left(0,\frac{\lmin(K)}{\lmax(K)}\rho\right)\\
            &= B\left(X+Y_{||},e_\Lambda(K)^{-1}\rho\right).
        \end{split}
    \end{equation}
    Combining \eqref{prelim - round and elliptic flatness, helper inclusion 1} with \eqref{prelim - label requested by Tatiana 1} we get
    \begin{equation}
        \label{prelim - empty intersection}
        B\left(X+Y_{||},e_{\Lambda}(K)^{-1}\rho\right)\cap B_\Lambda(X,\overline{X},r)=\varnothing.
    \end{equation}
    In particular, since $Y\in B_\Lambda(X,\overline{X},r)$, \eqref{prelim - round and elliptic flatness, bound 1} and \eqref{prelim - empty intersection} imply that
    \begin{equation*}
        \begin{split}
            |Y_\perp|= |Y-(X+Y_{||})|&\geq e_\Lambda(K)^{-1}\rho\geq e_\Lambda(K)^{-1}\dist(X+Y_{||},P\cap B_\Lambda(X,\overline{X},r)),
        \end{split}
    \end{equation*}
    which gives
    $$\dist(X+Y_{||},P\cap B_\Lambda(X,\overline{X},r))\leq e_\Lambda(K)|Y_\perp|.$$
    From this estimate and \eqref{prelim - round flatness assumption}, which ensures that $|Y_\perp|\leq\delta r'$, we deduce that
    \begin{equation*}
        \begin{split}
            \dist(Y,P\cap B_\Lambda(X,\overline{X},r))&\leq |Y_\perp| + \dist(X+Y_{||},P\cap B_\Lambda(X,\overline{X},r))\\
            &\leq (1+e_\Lambda(K))|Y_\perp|\leq (1+e_\Lambda(K))\delta r',
        \end{split}
    \end{equation*}
    which proves \eqref{prelim - round and elliptic flatness, first proof conclusion}
.\\
\\
2. Next, we show that 
\begin{equation}
    \label{prelim - round and elliptic flatness, second proof conclusion}
    P\cap B_\Lambda(X,\overline{X},r)\subset (\Sigma\cap B_\Lambda(X,\overline{X},r);(2+e_\Lambda(K))\delta r').
\end{equation}
Let $Y\in P\cap B_\Lambda(X,\overline{X},r)$. We would like to use \eqref{prelim - round flatness assumption} to obtain a point in $\Sigma$ which is close to $Y$. However, if we do this directly at $Y$, the resulting point in $\Sigma$ may not necessarily be contained in $B_\Lambda(X,\overline{X},r)$. We compensate for this by adjusting $Y$ in the following way. Write $Y=X+\Lambda(\overline{X})W$, where $|W|<r$. Let $W'=(1-\rho)W$, where $\rho\in (0,1)$ will be chosen later, and let $Y'=X+\Lambda(\overline{X})W'$.

We will first find a ball with center $Y'$ that is contained in $B_\Lambda(X,\overline{X},r)$. To do this, note that because $|W'|<(1-\rho)r$, we have $B(W',\rho r)\subset B(0,r)$. Therefore
\begin{equation}
\label{prelim - round and elliptic flatness, helper inclusion 2}
    \begin{split}
        X+\Lambda(\overline{X})B(W',\rho r)&\subset X+\Lambda(\overline{X})B(0,r)= B_\Lambda(X,\overline{X},r).
    \end{split}
\end{equation}
Now, note that 
\begin{equation}
\label{prelim - helper inclusion *}
    \begin{split}
        X+\Lambda(\overline{X})B(W',\rho r)&= X+\Lambda(\overline{X})W'+\Lambda(\overline{X})B(0,\rho r)\\
    &= Y'+\Lambda(\overline{X})B(0,\rho r)\\
    &\supset Y'+B(0,\lmin(K)\rho r).
    \end{split}
\end{equation}
Combining this with \eqref{prelim - round and elliptic flatness, helper inclusion 2} gives
\begin{equation}
    \label{prelim - round and elliptic flatness, helper inclusion 3}
    B(Y',\lmin(K)\rho r)\subset B_\Lambda(X,\overline{X},r).
\end{equation}
Next, note that since $Y\in P\cap B_\Lambda(X,\overline{X},r)$, by construction we have $Y'\in P\cap B_\Lambda(X,\overline{X},r)$ as well, so in particular,
$$Y'\in P\cap B(X,r').$$
We can now use \eqref{prelim - round flatness assumption} to deduce that there exists $Q\in \Sigma\cap B(X,r')$ such that
\begin{equation}
\label{prelim - round and elliptic flatness, helper inequality 2}
    |Y'-Q|\leq \delta r'.
\end{equation}
We will use $Q$ to approximate $Y$. We would like to ensure that $Q\in B_\Lambda(X,\overline{X},r)$. To do this, notice that by \eqref{prelim - round and elliptic flatness, helper inclusion 3} it suffices to show that
\begin{equation}
\label{prelim - round and elliptic flatness, helper inequality 3}
    |Y'-Q|<\lmin(K)\rho r.
\end{equation}
But from \eqref{prelim - round and elliptic flatness, helper inequality 2}, we see that this holds as long as $e_\Lambda(K)\delta < \rho$. To ensure that this is the case, we assume that $\delta_K<e_\Lambda(K)^{-1}$ and
\begin{equation}
\label{prelim - delta assumption 1}
    \rho\in (e_\Lambda(K)\delta_K,1). 
\end{equation}
In this scenario \eqref{prelim - round and elliptic flatness, helper inequality 3} holds, which implies that $Q\in B_\Lambda(X,\overline{X},r)$. Moreover, since $Q\in\Sigma$, we have $Q\in \Sigma\cap B_\Lambda(X,\overline{X},r)$. To conclude, we estimate
\begin{equation}
\label{prelim - round and elliptic flatness, helper inequality 4}
    \begin{split}
        |Q-Y|&\leq |Q-Y'|+|Y'-Y|\leq \lmax(K)\delta r + |\Lambda(X)(W'-W)|\leq \delta r'+\rho r'.
    \end{split}
\end{equation}
We now assume, in addition to \eqref{prelim - delta assumption 1}, that
\begin{equation}
    \label{prelim - delta assumption 2}
\rho<(1+e_\Lambda(K))\delta_K.
\end{equation}
 Then \eqref{prelim - round and elliptic flatness, helper inequality 4} implies $|Q-Y|\leq (2+e_\Lambda(K))\delta r'$,
proving \eqref{prelim - round and elliptic flatness, second proof conclusion}. Now \eqref{prelim - elliptic flatness conclusion} follows from \eqref{prelim - round and elliptic flatness, first proof conclusion} and \eqref{prelim - round and elliptic flatness, second proof conclusion}.

Next, we assume \eqref{prelim - elliptic flatness assumption} and prove \eqref{prelim - round flatness conclusion}. We proceed in two steps as before.\\
1. First, we claim that
\begin{equation}
    \label{prelim - helper inclusion 6}
    P\cap B(X,r'')\subset (\Sigma\cap B(X,r'');2\delta r).
\end{equation}
To prove this, let $Y\in P\cap B(X,r'')$. Consider
$$Y' = Y-\frac{\delta}{\lmin(K)}(Y-X).$$
Notice that $Y'\in P$. Moreover, if $\delta_K<\lmin(K)$, then since $Y\in B(X,r'')$,
$$|Y'-X|=\left(1-\frac{\delta}{\lmin(K)}\right)|Y-X|<r'',$$ so $Y'\in B(X,r'')$ as well. Now, since $B(X,r'')\subset B_\Lambda(X,\overline{X},r)$, by \eqref{prelim - elliptic flatness assumption} there exists $Z\in \Sigma\cap B_\Lambda(X,\overline{X},r)$ such that
\begin{equation}
    \label{prelim - helper inequality 6}
    |Y'-Z|\leq\delta r.
\end{equation}
This implies
\begin{equation}
    \begin{split}
        |X-Z|&\leq |X-Y'|+|Y'-Z|\\
        &\leq  \left(1-\frac{\delta}{\lmin(K)}\right)|X-Y|+\delta r< r'' - \delta r+\delta r= r'',
    \end{split}
\end{equation}
so $Z\in\Sigma\cap B(X,r'')$. Moreover, by \eqref{prelim - helper inequality 6} $Z$ satisfies
\begin{equation}
    \begin{split}
        |Y-Z|&\leq |Y-Y'|+|Y'-Z|\leq \frac{\delta}{\lmin(K)}|X-Y|+ \delta r= 2\delta r.
    \end{split}
\end{equation}
This proves \eqref{prelim - helper inclusion 6}.\\
2. Now we show that
\begin{equation}
    \label{prelim - helper inclusion 7}
    \Sigma\cap B(X,r'')\subset (P\cap B(X,r''); \delta r).
\end{equation}
Let $Y\in \Sigma\cap B(X,r'')$. Let $Z$ be the orthogonal projection of $Y$ onto $P$. Then because $P$ contains $X$, we have $Z\in B(X,r'')$, so $Z\in P\cap B(X,r'')$. Moreover, using that $B(X,r'')\subset B_\Lambda(X,\overline{X},r)$ and \eqref{prelim - elliptic flatness assumption}, we get
\begin{equation}
    \begin{split}
        |Y-Z|&=\dist(Y,P\cap B(X,r''))=\dist(Y,P\cap B_\Lambda(X,\overline{X},r))\leq \delta r.
    \end{split}
\end{equation}
This gives \eqref{prelim - helper inclusion 7}. Combining equations \eqref{prelim - helper inclusion 6} and \eqref{prelim - helper inclusion 7} we obtain \eqref{prelim - round flatness conclusion}, which completes the proof of Lemma \ref{prelim - lambda and normal reif flat}.
\end{proof}

\section{Moment estimates}\label{moments}
Here we start deriving geometric information about a measure $\mu$ under the assumption that $\mu$ and $\Lambda$ satisfy the density and continuity conditions of Theorem \ref{theorem 1}, i.e. equations \eqref{introduction - main density estimate} and \eqref{prelim - holder continuity of lambda} (no flatness assumption needs to be made at this point). To accomplish this we consider certain \textit{moments}, an idea that has already been successfully exploited in the literature, most remarkably in the study of uniform measures (see \cite{Pre87}, \cite{KP87}), as well as in the case of measures that are not necessarily uniform but rather asymptotically uniform in a sense, such as the ones considered in \cite{DKT01}. \textit{A priori}, an appropriate notion of moment in our setting would incorporate suitable $\Lambda$ terms. However, we will instead consider a transformation $\tilde\mu$ of $\mu$ for which the standard notion of moment will suffice.

From now on $K\subset\mathbb R^{n+1}$ will be a fixed compact set with $K\cap\Sigma\neq\varnothing$, and $X_0$ will denote an arbitrary point in $K\cap\Sigma$. We will study the regularity of $\Sigma$ near $X_0$ by considering the following transformation. Let

\begin{equation}
    \label{moments - transformation of K and Sigma}
    \tilde{K} = \Lambda(X_0)^{-1}K, \quad \tilde{\Sigma} = \Lambda(X_0)^{-1}(\Sigma)=\mathrm{spt}(\tilde{\mu}),
\end{equation}

\begin{equation}
\label{moments - transformation of mu and lambda}
\tilde{\mu} = \Lambda(X_0)^{-1}[\mu],\quad \tilde{\Lambda}(Y) = \Lambda(X_0)^{-1}\Lambda(\Lambda(X_0)Y),
\end{equation}
where $Y\in\tilde\Sigma\cap \tilde K$ and $\Lambda(X_0)^{-1}[\cdot]$ denotes push-forward via $\Lambda(X_0)^{-1}$. As we will see, the regularity of $\Sigma$ near $X_0$ will be determined by that of $\tilde\Sigma$ near $Y_0=\Lambda(X_0)^{-1}X_0$. The main benefits of performing this transformation come from the fact that
\begin{equation}\label{moments - identity at x0}
    \tilde{\Lambda}(Y_0) = \mathrm{Id}.
\end{equation}
Let us start by using the density assumption on $\mu$ in Theorem \ref{theorem 1} to derive a corresponding estimate for $\tilde\mu$. If $X\in\Sigma\cap K$ and $Y=\Lambda(X_0)^{-1}X\in\tilde\Sigma\cap\tilde K$ (notice that a generic point of $\tilde{\Sigma}\cap\tilde{K}$ can always be written in this way), then
\begin{align*}
    \mu(B_\Lambda(X,r)) &= \mu (X + \Lambda(X)B(0,r))\\
    &= \mu (\Lambda(X_0)[\Lambda(X_0)^{-1}X + \Lambda(X_0)^{-1}\Lambda(X)B(0,r)])\\
    &=\tilde\mu(\Lambda(X_0)^{-1}X + \tilde\Lambda(\Lambda(X_0)^{-1}X))= \tilde{\mu}(B_{\tilde{\Lambda}}(\Lambda(X_0)^{-1}X, r))= \tilde{\mu}(B_{\tilde{\Lambda}}(Y, r)).
\end{align*}
Thus, \eqref{introduction - main density estimate} implies that for every $Y\in\tilde{\Sigma}\cap\tilde{K}$ and $r\in (0,1]$,
\begin{equation}
\label{moments - holder bound on mu tilde}
\left\lvert \frac{\tilde{\mu}(B_{\tilde{\Lambda}}(Y,r))}{\omega_nr^n} - 1 \right\rvert \leq C_Kr^\alpha.
\end{equation}
We will often use this estimate in the form
\begin{equation}
    \label{moments - density assumption expanded out}
    \omega_nr^n-C_Kr^{n+\alpha}\leq\tilde\mu(B_{\tilde\Lambda}(Y,r))\leq  \omega_nr^n+C_Kr^{n+\alpha}.
\end{equation}

\begin{remark}
\label{moments - on the dependence of local constants}
    By our assumptions on $\Lambda$, we have for every $Y,Y'\in\tilde\Sigma\cap\tilde K$,

    $$\|\tilde\Lambda(Y)-\tilde\Lambda(Y')\|\leq e_\Lambda(K)H_K|Y-Y'|^\beta,$$
    with $H_K$ as in \eqref{prelim - holder continuity of lambda} and $e_\Lambda(K)$ as in \eqref{prelim - eccentricity}. This guarantees that as we work with $\tilde{\mu}$ and $\tilde{\Lambda}$ throughout the rest of this section, any local constants that arise from \eqref{moments - holder bound on mu tilde} and the continuity of $\tilde{\Lambda}$ (including the lemmas in Section \ref{prelim}) can be taken to depend on $K$ and $\Lambda$, but not on the particular choice of $X_0$ (or equivalently $\tilde{K}$). This will become important later on.
\end{remark}
We consider the following moments of $\tilde\mu$ at $Y_0$:

\begin{equation}
    \label{moments - definition of b}
    b=\frac{n+2}{2\omega_nr^{n+2}}\int_{B(Y_0,r)}(r^2 - |Z-Y_0|^2)(Z-Y_0)\de\tilde\mu(Z),
\end{equation}
 \begin{equation}
     \label{moments - definition of Q}
     Q(Y) = \frac{n+2}{\omega_nr^{n+2}}\int_{B(Y_0,r)}\langle Y, Z-Y_0\rangle^2\de\tilde\mu(Z),
 \end{equation}
as well as the trace of the quadratic form $Q$,

 $$\mathrm{tr}(Q)=\frac{n+2}{\omega_nr^{n+2}}\int_{B(Y_0,r)}|Z|^2\de\tilde\mu(Z).$$
The fact that these quantities are well suited to $\tilde\mu$ is a consequence of \eqref{moments - identity at x0}. As in \cite{DKT01}, we will use $b$ and $Q$ to show that near $Y_0$, $\tilde\Sigma$ is close to the zero set of a quadratic polynomial. This is the content of the main result of this section.


\begin{proposition}
\label{moments - main proposition}
    Suppose $\Lambda$ and $\mu$ satisfy the continuity and density assumptions of Theorem \ref{theorem 1}.  Let $X_0\in\Sigma\cap K$, where $K\subset\R^{n+1}$ is compact, and let $\tilde{\mu}$, $\tilde{\Lambda}$, $\tilde{\Sigma}$, $\tilde{K}$ be as in \eqref{moments - transformation of mu and lambda}, and $Y_0 = \Lambda(X_0)^{-1}X_0$. Then with $b$ and $Q$ as defined in \eqref{moments - definition of b} and \eqref{moments - definition of Q}, there exist $C_K>0$ and $r_K>0$ depending only on $K$, $\Lambda$ and $n$, such that
    \begin{equation}
    \label{moments - trace estimate}
        |\mathrm{tr}(Q) - n|\leq C_Kr^\alpha,
    \end{equation}
    \begin{equation}
        \label{moments - quadratic estimate}
        | 2\langle b, Y-Y_0\rangle + Q(Y-Y_0) - |Y-Y_0|^2|\leq C_K\left(\frac{|Y-Y_0|^3}{r} + r^{2+\min\{\alpha,\beta\}}\right),
    \end{equation}
    whenever $r\in (0,r_K]$ and $Y\in \tilde{\Sigma}\cap B(Y_0, r/2)$. 
\end{proposition}

\begin{remark}

    It will be useful to keep in mind that even though $\tilde\mu$, $\tilde\Sigma$ and $\tilde K$ depend on $X_0$, the constants $C_K$ and $r_K$ in this result are independent of the particular choice of $X_0\in \Sigma\cap K$.
\end{remark}

\begin{proof}[Proof of Proposition \ref{moments - main proposition}]
    Let us assume without loss of generality that $Y_0 = 0$, and record for later use the fact that $\tilde{\Lambda}(0)= \mathrm{Id}$. We start by proving \eqref{moments - trace estimate}. Here and in what follows we will make repeated use of the following consequence of Fubini's theorem, valid for any measurable set $E$ and any non-negative measurable function $f$:
    $$\int_{E}f(Z)\de\tilde\mu(Z) = \int_0^\infty\tilde\mu(\{Z\in E: f(Z)>t\})\de t.$$
     We see that
    \begin{align*}
        \int_{B(0,r)}|Z|^2\de\tilde\mu(Z) &= \int_{B(0,r)}|Z|^2\de\tilde\mu(Z)\\
        & = \int_0^{r^2}\tilde\mu(\{Z\in B(0,r): |Z|^2>t\})\de t= \int_0^{r^2}\tilde\mu(\{B(0,r)\backslash B(0,\sqrt{t})\})\de t.\\
    \end{align*}
    Now by \eqref{moments - density assumption expanded out} and because $\tilde\Lambda(0)=\mathrm{Id}$, we have for $0<t<r^2$,
    \begin{equation*}
    |\mu(B(0,r)\backslash B(0,\sqrt{t})) - \omega_n(r^n - t^{n/2})|\leq C_K(r^{n+\alpha} + t^{(n+\alpha)/2})\leq C_Kr^{n+\alpha}.
    \end{equation*}
    Therefore, 
    \begin{align*}
        \left\lvert \int_{B(0,r)}|Z|^2\de\tilde\mu(Z) - \int_0^{r^2} \omega_n(r^n - t^{n/2})\de t\right\rvert&\leq \int_0^{r^2}|\tilde\mu(B(0,r)\backslash B(0,\sqrt{t})) - \omega_n(r^n - t^{n/2})|\de t\\
        &\leq C_Kr^{n+\alpha + 2},
    \end{align*}
    which gives
    \begin{equation*}
\left\lvert\frac{n+2}{\omega_nr^{n+2}}\int_{B(0,r)}|Z|^2\de\tilde\mu(Z)-n\right\rvert \leq \left\lvert\frac{n+2}{\omega_nr^{n+2}}\int_0^{r^2}\omega_n(r^n-t^{n/2})\de t - n\right\rvert + C_Kr^{\alpha}\leq C_Kr^\alpha,
    \end{equation*}
    proving \eqref{moments - trace estimate}.

    We now prove \eqref{moments - quadratic estimate}. Assume $0<r<1/2$, and let $Y\in \tilde\Sigma\cap B(0,r/2)$. We consider some ellipses that will help us obtain the necessary estimates. Let 
    \begin{align*}
        &D_1 = B_{\tilde\Lambda}(Y, r - |Y| - C_Kr^{1+\beta}),\ \  D_3 = B_{\tilde\Lambda}(Y,r),\\
        &D_2= B_{\tilde\Lambda}(0,r)=B(0,r),\quad\quad\ \ \ \ \  
        D_4 = B_{\tilde\Lambda}(Y, r + |Y| + C_Kr^{1+\beta}).
    \end{align*}
    If $r$ is small enough depending on $\Lambda$ and $K$, all four radii above are positive, and Lemma \ref{prelim - lemma on nonconcentric ellipses} ensures that
    \begin{equation}
    \label{moments - inclusions of Dj}
        D_1\subset D_2\subset D_4,\quad 
        D_1\subset D_3\subset D_4.
    \end{equation}
    Let, for each $j\in\{ 1,2,3,4\}$, 
    $$J_i = \int_{D_j}(r^2 - |\tilde\Lambda(Y)^{-1}(Z-Y)|^2)^2\de\tilde\mu(Z).$$
    Notice that \eqref{moments - inclusions of Dj} implies
    $$J_1\leq J_2\leq J_4,\quad J_1\leq J_3\leq J_4,$$
    so 
    \begin{equation}
        |J_2-J_3|\leq J_4 - J_1.
    \end{equation}
    We first estimate the right hand side of this inequality. If $Z\in D_4\backslash D_1$, we can write $Z = Y+\tilde\Lambda(Y)W$, where $|\tilde\Lambda(Y)^{-1}(Z-Y)| = |W|$ satisfies 
    \begin{equation*}
        r - |Y| - C_Kr^{1+\beta} <|W|< r + |Y| + C_Kr^{1+\beta}.
    \end{equation*}
    Using this and the fact that $|Y|\leq r/2$ and $r\leq 1/2$,
    \begin{align*}
        |r^2 - |\tilde\Lambda(Y)^{-1}(Z-Y)|^2| &= |r - |W||(r + |W|)\leq (|Y|+ C_Kr^{1+\beta}))(r + r + |Y| + C_Kr^{1+\beta})\\
        &\leq 2|Y|r + |Y|^2 + C_K|Y|r^{1+\beta} + C_Kr^{2+\beta}\leq C_K(r|Y| + r^{2+\beta}).
    \end{align*}
    Therefore,
    \begin{equation}
    \label{moments - J4 minus J1}
    \begin{split}
        J_4 - J_1 &= \int_{D_4\backslash D_1}(r^2 - |\tilde\Lambda(Y)^{-1}(Z-Y)|^2)^2\de\tilde\mu(Z)\leq C_K(r|Y| + r^{2+\beta})^2\tilde\mu(D_4\backslash D_1).
    \end{split}
    \end{equation}
    Now, by \eqref{moments - density assumption expanded out} we have
    \begin{equation}
        \label{moments - step with binomial identity}
        \tilde\mu(D_4\backslash D_1)\leq \omega_n \left[(r + |Y| + C_Kr^{1+\beta})^n-(r - |Y| - C_Kr^{1+\beta})^n   \right]
       + C_{K}(r^{n + \alpha} + r^{n + \alpha + \beta}).
        \end{equation}
        To estimate the  term in brackets we use the fact that if $r>0$ and $\rho\leq Cr$, then 
        \begin{equation}
            \label{moments - helper binomial estimate}
            (r+\rho)^n - (r - \rho)^n\leq Cr^{n-1}\rho.
        \end{equation}
        We use \eqref{moments - helper binomial estimate} with $r$ as in \eqref{moments - step with binomial identity} and $\rho=|Y|+C_Kr^{1+\beta}$. Recall that $|Y|\leq r/2$, so if $r$ is small enough depending on $K$ and $\Lambda$, then $\rho\leq \frac{3}{4}r$. It follows that
        $$(r + |Y| + C_Kr^{1+\beta})^n-(r - |Y| - C_Kr^{1+\beta})^n   \leq Cr^{n_1}(|Y|+C_Kr^{1+\beta}),$$
        which we combine with \eqref{moments - step with binomial identity} to deduce that for $r$ small depending on $K$ and $\Lambda$,
        
        \begin{equation}
            \label{moments - conclusion of binomial step}
            \tilde\mu(D_4\backslash D_1)\leq Cr^{n-1}(|Y| + C_Kr^{1+\beta}) + C_{K}(r^{n+\alpha} + r^{n+\alpha+\beta}).
        \end{equation}
        Thus, by \eqref{moments - J4 minus J1},
        \begin{equation}
            \label{moments - j4 - j1}
            \begin{split}
                J_4 - J_1 &\leq  C_K(r|Y| + r^{2+\beta})^2\left[ r^{n-1}(|Y| + C_Kr^{1+\beta}) +(r^{n+\alpha} + r^{n+\alpha+\beta})\right]\\
            &\leq C_Kr^{n+1}|Y|^3 + C_Kr^{n + 4 + \min\{\alpha,\beta\}}.
            \end{split}
        \end{equation}
        
        We now estimate $J_3$. Write
        \begin{align*}
            J_3 &= \int_{B_{\tilde\Lambda}(Y,r)}(r^2 - |\tilde\Lambda(Y)^{-1}(Z-Y)|^2)^2\de\tilde\mu(Z)\\
            &= \int_0^{r^4}\tilde\mu(\{Z\in B_{\tilde\Lambda}(Y,r): (r^2 - |\tilde\Lambda(Y)^{-1}(Z-Y)|^2)^2 > t\})\de t\\
            &= \int_0\tilde\mu(\{Z\in B_{\tilde\Lambda}(Y,r): |\tilde\Lambda(Y)^{-1}(Z-Y)|<(r^2 - \sqrt{t})^{1/2}\})\de t\\
            &= \int_0^{r^4}\tilde\mu(\{B_{\tilde\Lambda}(Y,(r^2- \sqrt{t})^{1/2})\})\de t.
        \end{align*}
        Let $h(t)=(r^2- \sqrt{t})^{1/2}$. Equation \eqref{moments - density assumption expanded out} implies
        \begin{align*}
            \left\lvert \mu(B_{\tilde\Lambda}(Y,h(t))) - \omega_n h(t)^n\right\rvert &\leq C_Kh(t)^{n+\alpha}\leq C_Kr^{n+\alpha},
        \end{align*}
        so if we let
        $$I(r) = \int_0^{r^4}\omega_nh(t)^n\de t,$$
        then
        
        \begin{equation}
        \label{moments - j3 - i(r)}
            \left\lvert J_3 - I(r)\right\rvert \leq C_K\int_0^{r^4}r^{n+\alpha}\de t\leq C_Kr^{n+4+\alpha}.
        \end{equation}
        Similarly,
        \begin{equation}
            \label{moments - j3 minus thing}\left\lvert\int_{B(0,r)}(r^2 - |Z|^2)^2\de\tilde\mu(Z) - I(r)\right\rvert \leq C_Kr^{n+4+\alpha}.
        \end{equation}
        Combining \eqref{moments - j3 - i(r)} and \eqref{moments - j3 minus thing}, we get
        \begin{equation}
        \label{moments - j3 - center at 0}
            \left\lvert J_3 - \int_{B(0,r)}(r^2 - |Z|^2)^2\de\tilde\mu(Z)\right\rvert\leq C_Kr^{n+4+\alpha}.
        \end{equation}
        Set now
        \begin{equation}
        \label{moments - definition of I}
        \begin{split}
            I &= J_2 - \int_{B(0,r)}(r^2 - |Z|^2)^2\de\tilde\mu(Z)\\
            &= \int_{B(0,r)}(r^2 - |\tilde\Lambda(Y)^{-1}(Z-Y)|^2)^2 - (r^2 - |Z|^2)^2\de\tilde\mu(Z).
        \end{split}
        \end{equation}
        By \eqref{moments - j4 - j1} and \eqref{moments - j3 - center at 0},
        \begin{equation}
        \label{moments - I}
        \begin{split}
            |I| & \leq  |J_2 - J_3| + C_Kr^{n + 4 + \alpha} \\
             & \leq  J_4 - J_1 + C_Kr^{n+4+\alpha}\leq  C_K(r^{n+1}|Y|^3 + r^{n+4+\min\{\alpha,\beta\}}).
        \end{split}
        \end{equation}
        We would now like to replace the term $\tilde\Lambda(Y)$ in the definition of $I$ with $\tilde\Lambda(0)=\mathrm{Id}$. Using that $Z\in B_{\tilde\Lambda}(0,r)$, $|Y|\leq r/2$ and the continuity of $\tilde\Lambda$,
        \begin{align*}
            ||\tilde\Lambda(Y)^{-1}(Z-Y)|^2 - |Z-Y|^2|
            &\leq (|\tilde\Lambda(Y)^{-1}(Z-Y)|+|Z-Y|)||\tilde\Lambda(Y)^{-1}(Z-Y)| - |Z-Y||\\
            &\leq C_Kr|(\tilde\Lambda(Y)^{-1}-\mathrm{Id})(Z-Y)|\leq C_K r^{2+\beta}.
        \end{align*}
        Therefore
        \begin{equation}
        \label{moments - helper estimate X}
        \begin{split}
            |(r^2 - |\tilde\Lambda(Y)^{-1}(Z-Y)|^2)^2 - (r^2 - |Z-Y|^2)^2|&\leq C_Kr^2||\tilde\Lambda(Y)^{-1}(Z-Y)|^2 - |Z-Y|^2|\\
            &\leq C_Kr^{4+\beta}.
            \end{split}
        \end{equation}
        If we now let
        $$I'' = \int_{B(0,r)}(r^2 - |Z-Y|^2)^2 - (r^2 - |Z|^2)^2\de\tilde\mu(Z),$$
        then \eqref{moments - helper estimate X} and \eqref{moments - density assumption expanded out} imply
        \begin{equation}
        \label{moments - I minus I''}
        \begin{split}
            |I-I''|&\leq\int_{B(0,r)}|(r^2-|\tilde\Lambda(Y)^{-1}(Z-Y)|^2)^2-(r^2-|Z-Y|^2)^2|\de\tilde\mu(Z) \\
            &\leq C_K\tilde\mu(B(0,r))r^{4+\beta}\leq C_K(\omega_nr^n + C_Kr^{n+\alpha})r^{4+\beta}\leq C_Kr^{n + 4 +\beta}.
        \end{split}
        \end{equation}
        
        The integral $I''$ will help us transition to the following integral, which as we will see is almost the quadratic polynomial in \eqref{moments - quadratic estimate},
        \begin{equation}
            \label{moments - definition of I'}
            I' = \int_{B(0,r)}\left\lbrace -2|Y|^2(r^2 - |Z|^2) + 4(r^2 - |Z|^2)\langle Z,Y\rangle + 4\langle Y, Z\rangle^2\right\rbrace\de\tilde\mu(Z).
        \end{equation}
        We will show that $I$ and $I'$ are close using $I''$. To this end, note that by \eqref{moments - I minus I''},
        \begin{equation}
        \label{moments - i' - i''}
            |I-I'|\leq |I'-I''| + C_Kr^{n+4+\beta}.
        \end{equation}
        Now we need to estimate $|I' - I''|$. Notice first that
        \begin{align*}
            I'' &= \int_{B(0,r)}(r^2 - |Z-Y|^2)^2 - (r^2 - |Z|^2)^2\de\tilde\mu(Z)\\
            &= \int_{B(0,r)} \left\lbrace r^4 - 2r^2|Z-Y|^2 + |Z-Y|^4 - r^4 + 2r^2|Z|^2 - |Z|^4\right\rbrace\de\tilde\mu(Z)\\
            &= \int_{B(0,r)}\{-2r^2(|Z|^2 - 2\langle Y,Z\rangle + |Y|^2) + (|Z|^2 - 2\langle Y,Z\rangle + |Y|^2)^2 \\
            &\ \quad\quad\quad\quad\quad\quad\quad\quad\quad\quad\quad\quad\quad\quad\quad\quad\quad\quad\quad\quad + 2r^2|Z|^2 - |Z|^4\}\de\tilde\mu(Z)\\
            &= \int_{B(0,r)} \{ 4r^2\langle Y,Z\rangle - 2r^2|Y|^2 + |Z|^4 + 4\langle Y,Z\rangle^2 + |Y|^4\\
            &\ \quad\quad\quad\quad\quad\quad\quad - 4|Z|^2\langle Y,Z\rangle + 2|Y|^2|Z|^2 - 4|Y|^2\langle Y,Z\rangle - |Z|^4\}\de\tilde\mu(Z)\\
            &= \int_{B(0,r)}\{-2(r^2 - |Z|^2)|Y|^2 + 4(r^2 - |Z|^2)\langle Y,Z\rangle + 4\langle Y,Z\rangle^2\\
            &\ \quad\quad\quad\quad\quad\quad\quad\quad\quad\quad\quad\quad\quad\quad\quad\quad\quad\quad\ + |Y|^4 - 4|Y|^2\langle Y,Z\rangle\}\de\tilde\mu(Z).
        \end{align*}
        Therefore, recalling the definition of $I'$ and using that $|Y|\leq r/2$ and \eqref{moments - density assumption expanded out},
        \begin{align*}
            |I' - I''| &\leq \int_{B(0,r)}||Y|^4 - 4|Y|^2\langle Y,Z\rangle|\de\tilde\mu(Z)\leq \tilde\mu(B(0,r))\left(\frac{|Y|^3r}{2} + 4|Y|^3r\right)\\ 
            &\leq C( r^n + C_K r^{n+\alpha})|Y|^3r\leq Cr^{n+1}|Y|^3 + C_Kr^{n+4+\alpha}.
        \end{align*}
        Combining this with \eqref{moments - i' - i''} we get
        \begin{equation}
            \begin{split}
                |I-I'|&\leq \displaystyle\frac{9}{2}\omega_nr^{n+1}|Y|^3 + C_Kr^{n+4+\alpha} + C_Kr^{n+4+\beta}\\
            &\leq  Cr^{n+1}|Y|^3 + C_Kr^{n+4+\min\{\alpha,\beta\}}.
            \end{split}
        \end{equation}
        To conclude, we obtain the desired quadratic polynomial from $I'$. Observe first that
        \begin{align*}
              \int_{B(0,r)}(r^2 - |Z|^2)\de\tilde\mu(Z) &= \int_0^{r^2}\tilde\mu(\left\lbrace Z\in B(0,r) : r^2 - |Z|^2 > t\right\rbrace)\de t\\
              &= \int_0^{r^2}\tilde\mu\left(B(0,\sqrt{r^2 - t})\right)\de t\\
              &= \int_0^{r^2}\left(\omega_n(r^2 - t)^{n/2} + C_K(r^2 - t)^{\frac{n+\alpha}{2}}\right)\de t\\
              &= \frac{2\omega_nr^{n+2}}{n+2} + O(r^{n+2+\alpha}),\\
        \end{align*}
        where $|O(r^{n+2+\alpha})|/r^{n+2+\alpha}\leq C_K$. From here it follows, by multiplying by $(n+2)|Y|^2/(2\omega_nr^{n+2})$, that
        \begin{equation}
            \label{momnents - approximation of |y|^2}
            \left\lvert |Y|^2 - \frac{n+2}{2\omega_n r^{n+2}}|Y|^2\int_{B(0,r)}(r^2 - |Z|^2)\de\tilde\mu(Z)\right\rvert \leq C_K|Y|^2r^\alpha.
        \end{equation}
        Combining \eqref{momnents - approximation of |y|^2} with \eqref{moments - definition of b}, \eqref{moments - definition of Q} and \eqref{moments - definition of I'}, we get
        \begin{align*}
            \left\lvert \frac{(n+2)}{4\omega_nr^{n+2}}I' - \left\lbrace -|Y|^2 + 2\langle b,Y\rangle + Q(Y)\right\rbrace \right\rvert &= \left\lvert |Y|^2 - \frac{n+2}{2\omega_nr^{n+2}}|Y|^2\int_{B(0,r)}(r^2 - |Z|^2)\de\tilde\mu(Z)\right\rvert\\
            &\leq C_K|Y|^2r^{\alpha}.
        \end{align*}
        Finally, combining this estimate with \eqref{moments - I} and \eqref{moments - i' - i''}, and keeping in mind that $|Y|\leq r/2$, we get
        \begin{align*}
            \left\lvert 2\langle b, Y\rangle + Q(Y) - |Y|^2\right\rvert &\leq \frac{n+2}{4\omega_nr^{n+2}}|I'| + C_K|Y|^2r^\alpha\\
            &\leq \frac{n+2}{4\omega_nr^{n+2}}(|I| + Cr^{n+1}|Y|^3 + C_Kr^{n+4+\min\{\alpha,\beta\}}) + C_Kr^{2+\alpha}\\
            &\leq \frac{C_K}{r^{n+2}}(r^{n+1}|Y|^3 + r^{n+4+\min\{\alpha,\beta\}})+ C_Kr^{2+\alpha}\\
            &\leq C_K\left(\frac{|Y|^3}{r} + r^{2+\min\{\alpha,\beta\}}\right).
        \end{align*}
        This shows that \eqref{moments - quadratic estimate} holds and completes the proof of Proposition \ref{moments - main proposition}.
\end{proof}


\section{Decay of $\beta$-numbers}\label{beta numbers}
In this section we continue to assume $\mu$ and $\Lambda$ satisfy the density and continuity assumptions of Theorem \ref{theorem 1}. The main goal here is to obtain an estimate on the decay of the quantity
\begin{equation}\beta_\Sigma(X,r) = \inf_P\left\lbrace \sup_{Y\in \Sigma\cap B(X,r)}\frac{\mathrm{dist}(Y,P)}{r}\right\rbrace,\label{beta numbers - definition of centered beta infty}
\end{equation}
where $X\in\Sigma=\mathrm{spt}(\mu)$, $r>0$, and the infimum is taken over all $n$-planes $P\subset\R^{n+1}$ such that $X\in P$. This quantity is a centered version of P. Jones' $\beta_\infty$ numbers introduced in \cite{Jon90}, as the planes in \eqref{beta numbers - definition of centered beta infty} all go through $X$. The numbers $\beta_\Sigma$ can be considered a unilateral version of $b\beta_\Sigma$, in that they capture if $\Sigma$ is locally close to a plane, but not the converse.

Consider a compact set $K\subset\mathbb R^{n+1}$ such that $\Sigma\cap K\neq\varnothing$. We will show that under suitable conditions, $\beta_\Sigma(\cdot,r)$ decays at a certain rate as $r\to 0$, uniformly on $\Sigma\cap K$. To prove this, we resort to Proposition \ref{moments - main proposition} and show that as in \cite{DKT01}, moment estimates can be used to control $\beta_\Sigma(\cdot,r)$, provided that $\Sigma$ is flat enough.

Before we state the main result of this section, let us recall the quantities $\lmin(K)$, $\lmax(K)$ and $e_\Lambda(K)$ defined in \eqref{prelim - lmin and lmax of K} and \eqref{prelim - eccentricity}, as well as the transformation introduced in \eqref{moments - transformation of K and Sigma} and \eqref{moments - transformation of mu and lambda}. Let us notice the following fact, which is a consequence of the continuity of $\Lambda$: for each compact set $K\subset\R^{n+1}$, there exists a number $d_K>0$ depending only on $K$ and $\Lambda$ such that for every $X_0\in\Sigma\cap K$,
     \begin{equation}
         \Lambda(X_0)((\tilde \Sigma\cap \tilde K;d_K))\subset (\Sigma\cap K,1) \label{beta numbers - neighborhood of K},
     \end{equation}
     where $\tilde\Sigma$ and $\tilde K$ are as in \eqref{moments - transformation of K and Sigma}. In fact, we may take $d_K=\lmax(K)^{-1}$. The main result of this section is the following.

\begin{proposition}
\label{beta numbers - main proposition}
    If $n\geq 3$, suppose that $\Sigma$ is Reifenberg flat with vanishing constant. Suppose that $\mu$ and $\Lambda$ satisfy the density and continuity assumptions of Theorem \ref{theorem 1}. Then, for every compact set $K\subset\R^{n+1}$ there exist $C_K>0$ and $r_K>0$, both depending only on $K$, $\Lambda$ and $n$, such that for all $X_0\in\Sigma\cap K$ and $r\in (0,r_K]$,
    \begin{equation}
        \label{beta numbers - main estinate}
        \beta_{\Sigma}(X_0,r)\leq C_Kr^\gamma,
    \end{equation}
    where $\gamma\in(0,1)$ depends on $\alpha$ and $\beta$.
\end{proposition}

\begin{remark}
    Note that the assumption that $\Sigma$ is Reifenberg flat with vanishing constant when $n\geq 3$ is \textit{a priori} stronger than the flatness assumption of Theorem \ref{theorem 1}. However, as we will see in Section \ref{flatness}, the assumptions of Theorem \ref{theorem 1} imply that $\Sigma$ is in fact Reifenberg flat with vanishing constant. This will make Proposition \ref{beta numbers - main proposition} applicable in the proof of Theorem \ref{theorem 1}. 
\end{remark}

\begin{proof}[Proof of Proposition \ref{beta numbers - main proposition}] Let $K$ and $X_0$
 be as in the statement. We consider the transformation $\tilde\mu$ of $\mu$ introduced in \eqref{moments - transformation of K and Sigma} and \eqref{moments - transformation of mu and lambda}, as well as $Y_0=\Lambda(X_0)^{-1}X_0$. It will be important to keep in mind that $\tilde\mu$ depends on $X_0$. The proof has two main steps.
 \begin{enumerate}
     \item[Step 1:] Bounding $\beta_{\tilde\Sigma}(Y_0,r)$. This will rely on Proposition \ref{moments - main proposition} and arguments in connection with  \cite[Proposition 8.6]{DKT01}, which deals with the Euclidean case, and whose statement we include below.

     \item[Step 2:]Bounding $\beta_{\Sigma}(X_0,r)$. This will be a consequence of our estimate on $\beta_{\tilde\Sigma}(Y_0,r)$ from Step 1 and particular features of the transformation $\mu\mapsto\tilde\mu$.
 \end{enumerate}  
    
    \begin{proposition}[\cite{DKT01}, Proposition 8.6]
        \label{beta numbers - dkt proposition}
        Let $\tilde{\mu}$ be a Radon measure in $\R^{n+1}$ with support $\tilde\Sigma$. Assume that for each compact set $\tilde K\subset\R^{n+1}$ there is a constant $C_{\tilde K}>0$ such that
        \begin{equation}
            \label{beta numbers - density condition in dkt}
            \left|\frac{\tilde\mu(B(Y,r))}{\omega_nr^n}-1\right|\leq C_{\tilde K}r^\alpha,
        \end{equation}
        for all $Y\in \tilde K\cap\tilde\Sigma$ and $0<r<1$. If $n\geq 3$, assume that $\tilde\Sigma$ is Reifenberg flat with vanishing constant. Then for each compact set $\tilde K\subset \R^{n+1}$ there exists $r_{\tilde{K}}>0$ depending on $n$, $\alpha$ and $\tilde K$, so that for all $Y\in \tilde K$ and $0<r\leq r_{\tilde{K}}$,
        \begin{equation}
            \label{beta numbers - beta estimate of sigma tilde at any Y}
            \beta_{\tilde\Sigma}(Y,r)\leq C_{\tilde K}r^\gamma,
        \end{equation}
        where $\gamma\in(0,1)$ depends on $\alpha$ and $\beta$.
    \end{proposition}
    \begin{remark}
    \label{beta numbers - actual flatness assumption in DKT}
    It is worth mentioning, although not necessary for our arguments, that the proof of this result remains valid if $\tilde\Sigma$ is only assumed to be Reifenberg flat with constant $\delta_n$, where $\delta_n>0$ is small enough depending only on $n$.
    \end{remark}

    It should be noted that the transformation $\tilde\mu$ of $\mu$ from \eqref{moments - transformation of K and Sigma} and \eqref{moments - transformation of mu and lambda} does not satisfy the assumptions on the measure $\tilde\mu$ in the above proposition. However, we will draw a parallel between them and show that both measures still satisfy similar conclusions. Let us briefly recall the main elements in the transformation  $\mu\mapsto\tilde\mu$:
    \begin{equation}
        \label{beta numbers - recalling transformation}
        \tilde{\mu} = \Lambda(X_0)^{-1}[\mu],\quad \tilde{\Lambda}(X) = \Lambda(X_0)^{-1}\Lambda(\Lambda(X_0)X),\quad \tilde{K} = \Lambda(X_0)^{-1}(K), \quad \tilde{\Sigma} = \Lambda(X_0)^{-1}(\Sigma).
    \end{equation}

    \subsection{Step 1: Bound for $\beta_{\tilde\Sigma}(Y_0,r)$} The first observation we need to make is that, as mentioned above, we cannot directly apply Proposition \ref{beta numbers - dkt proposition} to the transformation $\tilde\mu$ of $\mu$ given by \eqref{beta numbers - recalling transformation}, the reason being that such $\tilde\mu$ only satisfies \eqref{beta numbers - density condition in dkt} at $Y_0$, whereas at other points $Y\neq Y_0$, $B(Y,r)$ needs to be replaced with $B_{\tilde\Lambda}(Y,r)$. Therefore, instead of applying Proposition \ref{beta numbers - dkt proposition}, we will argue that its proof can still be adapted in our setting to obtain a somewhat weaker conclusion:  
\begin{equation}
    \label{beta numbers - beta of sigma'}
    \begin{split}
    &\text{There exist $C_{\tilde K}>0$ and $r_{\tilde K}>0$, both depending on $\tilde K$, and there  exists $\gamma\in (0,1)$}\\
    &\text{ depending only on $\min\{\alpha,\beta\}$, such that for every $r\in (0,r_{\tilde K}]$, }\beta_{\tilde\Sigma}(Y_0,r)\leq C_{\tilde K}r^\gamma.
    \end{split}
\end{equation}
   
    Note that this condition is only different from the conclusion of Proposition \ref{beta numbers - dkt proposition} in that the $\beta$-number estimate in \eqref{beta numbers - beta of sigma'} only holds at $Y_0$, as opposed to an arbitrary point of $\tilde\Sigma\cap\tilde K$. Therefore, what we need to discuss is the extent to which the arguments in \cite{DKT01} carry over when proving not the full conclusion of Proposition \ref{beta numbers - dkt proposition} in our setting, but rather its validity at $Y_0$. By an inspection of \cite{DKT01}, we see that those arguments rely only on two components:
    \begin{enumerate}[(i)]
        \item A density estimate and two moment estimates for $\tilde\mu$ at $Y_0$; and
        \item $\tilde\Sigma$ being Reifenberg with vanishing or small constant.
    \end{enumerate}
    We will show that both components are still available in our setting, only with minor differences that do not interfere with the proof of Proposition \ref{beta numbers - dkt proposition}, from which the validity of \eqref{beta numbers - beta of sigma'} will follow.\\
    
    (i) \textit{Density and moment estimates. }These are inequalities whose corresponding analogues have been established in the previous section. We first recollect them for the sake of convenience. By \eqref{moments - holder bound on mu tilde} and because $\tilde\Lambda(Y_0)=\mathrm{Id}$, we have for all $r\in (0,1]$,
    \begin{equation}\left|\frac{\tilde\mu(B(Y_0,r))}{\omega_nr^n}-1\right|\leq C_{K}r^\alpha.\label{beta numbers - key estimate 1}\end{equation}
    Also, by Proposition \ref{moments - main proposition} we know that with $ b$ and $Q$ as defined in \eqref{moments - definition of b} and \eqref{moments - definition of Q}, we have
    \begin{equation}|\mathrm{tr}(Q) - n|\leq C_{K}r^\alpha,\label{beta numbers - key estimate 2}\end{equation}
    and
    \begin{equation}\label{beta numbers - quadratic estimate}
        ||Y-Y_0|^2 - 2\langle b, Y-Y_0\rangle - Q(Y-Y_0)|\leq C_{ K}\left(\frac{|Y-Y_0|^3}{r} + r^{2+\min\{\alpha,\beta\}}\right),\end{equation}
        for all $Y\in\tilde\Sigma\cap B(Y_0,r/2)$ and $r\in (0,r_K]$.
    These estimates are very similar to the ones required in the argument of \cite{DKT01} for the proof of Proposition \ref{beta numbers - dkt proposition}. There are only a few differences, but we can see why none of them interfere with their argument.

    \begin{itemize}
        \item The first difference is that the exponent on the last term in \eqref{beta numbers - quadratic estimate} is $\min\{\alpha,\beta\}$, as opposed to $\alpha$ as in \cite{DKT01}. This is not a problem, since we can adjust all three estimates above by replacing $\alpha$ with $\min\{\alpha,\beta\}$.

        \item The second one is that, as mentioned above, \eqref{beta numbers - key estimate 1} implies that \eqref{beta numbers - density condition in dkt} holds at $Y=Y_0$, but not necessarily at other points $Y$. However, an inspection of the arguments in \cite{DKT01} shows that the validity of \eqref{beta numbers - density condition in dkt} at points $Y\neq Y_0$ is only needed in order to ensure that two moment estimates analogous to \eqref{beta numbers - key estimate 2} and \eqref{beta numbers - quadratic estimate} hold. In our case, the validity of both moment estimates has already been established in Section \ref{moments}.

        \item The third difference is that while \eqref{beta numbers - key estimate 2} and \eqref{beta numbers - quadratic estimate} hold with constants $C_K$ and $r_K$ that do not depend on $X_0$, the analogous moment estimates for $\tilde\mu$ needed in \cite{DKT01} hold with constants that depend on $\tilde K$, and therefore also on $X_0$ (see \eqref{beta numbers - recalling transformation}). This is also not a problem, since it only means that the constants in \eqref{beta numbers - key estimate 2} and \eqref{beta numbers - quadratic estimate} enjoy extra uniformity.

        \item The fourth one is that the analogues of \eqref{beta numbers - key estimate 2} and \eqref{beta numbers - quadratic estimate} in \cite{DKT01} hold with $r\in (0,1/2]$, as opposed to $r\in (0,r_K]$ as in our setting. But this is also not a problem since the threshold radius in \eqref{beta numbers - beta of sigma'} can be adjusted accordingly.

        \item The last one is that the constant $C_{K}$ in \eqref{beta numbers - quadratic estimate} multiplies the entire right hand side, as opposed to just the last term as in \cite{DKT01}. However, an inspection of their argument shows that this does not interfere either, since the only difference is that some of the absolute constants that arise in their setting will now depend on $K$.

    \end{itemize}
        \vspace{2mm}
    (ii) \textit{Flatness of $\tilde\Sigma$.} The statement of Proposition \ref{beta numbers - dkt proposition} assumes that $\tilde\Sigma$ is Reifenberg flat with vanishing or small constant. However, all that the proof of Proposition \ref{beta numbers - dkt proposition} in \cite{DKT01} requires is that this condition holds near $\tilde K$, in the following sense: 
    \begin{center}
        There exist $d_{\tilde K}>0$ and $t_{\tilde K}>0$ such that for all $r\in (0,t_{\tilde K}]$ and $Y\in\tilde\Sigma\cap ({\tilde K};d_{\tilde K})$,
    \begin{equation}
        \label{beta numbers - flatness of sigma tilde}b\beta_{\tilde\Sigma}(Y,r)\leq\delta_n,
    \end{equation}
    \end{center}
    where $\delta_n>0$ is small enough, depending only on $n$. We will show that this holds in our setting as a consequence of $\Sigma$ being Reifenberg flat with vanishing constant.
    To see this, let $\varepsilon>0$ and take $d_{\tilde K}=d_K$, with $d_K$ as in \eqref{beta numbers - neighborhood of K}, let $Y\in\tilde\Sigma\cap(\tilde K;d_K)$, and write $Y=\Lambda(X_0)^{-1}X$ for some $X\in\Sigma\cap (K,1)$. Since $\Sigma$ is Reifenberg flat with vanishing constant, there exists $r_K>0$ such that if $0<r\leq r_K$, then
    \begin{equation}
        \label{beta numbers - flatness assumption}b\beta_\Sigma(X,r)\leq\varepsilon.
    \end{equation}
    Let $t_{ K}=r_K\lmax(K)^{-1}$ and suppose $0<r\leq t_K$. Assume also that $\varepsilon$ is small enough so that the assumptions of Lemma \ref{prelim - lambda and normal reif flat} are satisfied. Let $P$ be an $n$-plane through $X$ that attains the infimum in the definition of $b\beta_\Sigma(X,\lmax(K)r)$, and denote $\tilde P=\Lambda(X_0)^{-1}P$. Then, by Lemma \ref{prelim - lambda and normal reif flat} and \eqref{beta numbers - flatness assumption},

\begin{align*}
    D\left[ \tilde \Sigma\cap B(Y,r); \tilde P\cap B(Y,r)\right] &\leq \lmin(X_0)^{-1}D\left[\Sigma\cap \Lambda(X_0)B(Y,r); \Lambda(X_0)(\tilde P \cap B(Y,r))\right]\\
    &\leq \lmin(K)^{-1}D\left[\Sigma\cap \left\lbrace X+\Lambda(X_0)B(0,r)\right\rbrace; P\cap \left\lbrace X+\Lambda(X_0)B(0,r)\right\rbrace \right]\\
    &\leq \lmin(K)^{-1}(2+e_\Lambda(K))\lmax(K)\varepsilon r\leq C_K\varepsilon r.
\end{align*}
    Since $\varepsilon>0$ is arbitrary, this shows that 
    \eqref{beta numbers - flatness of sigma tilde} holds with $t_{\tilde K}=t_K$, and in fact $\tilde\Sigma$ is Reifenberg flat with vanishing constant too.
    \begin{remark}
    \label{beta numbers - dependence of rK on K}
        It will be important to notice that the value of $t_{\tilde K}$ found above depends on $K$, but not on the particular choice of $X_0$, so it enjoys extra uniformity.
    \end{remark}
    \vspace{2mm}
    This completes our justification that the proof of Proposition \ref{beta numbers - dkt proposition} is applicable to $\tilde\mu$ at $Y_0$, and as a consequence, \eqref{beta numbers - beta of sigma'} holds, concluding step 1.

\subsection{Step 2: Bound for $\beta_\Sigma(X_0,r)$} We now use \eqref{beta numbers - beta of sigma'} and translate it into a estimate for $\beta_\Sigma(X_0,\cdot)$. The main aspect we need to deal with is the fact that the constants $C_{\tilde K}$ and $r_{\tilde K}$ in \eqref{beta numbers - beta of sigma'} depend \textit{a priori} on $\tilde K$, and therefore on the choice of $X_0\in \Sigma\cap K$. However, as some of the above arguments suggest, these constants should in fact depend on $K$, but not on the particular choice of $X_0$. We will justify that this is the case, and then use this information to estimate $\beta_\Sigma$ as follows:
\begin{enumerate}[(i)]
    \item $C_{\tilde K}$ can be taken to be independent of $X_0$;
    \item $r_{\tilde K}$ can be taken to be independent of $X_0$;
    \item $\beta_\Sigma$ satisfies \eqref{beta numbers - main estinate}.
\end{enumerate}

(i) \textit{$C_{\tilde K}$ is independent of $X_0\in \Sigma\cap K$. }
An examination of the proof of Proposition \ref{beta numbers - dkt proposition} shows that the constant $C_{\tilde K}$ in \eqref{beta numbers - beta of sigma'} comes from its occurrence in the density and moment estimates discussed in Step 1 (i) above, and subsequent multiplication by various absolute constants. However, as noted before, the constants in these density and moment estimates can be taken to depend on $K$ only. Therefore, the same applies to $C_{\tilde K}$ in \eqref{beta numbers - beta of sigma'}, and we may write $C_{\tilde K}=C_K$.\\

(ii) \textit{$r_{\tilde K}$ is independent of $X_0\in\Sigma\cap K$. }First note that the way the threshold $r_{\tilde K}$ of equation \eqref{beta numbers - beta estimate of sigma tilde at any Y} is chosen in \cite{DKT01} ($r_0$ in their notation), is as $r_{\tilde K} = \frac{1}{4}t_{\tilde K}^{1+\tau}$, $\tau\in (0,1)$, where $t_{\tilde K}$ is a threshold radius for which \eqref{beta numbers - flatness of sigma tilde} holds. But as noted in Remark \ref{beta numbers - dependence of rK on K}, such threshold can be taken to be independent of $X_0$, so the same is true about $\tilde r_K$. Thus we may write $r_{\tilde K}=r_K$.\\

(iii) \textit{Decay of $\beta_\Sigma(X_0,r)$. }To estimate $\beta_\Sigma(X_0,r)$, first notice that by points (i) and (ii), the estimate in \eqref{beta numbers - beta of sigma'} becomes
\begin{equation}
    \label{beta numbers - beta of sigma tilde updated}
    \beta_{\tilde\Sigma}(Y_0,r)\leq C_Kr^\gamma,
\end{equation}
for all $r\in (0,r_K]$, where $C_K>0$ and $r_K>0$ depend only on $K$, $\Lambda$ and $n$. Let $r\in (0,r_K]$ and let $\tilde P$ be an $n$-plane through $Y_0$ attaining the infimum in the definition of $\beta_{\tilde\Sigma}(Y_0,r)$. As before, we can write  $\tilde P=\Lambda(X_0)^{-1}P$, where $P$ is an $n$-plane through $X_0$. We will estimate $\beta_\Sigma(X_0,\lmin(K)r)$. Notice that
\begin{equation}
    \begin{split}
        B(X_0,\lmin(K)r)&\subset B(X_0,\lmin(X_0)r)\subset B_\Lambda(X_0,r)= \Lambda(X_0)B(Y_0,r).
    \end{split}
\end{equation}
Thus given any $W\in \Sigma\cap B(X_0,\lmin(K)r)$, we can write $W = \Lambda(X_0)Z$, with $Z\in \tilde\Sigma \cap B(Y_0,r)$. Then by \eqref{beta numbers - beta of sigma tilde updated},
\begin{align*}
    \mathrm{dist}(W,P) = \mathrm{dist}(\Lambda(X_0)Z, \Lambda(X_0)\tilde P)&\leq \lmax(K)\mathrm{dist}(Z,\tilde P)\leq C_Kr^{1+\gamma}.
\end{align*}
This implies that $ \beta_{\Sigma}(X_0,\lmin(K)r)\leq C_K r^\gamma$,
for all $r\in (0,r_K]$, or equivalently, 
\begin{equation}
    \label{beta numbers - final conclusion}
    \begin{split}
        \beta_\Sigma(X_0,r)&\leq C_K(\lmin(K)^{-1}r)^\gamma\leq C_Kr^\gamma,
    \end{split}
\end{equation}
for all $r\in (0,\lmin(K)r_K]$. This shows that \eqref{beta numbers - main estinate} holds and completes the proof of Proposition \ref{beta numbers - main proposition}.
\end{proof}


\section{$\Lambda$-pseudo tangents of $\mu$ and proof of Theorem \ref{theorem 1}}\label{pseudo tangents}
The proof of Theorem \ref{theorem 1} will be complete if we can combine Proposition \ref{beta numbers - main proposition} and the following result.

\begin{proposition}[\cite{DKT01} - Proposition 9.1]
\label{reif flatness - dkt theorem}
Let $\gamma\in (0,1]$. Suppose $\Sigma$ is a Reifenberg-flat set with vanishing constant of dimension $m$ in $\R^{n+1}$, $m\leq n+1$, and that for each compact set $K\subset\R^{n+1}$ there is a constant $C_K>0$ such that
\begin{equation}
    \label{pseudo tangents - decay of beta numbers}
    \beta_\Sigma(X,r)\leq C_Kr^\gamma,
\end{equation}
for all $X\in K\cap\Sigma$ and $r\in (0,r_K]$. Then $\Sigma$ is a $C^{1,\gamma}$ submanifold of dimension $m$ of $\R^{n+1}$.
\end{proposition}

   As mentioned before, the assumption that $\Sigma$ is Reifenberg flat with vanishing constant is stronger than the flatness assumption in Theorem \ref{theorem 1}. However, the following result ensures that the latter suffices in our setting.

\begin{proposition}
    \label{pseudo tangents - main flatness result}Suppose $\mu$ and $\Lambda$ satisfy the density and continuity assumptions of Theorem \ref{theorem 1}, and let $\Sigma=\mathrm{spt}(\mu)$. If $n\geq 3$, suppose also that for any compact set $K\subset\R^{n+1}$ there exists $r_K>0$ such that 
    \begin{equation}
    \label{pseudo tangents - flatness assumption}
        b\beta_\Sigma(K,r_K)=\sup_{r\in (0,r_K]}\sup_{X\in\Sigma\cap K}b\beta_\Sigma(X,r)\leq\delta_K,
    \end{equation}
    where $\delta_K>0$ is small enough depending on $K$ and $\Lambda$. Then $\Sigma$ is Reifenberg flat with vanishing constant.
\end{proposition}
We first show why this is enough in order to complete the proof of Theorem \ref{theorem 1}.
\begin{proof}[Proof of Theorem \ref{theorem 1}]
    Let $\mu$ and $\Lambda$ be as in the assumptions of the theorem. By Proposition \ref{pseudo tangents - main flatness result}, $\Sigma$ is Reifenberg flat with vanishing constant. Therefore, Proposition \ref{beta numbers - main proposition} ensures that \eqref{pseudo tangents - decay of beta numbers} holds, and the conclusion of Theorem \ref{theorem 1} follows from Proposition \ref{reif flatness - dkt theorem}.

\end{proof}

To prove Proposition \ref{pseudo tangents - main flatness result} we follow an approach based on that of \cite{KT99} in the Euclidean setting, with two main steps:
\begin{enumerate}
    \item[Step 1.] Show that all $\Lambda$-pseudo tangents to $\mu$ are uniform (see definitions below); and
    \item[Step 2.] Prove, via a result of Kowalski and Preiss \cite{KP87}, that \eqref{pseudo tangents - flatness assumption} implies that those $\Lambda$-pseudo tangents are flat, and use this to conclude.
\end{enumerate}

This section is devoted to the first step, which happens to be independent of the smallness of $\delta_K$. We first consider some relevant definitions and facts that will be needed later. 

\subsection{$\Lambda$-pseudo tangent measures}
Given a point $X\in\Sigma$ and a radius $r>0$, consider the mapping 

\begin{equation}
    \label{pseudo tangents - blow up map}
    T^\Lambda_{X,r}(X) = \Lambda(X)^{-1}\left(\frac{Z-X}{r}\right),\quad Z\in\R^{n+1},
\end{equation}
and the measure 
\begin{equation}
    \label{pseudo tangents - def of scaled measure}
    \mu_{P,r} =\frac{1}{\mu(B_\Lambda(P,r))} T^\Lambda_{P,r}[\mu].
\end{equation}
Here $T^\Lambda_{P,r}[\cdot]$ denotes push-forward via $T^\Lambda_{P,r}$, so for $E\subset\R^{n+1}$, 
$$\mu_{P,r}(E) = \frac{\mu(P + r\Lambda(P)E)}{\mu(B_\Lambda(P,r))}.$$

\begin{definition}[$\Lambda$-pseudo tangent measure]
\label{pseudo tangents - definition of pseudo tangents}
    A measure $\nu\not\equiv 0$ is a $\Lambda$-\textit{pseudo tangent measure} of $\mu$ at $Q\in\Sigma$ if there exists a sequence of points $Q_i\in\Sigma$ and radii $\rho_i>0$ with $Q_i\to Q$ and $\rho_i\to 0$ as $i\to\infty$, such that 
    $$\mu_{Q_i,\rho_i}\rightharpoonup \nu.$$
\end{definition}
Here, the symbol $\rightharpoonup$ denotes weak convergence of Radon measures. Note that when the points $Q_i$ in Definition \ref{pseudo tangents - definition of pseudo tangents} satisfy $Q_i=Q$ for all $i$, the resulting measure $\nu$ is a $\Lambda$-tangent measure of $\mu$ (see \cite{CGTW25}). If $\Lambda(Q_i)=\mathrm{Id}$, then $\nu$ is a (pseudo) tangent measure of $\mu$ (see \cite{KT99}). The following are well-known facts about tangent measures in the Euclidean setting (see \cite{Ma95}).
\begin{lemma}[Existence of $\Lambda$-pseudo tangent measures]\label{pseudo tangents - existence of pseudo tangents}
    Let $\mu$ be a Radon measure with support $\Sigma\subset\R^{n+1}$, such that for each compact set $K\subset\R^{n+1}$ with $\Sigma\cap K\neq\varnothing$,
    $$\sup_{\substack{0<r\leq 1\\ X\in\Sigma\cap K}}\frac{\mu(B_\Lambda(X,2r))}{\mu(B_\Lambda(X,r))}<\infty.$$
    Then every sequence of numbers $r_i>0$ with $r_i\searrow 0$ and points $Q_i\in\Sigma$ contains a subsequence $r_{i_l}$, $Q_{i_l}$ such that the measures $\mu_{{Q_{i_l}},r_{i_l}}$ converge to a $\Lambda$-pseudo tangent measure of $\mu$ at $X$.
\end{lemma}
\begin{proof}
    Let $K$ be a compact set with $Q_i\in K$ for all $i$, and denote by $c$ the supremum in the statement of the lemma. Then for every $k\in\mathbb N$ we have
    \begin{equation}
        \begin{split}
            \limsup_{i\to\infty}\mu_{{Q_{i}},r_{i}}(B(0,2^k))&=\limsup_{i\to\infty}\frac{1}{\mu(B_\Lambda(Q_i,r_i))}T^\Lambda_{Q_i,r_i}[\mu](B(0,2^k))\\
            &=\limsup_{i\to\infty}\frac{\mu(B_\Lambda(Q_i,2^kr_i))}{\mu(B_\Lambda(Q_i,r_i))}\leq c^k<\infty.
        \end{split}
    \end{equation}
    It follows that the sequence $\mu_{Q_i,r_i}(F)$ is bounded for every compact set $F\subset\R^{n+1}$, and the conclusion of the lemma follows by a standard compactness result for Radon measures (see \cite[Theorem 1.23]{Ma95}).
\end{proof}

\begin{lemma}
    \label{pseudo tangents - 0 is in the support}
    If $\mu$ satisfies the assumptions of Theorem \ref{theorem 1} and $\nu$ is a $\Lambda$-pseudo tangent of $\mu$, then $0\in\mathrm{spt}(\nu)$ .
\end{lemma}

       \begin{proof} Recall that under the assumptions of Theorem \ref{theorem 1}, for every $X\in\Sigma\cap K$ and $r\in (0,1]$ we have
    \begin{equation}
        \label{pseudo tangents - density condition expanded out}
        \omega_nr^n-C_Kr^{n+\alpha}\leq\mu(B_\Lambda(X,r))\leq \omega_nr^n+C_Kr^{n+\alpha}.
    \end{equation}
    Thus, if $R>0$, $X\in\Sigma\cap K$ and $r>0$ is small enough, 
    \begin{equation}
        \label{pseudo tangents - bound 1}
        \begin{split}
            \mu_{X,r}(B(0,R))&=\frac{\mu(B_\Lambda(X,rR))}{\mu(B_\Lambda(X,r))}\geq \frac{(rR)^n - C_K(rR)^{n+\alpha}}{r^n+C_Kr^{n+\alpha}}\geq \frac{R^n}{2}.
        \end{split}
    \end{equation}
    Now, since $\nu$ is a $\Lambda$-pseudo tangent measure of $\mu$, we have $\mu_{P_i,\rho_i}\rightharpoonup \nu$ for some $P_i\in\Sigma\cap K$, where $K\subset\mathbb R^{n+1}$ is a compact set, $\rho_i>0$ and $\rho_i\to 0$. Therefore, applying \eqref{pseudo tangents - bound 1} with $X=P_i$ and $r=\rho_i$, we get 
    \begin{equation*}
        \begin{split}
            \nu(B(0,2R))&\geq \nu(\overline{B(0,R)})\geq\limsup_{i\to\infty}\mu_{P,\rho_i}(\overline{B(0,R)})\\
            &\geq \limsup_{i\to\infty}\mu_{P,\rho_i}(B(0,R))\geq\frac{R^n}{2}>0,
        \end{split}
    \end{equation*}
    from which the desired conclusion follows.
\end{proof}

The key point about $\Lambda$-pseudo tangents in our context is that if a measure $\mu$ satisfies the density assumption of Theorem \ref{theorem 1}, then all its $\Lambda$-pseudo tangent measures are $n$-uniform, as shown below under a more relaxed assumption on $\mu$ (see Definition \ref{pseudo tangents - lambda asod} and Proposition \ref{pseudo tangents - pseudo tangents are uniform}).

\begin{definition}\label{pseudo tangents - lambda asod}
    A Radon measure $\mu$ in $\R^{n+1}$ with support $\Sigma$ is called $\Lambda$-\textit{asymptotically optimally doubling of dimension $n$} if for every compact set $K\subset\R^{n+1}$ ,
    \begin{equation}
        \label{pseudo tangents - lambda asod definition}
        \lim_{r\to 0}\sup_{\substack{X\in \Sigma\cap K\\ \tau\in [\frac{1}{2},1]}}\left|\frac{\mu(B_\Lambda(X,\tau r))}{\mu(B_\Lambda(X,r))}-\tau^n\right|=0.
    \end{equation}
\end{definition}
The corresponding Euclidean version of this notion is considered in \cite{DKT01}, Definition 1.4. We summarize a couple of facts about this condition and its connection with measures that satisfy the density condition \eqref{introduction - main density estimate} in Theorem \ref{theorem 1}.

\begin{proposition}
\label{pseudo tangents - density implies asod}
    Let $\mu$ be a Radon measure with support $\Sigma$ in $\mathbb R^{n+1}$.
    \begin{enumerate}
        \item If $\mu$ satisfies \eqref{introduction - main density estimate}, then it also satisfies \eqref{pseudo tangents - lambda asod definition}.
        \item If $\mu$ satisfies \eqref{pseudo tangents - lambda asod definition}, then for every $K\subset\mathbb R^{n+1}$ compact,
    \begin{equation}
        \label{pseudo tangents - asymptotically optimally doubling}
        \lim_{r\to 0}\sup_{t\in (0,1)}\sup_{X\in \Sigma\cap K}\left\lvert \frac{\mu(B_\Lambda(X,t r))}{\mu(B_\Lambda(X,r))} - t^n\right\rvert = 0.
    \end{equation}
    \end{enumerate}
\end{proposition}
\begin{proof}
    For the proof of the first statement, note that by \eqref{introduction - main density estimate}, if $r>0$ is small enough then
    $$\left|\frac{\mu(B_\Lambda(X,\tau r))}{\omega_n(\tau r)^n} - 1\right|\leq C_K(\tau r)^\alpha  \quad\text{and}\quad \left| \frac{\omega_nr^n}{\mu(B_\Lambda(X,r))} - 1\right\rvert\leq C_Kr^\alpha.$$
    Therefore,
    \begin{equation}
\label{pseudo tangents - quantitative asod}
    \begin{split}
        \left\lvert \frac{\mu(B_\Lambda(X,\tau r))}{\mu(B_\Lambda(X,r))} - \tau^n\right\rvert
    &\leq \tau^n\left\lbrace \left\lvert \frac{\omega_nr^n}{\mu(B_\Lambda(X,r))}\left(\frac{\mu(B_\Lambda(X,\tau r))}{\omega_n(\tau r)^n} - 1\right)\right\rvert  +  \left\lvert \frac{\omega_nr^n}{\mu(B_\Lambda(X,r))} - 1\right\rvert\right\rbrace\\
    &\ \\
    &\leq \tau^n\left\lbrace C_K(\tau r)^\alpha + C_K r^\alpha\right\rbrace\leq C_K\tau^n r^\alpha\leq C_Kr^\alpha.
    \end{split}
\end{equation}
This gives \eqref{pseudo tangents - lambda asod definition}. For the second statement, \eqref{pseudo tangents - lambda asod definition} implies that given $\varepsilon>0$, there exists $R>0$ so that for $r\in (0,R)$, $X\in\Sigma\cap K$ and $\tau\in [\frac{1}{2},1]$,
\begin{equation}
    \label{pseudo tangents - auxiliary doubling}|\mu(B_\Lambda(X,\tau r))-\tau^n\mu(B_\Lambda(X,r))|\leq\varepsilon\mu(B_\Lambda(X,r)).
\end{equation}

Let $t\in (0,1]$, and let $j\geq 1$ be such that $\frac{1}{2^j}\leq t<\frac{1}{2^{j-1}}$, so that $\tau:=t^{1/j}\in [\frac{1}{2},\frac{1}{\sqrt{2}})$. Then by \eqref{pseudo tangents - auxiliary doubling}, we have for $X\in\Sigma\cap K$, $r\in (0,R)$ and $k\geq 1$,
$$|\mu(B_\Lambda(X,\tau^{k} r))-\tau^n\mu(B_\Lambda(X,\tau^{k-1}r))|\leq\varepsilon\mu(B_\Lambda(X,r)).$$
Therefore,
\begin{equation*}
    \begin{split}
        |\mu(B_\Lambda(X,t r))-t^n\mu(B_\Lambda(X,r))|&\leq\sum_{k=0}^{j-1}\tau^{nk}|\mu(B_\Lambda(X,\tau^{j-k}r))-\tau^n\mu(B_\Lambda(X,\tau^{j-k-1}r))|\\
        &\leq \varepsilon\mu(B_\Lambda(X,r))\sum_{k=0}^{j-1}\tau^{nk}\\
        &\leq \varepsilon\mu(B_\Lambda(X,r))\sum_{k=0}^{\infty}\frac{1}{(\sqrt{2})^{nk}}= C\varepsilon\mu(B_\Lambda(X,r)),
    \end{split}
\end{equation*}
where $C>0$ depends only on $n$. This implies
$$\left|\frac{\mu(B_\Lambda(X,tr))}{\mu(B_\Lambda(X,r))}-t^n\right|\leq C\varepsilon,$$
for all $r\in (0,R)$, from which the desired conclusion follows.
\end{proof}

The following is the main result of this section.

\begin{proposition}
\label{pseudo tangents - pseudo tangents are uniform}
   Suppose that $\Lambda$ satisfies the continuity assumption of Theorem \ref{theorem 1} and $\mu$ is $\Lambda$-asymptotically optimally doubling of dimension $n$ in $\mathbb R^{n+1}$. If $\nu$ is a $\Lambda$-pseudo tangent measure of $\mu$, then $\nu$ is $n$-uniform. Moreover, \eqref{pseudo tangents - uniform measure} holds with $C=1$.
\end{proposition}
\begin{remark}
    This result remains valid in any codimension.
\end{remark}

To prove this we start with a description of the support of any given $\Lambda$-pseudo tangent measure of $\mu$.
\begin{lemma}
\label{pseudo tangents - characterization of support}
    Suppose $\mu$ is a $\Lambda$-asymptotically optimally doubling measure of dimension $n$ in $\R^{n+1}$ with support $\Sigma$. Let $\rho_i>0$ and $Q_i\in\Sigma$ be such that $\rho_i\to 0$, $Q_i\to Q\in\Sigma$ and $\mu_{Q_i,\rho_i}\rightharpoonup \nu$ as $i\to\infty$, where $\nu$ is a $\Lambda$-pseudo tangent measure of $\mu$. If $T^\Lambda_{Q_i,\rho_i}$ is defined as in \eqref{pseudo tangents - blow up map} and $X\in\R^{n+1}$, then $X\in\mathrm{spt}(\nu)$ if and only if there exist $X_i\in T^\Lambda_{Q_i,r_i}(\Sigma)$ such that $X_i\to X$ as $i\to\infty$.
\end{lemma}

\begin{proof}[Proof of Lemma \ref{pseudo tangents - characterization of support}]
    For the forward direction, let $X\in\mathrm{spt}(\nu)$. Suppose, for contradiction, that there exist $\varepsilon_0>0$ and $i_k\in\mathbb{N}$ with $i_k\to\infty$, and for every $i_k$
    \begin{equation}
        \label{pseudo tangents - big distance}
        \mathrm{dist}(X,T^\Lambda_{Q_{i_k},r_{i_k}}(\Sigma))\geq\varepsilon_0.
    \end{equation}
    If $\varphi\in C_c(B(X,\varepsilon_0/2))$ and $\chi_{B(X,\varepsilon_0/4)}\leq\varphi\leq\chi_{B(X,\varepsilon_0/2)}$, by \eqref{pseudo tangents - big distance} we have $\varphi\left(T^\Lambda_{Q_{i_k},r_{i_k}}(Y)\right)=0$ for every $Y\in\Sigma$. Therefore,
    \begin{equation*}
\nu(B(X,\varepsilon_0/4))\leq \int\varphi\de\nu = \lim_{k\to\infty}\frac{1}{\mu(B_\Lambda(Q_{i_k},\rho_{i_k}))}\int_\Sigma\varphi\left(T^\Lambda_{Q_{i_k},r_{i_k}}(Y)\right)\de\mu(Y)=0,
    \end{equation*}
    which contradicts the assumption that $X\in\mathrm{spt}(\nu)$.

    To prove the converse, let $X_i\in T^\Lambda_{Q_i,\rho_i}$ be such that $X_i\to X$ as $i\to\infty$, and write 
    $$X_i = \Lambda(Q_i)^{-1}\left(\frac{Z_i - Q_i}{\rho_i}\right),\quad Z_i\in\Sigma.$$
    Given $r>0$,
    \begin{equation}
    \label{pseudo tangents - initial estimate}
        \begin{split}
\mu_{Q_i,\rho_i}(B(X,r)) &= \frac{\mu(\rho_i\Lambda(Q_i)B(X,r) + Q_i)}{\mu(B_\Lambda(Q_i,\rho_i))}\\
&= \frac{\mu(\rho_i\Lambda(Q_i)(B(0,r) + X) + Q_i)}{\mu(B_\Lambda(Q_i,\rho_i))}\\
        &= \frac{\mu(\Lambda(Q_i)B(0,r\rho_i) + Q_i + \rho_i\Lambda(Q_i)X)}{\mu(B_\Lambda(Q_i,\rho_i))}= \frac{\mu(B_\Lambda(Q_i+\rho_i\Lambda(Q_i)X,r\rho_i))}{\mu(B_\Lambda(Q_i,\rho_i))}.
        \end{split}
    \end{equation}
    To get a lower bound, we need to shift the center $Q_i + \rho_i\Lambda(Q_i)X$ in the numerator to a point in $\Sigma$ so that we can use the doubling assumption. Notice that 
    \begin{equation}
    \label{pseudo tangents - notice 1}
        \begin{split}
            |(Q_i + \rho_i\Lambda(Q_i)X) - Z_i| &= |(Q_i + \rho_i\Lambda(Q_i)X) - (Q_i + \rho_i\Lambda(Q_i)X_i)|\\
        &= \rho_i|\Lambda(Q_i)(X-X_i)|,
        \end{split}
    \end{equation}
    so by Lemma \ref{prelim - lemma on nonconcentric ellipses},
    $$B_\Lambda(Q_i+\rho_i\Lambda(Q_i)X,r\rho_i) \supset B_\Lambda(Z_i, r\rho_i - \lmin(Q_i+\rho_i\Lambda(Q_i)X)^{-1}\rho_i|\Lambda(Q_i)(X-X_i)| - C_K(r\rho_i)^{1+\beta}).$$
    Assuming $i$ is large enough depending on $r$, $K$ and $\Lambda$, we have
    $$\lmin(Q_i+\rho_i\Lambda(Q_i)X)^{-1}|\Lambda(Q_i)(X-X_i)|\leq \frac{r}{4},\quad C_K(r\rho_i)^{1+\beta}\leq\frac{r}{4}.$$
    It follows from the last inclusion above that for all $i$ large enough,
    $$B_\Lambda(Q_i + \rho_i\Lambda(Q_i)X,r\rho_i)\supset B_\Lambda(Z_i,r\rho_i/4).$$
    From this and \eqref{pseudo tangents - initial estimate} we get
    \begin{equation}
        \label{pseudo tangents - spt lower bound}
        \mu_{Q_i\rho_i}(B(X,r))\geq \frac{\mu(B_\Lambda(Z_i,r\rho_i/4))}{\mu(B_\Lambda(Q_i,\rho_i))}.
    \end{equation}
    
    Next, we proceed similarly as above to change the center once more, so that both centers coincide. Note that 
    $$|Q_i - Z_i|= \rho_i|\Lambda(Q_i)X_i|\leq C_K\rho_i,$$
    where $C_{K}>0$ is a constant depending on $X$, $K$ and $\Lambda$. Thus by an application of Lemma \ref{prelim - lemma on nonconcentric ellipses}, equation \eqref{prelim - nested ellipses, larger radius}, we get
    \begin{equation*}
        B_\Lambda(Q_i,\rho_i)\subset B_\Lambda(Z_i, \rho_i + \lmin(Q_i)^{-1}\rho_i|\Lambda(Q_i)X_i| + C_K\rho_i^{1+\beta})\subset B_\Lambda(Z_i,C_K\rho_i).
    \end{equation*}
    Combining this with \eqref{pseudo tangents - spt lower bound} and using the doubling assumption on $\mu$, if $i$ is large enough depending on $r$ and $C_K$,
    \begin{equation*}
        \mu_{Q_i,\rho_i}(B(X,r))\geq \frac{\mu(B_\Lambda(Z_i,r\rho_i/4))}{\mu(B_\Lambda(Z_i,C_K\rho_i))}\geq \frac{1}{2}\left(\frac{r}{4C_K}\right)^n.
    \end{equation*}
    Therefore, since $\mu_{Q_i,\rho_i}\rightharpoonup\nu$, 
    \begin{align*}
        \nu(B(X,2r))\geq \nu(\overline{B(X,r)})\geq \limsup_{i\to\infty}\mu_{Q_i,\rho_i}(\overline{B(X,r)})\geq\frac{1}{2}\left(\frac{r}{4C_K}\right)^n.
    \end{align*}
    This implies that $\nu(B(X,r))>0$ for every $r>0$, which in turn shows that $X\in\mathrm{spt}(\nu)$ as desired.
\end{proof}

\subsection{Proof of Proposition \ref{pseudo tangents - pseudo tangents are uniform}}

\begin{proof}Let $\nu$ be a $\Lambda$-pseudo tangent of $\mu$, and let $Q_i\in\Sigma$ and $\rho_i>0$ be such that $Q_i\to Q$, $\rho_i\to 0$ and  $\mu_{Q_i,\rho_i}\rightharpoonup \nu$ as $i\to\infty$. By Lemma \ref{pseudo tangents - characterization of support}, there exist $X_i\in T^\Lambda_{Q_i,\rho_i}(\Sigma)$ such that $X_i\to X$ as $i\to\infty$. Write
$$X_i = \Lambda(Q_i)^{-1}\left(\frac{Z_i - Q_i}{\rho_i}\right),\quad Z_i\in\Sigma.$$
Let $r>0$. We need to get lower and upper bounds for 
$$\mu_{Q_i,\rho_i}(B(X,r)) = \frac{\mu(B_\Lambda(Q_i + \rho_i\Lambda(Q_i)X), r\rho_i)}{\mu(B_\Lambda(Q_i,\rho_i))}.$$

We start with an upper bound. Let $\varepsilon>0$. As in the proof of Lemma \ref{pseudo tangents - characterization of support}, by \eqref{pseudo tangents - notice 1},
\begin{align*}
    |(Q_i+\rho_i\Lambda(Q_i)X) - Z_i| &= \rho_i|\Lambda(Q_i)(X-X_i)|.
\end{align*}
So an application of Lemma \ref{prelim - lemma on nonconcentric ellipses}, equation \eqref{prelim - nested ellipses, larger radius}, gives
\begin{equation*}
B_\Lambda(Q_i+\rho_i\Lambda(Q_i)X,r\rho_i)\subset B_\Lambda(Z_i, r\rho_i + \lmin(Q_i+\rho_i\Lambda(Q_i)X)^{-1}\rho_i|\Lambda(Q_i)(X-X_i)| + C_K(r\rho_i)^{1+\beta}).
\end{equation*}
If $i$ is large enough depending on $r$, $X$, $K$ and $\Lambda$, we can guarantee that
$$\lmin(Q_i+\rho_i\Lambda(Q_i)X)^{-1}\rho_i|\Lambda(Q_i)(X-X_i)|\leq \varepsilon r\rho_i
,\quad C_K(r\rho_i)^{1+\beta}\leq \varepsilon r\rho_i.$$
It then follows from the inclusion above that for such $i$,
\begin{equation}
\label{pseudo tangents - uniform, first inclusion}
B_\Lambda(Q_i + \rho_i \Lambda(Q_i)X, r\rho_i)\subset B_\Lambda(Z_i, r\rho_i(1+ 2\varepsilon)),
\end{equation}
and consequently,
\begin{equation}
\label{pseudo tangents - uniform, first upper bound}
    \mu_{Q_i,\rho_i}(B(X,r)) \leq \frac{\mu(B_\Lambda(Z_i,r\rho_i(1 + 2\varepsilon)))}{\mu(B_\Lambda(Q_i,\rho_i))}.
\end{equation}
Write
\begin{equation}
    \label{pseudo tangents - uniform, two factors}
    \frac{\mu(B_\Lambda(Z_i, r\rho_i(1+ 2\varepsilon)))}{\mu(B_\Lambda(Q_i,\rho_i))} = 
\frac{\mu(B_\Lambda(Z_i,r\rho_i(1+\varepsilon)))}{\mu(B_\Lambda(Q_i,r\rho_i(1+2\varepsilon)))}\cdot\frac{\mu(B_\Lambda(Q_i,r\rho_i(1+2\varepsilon)))}{\mu(B_\Lambda(Q_i,\rho_i))}.
\end{equation}
Assume without loss of generality that $Q_i\in\Sigma\cap B(Q,1)$. If $i$ is large enough depending on $r$, $\varepsilon$ and $\Lambda$, then by the doubling assumption on $\mu$, the second factor above satisfies
\begin{equation}
    \label{pseudo tangents - uniform, second factor*}
\frac{\mu(B_\Lambda(Q_i,r\rho_i(1+2\varepsilon)))}{\mu(B_\Lambda(Q_i,\rho_i))}\leq (1+\varepsilon)[r(1+2\varepsilon)]^n.
\end{equation}
To deal with the first factor, we would like to move the center $Q_i$ to $Z_i$. However, doing so directly would introduce an error comparable to $\rho_i$, which is a larger order of magnitude than what we can allow if $r$ is small. The following estimate avoids this obstacle. Let $\kappa>0$ be a large constant to be determined. Then for $i$ large depending on $\kappa$  and $\varepsilon$,

\begin{equation}
\label{pseudo tangents - uniform, three fractions}
\begin{array}{rcl}
    \displaystyle\frac{\mu(B_\Lambda(Z_i,r\rho_i(1+\varepsilon)))}{\mu(B_\Lambda(Q_i,r\rho_i(1+2\varepsilon)))}  &=&\displaystyle\frac{\mu(B_\Lambda(Z_i,r\rho_i(1+2\varepsilon)))}{\mu(B_\Lambda(Z_i,\kappa r\rho_i(1+2\varepsilon)))}\cdot\frac{\mu(B_\Lambda(Z_i,\kappa r\rho_i(1+2\varepsilon)))}{\mu(B_\Lambda(Q_i,\kappa r\rho_i(1+2\varepsilon)))}\\
    &&\\
    & &\quad\quad\quad\quad\quad\quad\quad\quad\quad\quad\quad\cdot\displaystyle\frac{\mu(B_\Lambda(Q_i,\kappa r\rho_i(1+2\varepsilon)))}{\mu(B_\Lambda(Q_i, r\rho_i(1+2\varepsilon)))}\\
    &&\\
    &\leq& \displaystyle(1+\varepsilon)^2\cdot\frac{\mu(B_\Lambda(Z_i,\kappa r\rho_i(1+2\varepsilon)))}{\mu(B_\Lambda(Q_i,\kappa r\rho_i(1+2\varepsilon)))}.
\end{array}
\end{equation}

We can now make the centers coincide. Recall that
\begin{equation*}
    |Q_i - Z_i|=\rho_i|\Lambda(Q_i)X_i|\leq C_K\rho_i,
\end{equation*}
where $C_K>0$ is a constant that depends on $X$, $K$ and $\Lambda$. Therefore, by Lemma \ref{prelim - lemma on nonconcentric ellipses},
\begin{equation}
    \label{pseudo tangents - helper inclusion 1}
    \begin{split}
    B_\Lambda(Z_i,\kappa r\rho_i(1+2\varepsilon))&\subset B_\Lambda(Q_i, \kappa r\rho_i(1+2\varepsilon)) +\lmin(Z_i)^{-1}\rho_i|\Lambda(Q_i)X_i| + C_K(\kappa r\rho_i(1+2\varepsilon))^{1+\beta})\\
    &\subset B_\Lambda(Q_i,\kappa r\rho_i(1+2\varepsilon) + C_K\rho_i + C_K(\kappa r\rho_i(1+2\varepsilon))^{1+\beta})\\
    &\subset B_\Lambda\left(Q_i, \kappa r\rho_i\left[1 + 2\varepsilon + \frac{C_K}{\kappa r} + C_K(\kappa r\rho_i)^\beta(1+2\varepsilon)^{1+\beta}\right]\right).
\end{split}
\end{equation}

We now take $\kappa$ to be sufficiently large, depending on $X$, $K$, $\Lambda$, $r$ and $\varepsilon$, so that $\frac{C_K}{\kappa r}<\varepsilon$. In addition, we assume that $i$ is sufficiently large, depending on $X$, $K$, $\Lambda$, $r$ and $\varepsilon$, so that
$$C_K(\kappa r\rho_i)^\beta(1+2\varepsilon)^{1+\beta}\leq \varepsilon.$$
In this scenario, \eqref{pseudo tangents - helper inclusion 1} implies
$$B_\Lambda(Z_i,\kappa r\rho_i(1+2\varepsilon))\subset B_\Lambda(Q_i, \kappa r\rho_i(1+4\varepsilon)).$$
It follows from this inclusion and the doubling assumption on $\mu$,
\begin{align*}
    \frac{\mu(B_\Lambda(Z_i,\kappa r\rho_i(1+2\varepsilon)))}{\mu(B_\Lambda(Q_i,\kappa r\rho_i(1+2\varepsilon)))} &\leq \frac{\mu(B_\Lambda(Q_i,\kappa r\rho_i(1+4\varepsilon)))}{\mu(B_\Lambda(Q_i,\kappa r\rho_i(1+2\varepsilon)))}\\
    &\leq \frac{\mu(B_\Lambda(Q_i,\kappa r\rho_i(1+4\varepsilon)))}{\mu(B_\Lambda(Q_i,\kappa r\rho_i))}\leq (1+\varepsilon)(1+4\varepsilon)^n.
\end{align*}
Combining this with \eqref{pseudo tangents - uniform, three fractions} we get
\begin{equation}
    \label{pseudo tangents - reference for singular set section}\frac{\mu(B_\Lambda(Z_i,r\rho_i(1+\varepsilon)))}{\mu(B_\Lambda(Q_i,r\rho_i(1+2\varepsilon)))}\leq (1+\varepsilon)^3(1+4\varepsilon)^n.
\end{equation}

Putting this together with \eqref{pseudo tangents - uniform, second factor*} and \eqref{pseudo tangents - uniform, first upper bound}, we obtain
\begin{equation*}
    \mu_{Q_i,\rho_i}(B(X,r))\leq (1+\varepsilon)^3(1+4\varepsilon)^n(1+\varepsilon)[r(1+2\varepsilon)]^n\leq r^n(1+4\varepsilon)^{2n+4},
\end{equation*}
for all $i$ large depending on $X$, $K$, $\Lambda$, $r$ and $\varepsilon$. This shows that
\begin{equation}
    \label{pseudo tangents - uniform, final upper bound}
    \limsup_{i\to\infty}\mu_{Q_i,\rho_i}(B(X,r))\leq r^n.
\end{equation}
An analog argument gives 
\begin{equation}
    \label{pseudo tangents - uniform, final lower bound}
    \liminf_{i\to\infty}\mu_{Q_i,\rho_i}(B(X,r))\geq r^n.
\end{equation}

Combining \eqref{pseudo tangents - uniform, final upper bound} and \eqref{pseudo tangents - uniform, final lower bound} we can show that $\nu$ satisfies the desired conclusion. In fact, using that $\mu_{Q_i,\rho_i}\rightharpoonup \nu$ we get
\begin{equation}
\label{pseudo tangents - upper bound final}
\nu(B(X,r))\leq\liminf_{i\to\infty}\mu_{Q_i,\rho_i}(B(X,r))
    \leq r^n,
\end{equation}
and given any $\varepsilon\in (0,1)$,
\begin{equation}
\begin{split}
\label{pseudo tangents - lower bound final}
    \nu(B(X,r))&\geq \nu(\overline{B(X,(1-\varepsilon)r)})\geq\limsup_{i\to\infty}\mu_{Q_i,\rho_i}(\overline{B(X,(1-\varepsilon)r)})\\
    &\geq \limsup_{i\to\infty}\mu_{Q_i,\rho_i}(B(X,(1-\varepsilon)r))\geq [(1-\varepsilon)r]^n.
\end{split}
\end{equation}
Since this holds for every $\varepsilon>0$, we conclude from \eqref{pseudo tangents - upper bound final} and \eqref{pseudo tangents - lower bound final} that
$$\nu(B(X,r))=r^n,$$
completing the proof of Proposition \ref{pseudo tangents - pseudo tangents are uniform}.
\end{proof}

\section{Flatness of a measure with uniform $\Lambda$-pseudo tangents}\label{flatness}

In this section we complete Step 2 of the proof of Proposition \ref{pseudo tangents - main flatness result}. We do this by proving the more general statement that if all $\Lambda$-pseudo tangent measures of $\mu$ are $n$-uniform, and if $\Sigma=\mathrm{spt}(\mu)$ satisfies flatness condition \eqref{pseudo tangents - flatness assumption} when $n\geq 3$, then $\Sigma$ is Reifenberg flat with vanishing constant. This does not require density estimate \eqref{introduction - main density estimate} to be satisfied. However, when proving Theorem \ref{theorem 1}, the fact that all $\Lambda$-pseudo tangent measures of $\mu$ are $n$-uniform will be a consequence of \eqref{introduction - main density estimate}, as discussed in the previous section.

Except for $\delta_K$, all other local constants that arise will eventually be denoted by $C_K$ as before. It may be convenient to recall the quantities associated with $\Lambda$ and any compact set $K\subset\mathbb R^{n+1}$, $\lmin(K)$, $\lmax(K)$ and $e_\Lambda(K)$,  introduced in \eqref{prelim - lmin and lmax of K} and \eqref{prelim - eccentricity}. We will also consider the quantity

\begin{equation}
    \label{flat measures - Mk definition}
    \quad M_K=(2+e_\Lambda(K))\lmax(K).
\end{equation}

\begin{proposition}
\label{flat measures - main result}
    Let $\mu$ be a Radon measure on $\R^{n+1}$ such that all its $\Lambda$-pseudo tangent measures are $n$-uniform, where $\Lambda$ satisfies the continuity assumption of Theorem \ref{theorem 1}, and let $K\subset\R^{n+1}$ be compact. If $n\geq 3$, suppose also that there exists $r_K>0$ such that
    \begin{equation}
        \label{flat measures - fair flatness}
        b\beta_\Sigma(K,r_K)=\sup_{r\in (0,r_K]}\sup_{X\in\Sigma\cap K}\theta_\Sigma(X,r)\leq\delta_K,
    \end{equation}
    where is $\delta_K>0$ is small enough depending on $K$ and $\Lambda$. Then $$\lim_{r\searrow 0}b\beta_\Sigma(K,r)=0.$$
    In particular, if $n\leq 2$, or $n\geq 3$ and \eqref{flat measures - fair flatness} holds for every compact $K\subset \R^{n+1}$, then  $\Sigma$ is Reifenberg-flat with vanishing constant. 
\end{proposition}
Assuming this result momentarily, the proof of Proposition \ref{pseudo tangents - main flatness result} is short.
\begin{proof}[Proof of Proposition \ref{pseudo tangents - main flatness result}]
    By Proposition \ref{pseudo tangents - density implies asod}, $\mu$ is $\Lambda$-asymptotically optimally doubling of dimension $n$, so by Proposition \ref{pseudo tangents - pseudo tangents are uniform} all its $\Lambda$-pseudo tangent measures are $n$-uniform. Proposition \ref{flat measures - main result} then implies that $\Sigma$ is Reifenberg-flat with vanishing constant.
\end{proof}
At the core of the proof of Proposition \ref{flat measures - main result} is Theorem \ref{flat measures - kowalski preiss}. In our case, the measure $\nu$ in that theorem will be a suitable $\Lambda$-pseudo tangent measure of $\mu$ that captures how flat $\mu$ is. The key point is that the light cone in \eqref{flat measures - light cone} is not $\delta$-Reifenberg flat if for example $\delta<1/\sqrt{2}$. This implies that if $\nu$ inherits \eqref{flat measures - fair flatness}, then by Theorem \ref{flat measures - kowalski preiss}, $\nu$ must be flat. Such information can then be used to show that $\Sigma$ is Reifenberg flat with vanishing constant. This approach follows ideas developed by Kenig and Toro in \cite{KT99} in the Euclidean setting.\\ 

\begin{remark}
\label{flat measures - compactness of hausdorff distance}
    Before proceeding with the proof, we record for later use the following compactness property of Hausdorff distance: if $\Gamma_i\subset\R^{n+1}$ contains the origin for all $i\in\mathbb{N}$, then there exists a subsequence $i_k$ and a set $\Gamma\subset\R^{n+1}$ such that
$$\Gamma_{i_k}\to \Gamma,$$
with respect to Hausdorff distance, uniformly on compact subsets of $\R^{n+1}$.
\end{remark}

\begin{proof}[Proof of Proposition \ref{flat measures - main result}]
    Let $K\subset\R^{n+1}$ be compact. Consider
    $$\ell = \lim_{\tau\searrow 0}b\beta_\Sigma(K,\tau),$$
    where 
    $b\beta_\Sigma(K,\tau)$ is as in \eqref{flat measures - fair flatness}. We will show that $\ell=0$. Let $\tau_i>0$ be such that $\tau_i\searrow 0$ and $b\beta_\Sigma(K,\tau_i)\to\ell$. Let $Q_i\in\Sigma\cap K$ be points for which
    \begin{equation}
        \label{flat measures - tau i approximates ell}
        b\beta_\Sigma(Q_i,\tau_i)\to \ell.
    \end{equation}
     Since $\Sigma\cap K$ is compact, we may assume without loss of generality that $Q_i\to Q\in\Sigma\cap K$. We will need to work with the auxiliary scales
    \begin{equation}
    \label{flat measures - auxiliary scales}
    \rho_i=\lmin(K)^{-1}\tau_i.
    \end{equation}
    
    Recall the map
    $$T^\Lambda_{Q_i,\rho_i}(X) = \Lambda(Q_i)^{-1}\left(\frac{X-Q_i}{\rho_i}\right),\quad X\in \R^{n+1}.$$
    Notice first that $0\in T^\Lambda_{Q_i,\rho_i}(\Sigma)$ for all $i$. Thus, by Remark \ref{flat measures - compactness of hausdorff distance} we may assume modulo passing to a subsequence that there exists $\Sigma_\infty\subset\R^{n+1}$ such that
    \begin{equation}
    \label{flat measures - sigma infinity}
        T^\Lambda_{Q_i,\rho_i}(\Sigma)\to\Sigma_\infty,
    \end{equation}
    with respect to Hausdorff distance $D$, uniformly on compact sets. We may also assume upon taking a further subsequence that $\mu_{Q_i,\rho_i}\rightharpoonup \nu$, where $\mu_{Q_i,\rho_i}$ is as in \eqref{pseudo tangents - def of scaled measure} and $\nu$ is a $\Lambda$-pseudo tangent measure of $\mu$. Moreover, we know by Proposition \ref{pseudo tangents - pseudo tangents are uniform} that $\nu$ is $n$-uniform, and we may assume without loss of generality, upon multiplying $\nu$ by a suitable constant, that \eqref{pseudo tangents - uniform measure} is satisfied with with $C=\omega_n$, so that the assumptions of Theorem \ref{flat measures - kowalski preiss} are satisfied. Note that by Lemma \ref{pseudo tangents - characterization of support} and \eqref{flat measures - sigma infinity}, we have
    $$\mathrm{spt}(\nu)=\Sigma_\infty.$$
    Thus, by Theorem \ref{flat measures - kowalski preiss}, we know that $\Sigma_\infty$ must be an $n$-plane or a light cone as in \eqref{flat measures - light cone}.\\
    
    We will now use the fact that $T^\Lambda_{Q_i,\rho_i}\to\Sigma_\infty$ with respect to $D$ and \eqref{flat measures - fair flatness} to rule out the case in which $\Sigma_\infty$ is a light cone. Let $X\in\Sigma_\infty$. By Lemma \ref{pseudo tangents - characterization of support}, there exist points $Z_i\in\Sigma$ such that if
    $$X_i = T^\Lambda_{Q_i,\rho_i}(Z_i),$$
    then $X_i\to X$ as $i\to\infty$. Notice that this implies $|Z_i-Q_i|\to 0$. Assume without loss of generality that $|X-X_i|\leq 1/2$, $|Q-Q_i|\leq 1/2$ and $|Q_i-Z_i|\leq 1/2$. Observe that then $Z_i\in (\Sigma\cap K; 1)$. We consider two auxiliary radii that will help us compare $\Sigma$ with $\Sigma_\infty$,
    $$r_i = \rho_i(1+|X-X_i|),\quad s_i = \rho_i(1-|X-X_i|).$$
    
    We start with a compatibility statement about minimizing planes for $b\beta_\Sigma(Z_i,\cdot)$ at certain scales. For each $i$, let 
    $$r_i' = \lmax(K)r_i,\quad s_i' = \lmax(K)s_i.$$
    Since $\rho_i\to 0$ as $i\to\infty$, we can assume that $r_i'$, $s_i'\leq r_K$ if $i$ is large enough depending on $K$ and $\Lambda$, where $r_K$ is as in the statement of Proposition \ref{flat measures - main result}. First, by \eqref{flat measures - fair flatness} there are $n$-planes $P(Z_i, r_i')$, $P(Z_i, s_i')$ such that
    \begin{equation}
    \label{flat measures - flat ri'}
        D[\Sigma\cap B(Z_i,r_i');P(Z_i,r_i')\cap B(Z_i,r_i')]\leq \delta_Kr_i',
    \end{equation}
    \begin{equation}
    \label{flat measures - flat si'}
        D[\Sigma\cap B(Z_i,s_i');P(Z_i,s_i')\cap B(Z_i,s_i')]\leq \delta_Ks_i'.
    \end{equation}
    Note that by \eqref{flat measures - flat ri'}, \eqref{flat measures - flat si'} and Corollary \ref{prelim - lambda and normal reif flat corollary}, if $\delta_K<\min\{\lmin(K),e_\Lambda(K)^{-1}\}$, then
    \begin{equation}
    \label{flat measures - plane of ri}
        D[\Sigma\cap B_\Lambda(Z_i,r_i);P(Z_i,r_i')\cap B_\Lambda(Z_i,r_i)]\leq M_K\delta_Kr_i,
    \end{equation}
    \begin{equation}
    \label{flat measures - plane of si}
        D[\Sigma\cap B_\Lambda(Z_i,s_i);P(Z_i,s_i')\cap B_\Lambda(Z_i,s_i)]\leq M_K\delta_Ks_i.
    \end{equation}
    \\
    \textit{Claim:} if $\delta_K<\lmin(K)M_K^{-1}/3$, then
    \begin{equation}
    \label{flat measures - close planes}
        P(Z_i,r_i')\cap B_\Lambda(Z_i,s_i)\subset (P(Z_i,s_i')\cap B_\Lambda( Z_i,s_i);\sigma_K\delta_K(s_i+2r_i)),
    \end{equation}
    where $\sigma_K>0$ depends only on $K$ and $\Lambda$.
    \begin{proof}[Proof of the claim]

    The proof of this is analogue to the one in \cite{KT99} for round balls. Given $Y\in P(Z_i,r_i')\cap B_\Lambda(Z_i,s_i)$, write $Y=Z_i+\Lambda(Z_i)W$, where $|W|<s_i$. Consider
    $$\overline{Y}=Z_i+\Lambda(Z_i)\left(\left[1-\frac{M_K\delta_K r_i'}{\lmin(K)s_i'}\right]W\right).$$
    Using that $r_i'/s_i'\leq 3$ and our assumption on $\delta_K$, we see that for all $i$

    \begin{equation}
        \label{pseudo tangents - scaling factor is positive}
        1-\frac{M_K\delta_K r_i'}{\lmin(K)s_i'}>0.
    \end{equation}
    Next, since $(1-\frac{M_K\delta_Kr_i'}{\lmin(K)s_i'})|W|<|W|<s_i$, we have
    \begin{equation*}
        \begin{split}
            \overline{Y}\in B_\Lambda(Z_i,s_i)&\subset B_\Lambda(Z_i,r_i), 
        \end{split}
    \end{equation*}
    and
    \begin{equation}
    \begin{split}
    \label{flat measures - y bar minus y}
        |\Lambda(Z_i)^{-1}(\overline{Y}-Y)|=\frac{M_K\delta_K r_i'}{\lmin(K)s_i'}|W|\leq \frac{M_K\delta_Kr_i}{\lmin(K)s_i}|W|< \lmin(K)^{-1}M_K\delta_Kr_i.
        \end{split}
    \end{equation}
    Moreover, $Y\in P(Z_i,r_i')$ implies that $\overline{Y}\in P(Z_i,r_i')$ as well, by construction. Combining this with \eqref{flat measures - y bar minus y} and recalling that $\lmin(K)^{-1}M_K\delta_K<1/3$, we see that 
    $$\overline{Y}\in P(Z_i,r_i')\cap B_\Lambda(Z_i,r_i).$$
    Thus we can apply \eqref{flat measures - plane of ri} to obtain a point $Z\in \Sigma\cap B_\Lambda(Z_i,r_i)$ such that
    \begin{equation}
        \label{flat meaasures - p minus y bar}
        |Z-\overline{Y}|\leq M_K\delta_K r_i.
    \end{equation}
    Using \eqref{flat meaasures - p minus y bar} and the definition of $\overline{Y}$,
    \begin{align*}
        |\Lambda(Z_i)^{-1}(Z-Z_i)|&\leq |\Lambda(Z_i)^{-1}(Z-\overline{Y})|+|\Lambda(Z_i)^{-1}(\overline{Y}-Z_i)|\\
        &\leq \lmin(K)^{-1}|Z-\overline{Y}|+s_i-\lmin(K)^{-1}M_K\delta_K r_i\\
        &\leq \lmin(K)^{-1}M_K\delta_K r_i+s_i-\lmin(K)^{-1}M_K\delta_K r_i= s_i,
    \end{align*}
    so $Z\in B_\Lambda(Z_i,s_i)$. But we also know $Z\in\Sigma$, so $Z\in\Sigma\cap B_\Lambda(Z_i,s_i)$. Therefore, by \eqref{flat measures - plane of si} there exists $Y'\in P(Z_i,s_i')\cap B_\Lambda(Z_i,s_i)$ such that
    \begin{equation}
    \label{flat measures - y' minus p}
        |Y'-Z|\leq M_K\delta_Ks_i.
    \end{equation}
    Combining \eqref{flat measures - y bar minus y}, \eqref{flat meaasures - p minus y bar} and \eqref{flat measures - y' minus p}, we obtain
    \begin{equation}
    \begin{split}
    \label{pseudo tangents - y minus y bar}
        |Y-Y'|&\leq |Y-\overline{Y}|+|\overline{Y}-Z|+|Z-Y'|\\
        &\leq |\Lambda(Z_i)\Lambda(Z_i)^{-1}(\overline{Y}- Y)| + M_K\delta_Kr_i + M_K\delta_Ks_i\\
        &\leq e_\Lambda(K)M_K\delta_Kr_i + M_K\delta_Kr_i + M_K\delta_Ks_i\leq \sigma_K\delta_K(s_i+2r_i),
        \end{split}
    \end{equation}
    where $\sigma_K=M_K\max\{e_\Lambda(K),1\}$. This completes the proof of the claim.
    \end{proof}
    
    As a next step, we want to unravel \eqref{flat measures - sigma infinity} into estimates that capture how closely $\Sigma$ can be approximated by an affine copy of $\Sigma_\infty$ near $Q$. Let $\varepsilon>0$. Equation \eqref{flat measures - sigma infinity} guarantees that if $i$ is large enough depending on $\varepsilon$ and $K$, then
    \begin{equation}
        \label{flat measures - sigma is close to sigma infinity}
        D[\Sigma_\infty\cap B(X,1),T^\Lambda_{Q_i,\rho_i}(\Sigma)\cap B(X,1)]\leq\varepsilon.
    \end{equation}
    We will use this estimate to obtain inclusions in two directions.\\
\\
    1. On one hand, \eqref{flat measures - sigma is close to sigma infinity} implies
    \begin{align*}
        T^\Lambda_{Q_i,\rho_i}(\Sigma)\cap B(X_i,1-|X-X_i|)&\subset T^\Lambda_{Q_i,\rho_i}(\Sigma)\cap B(X,1)\subset (\Sigma_\infty\cap B(X,1);\varepsilon).
    \end{align*}
    Applying $(T^\Lambda_{Q_i,\rho_i})^{-1}(\cdot) = Q_i + \rho_i\Lambda(Q_i)(\cdot)$, we get
    \begin{equation}
    \label{flat measures - first auxiliary inclusion}
    \begin{split}
        \Sigma\cap [Q_i+\rho_i\Lambda(Q_i)B(X_i,1-|X-X_i|)]&\subset ((T^\Lambda_{Q_i,\rho_i})^{-1}[\Sigma_\infty\cap B(X,1)];\lmax(Q_i)\rho_i\varepsilon)\\
        &\subset ((T^\Lambda_{Q_i,\rho_i})^{-1}[\Sigma_\infty\cap B(X,1)];\lmax(K)\rho_i\varepsilon).
        \end{split}
    \end{equation}
    We would like to adjust the left hand side in a way that it looks like the intersection of $\Sigma$ with a suitable ellipse. We proceed as follows,
    \begin{align*}
        B_\Lambda(Z_i,s_i)
        &= Z_i + \rho_i\Lambda(Z_i)B(0,1-|X-X_i|)\\
        &\subset Z_i+\rho_i\Lambda(Q_i)B(0,1-|X-X_i|) + \rho_i(\Lambda(Z_i)-\Lambda(Q_i))B(0,1-|X-X_i|)\\
        &\subset Q_i + \rho_i\Lambda(Q_i)X_i + \rho_i\Lambda(Q_i)B(0,1-|X-X_i|) + B(0,\varepsilon\rho_i(1-|X-X_i|))\\
        &= Q_i+\rho_i\Lambda(Q_i)B(X_i,1-|X-X_i|) + B(0,\varepsilon\rho_i(1-|X-X_i|))\\
        &\subset (Q_i+\rho_i\Lambda(Q_i)B(X_i,1-|X-X_i|); \varepsilon\rho_i(1-|X-X_i|)),
    \end{align*}
where the third line holds for all $i$ large enough, depending on $K$ and $\Lambda$, by continuity of $\Lambda$. Combining this with \eqref{flat measures - first auxiliary inclusion} and recalling that $\rho_i(1-|X-X_i|)=s_i$, we get
\begin{equation}
\label{flat measures - first inclusion between sigmas}
\begin{split}
    \Sigma\cap B_\Lambda (Z_i,s_i)&\subset (\Sigma\cap [Q_i+\rho_i\Lambda(Q_i)B(X_i,1-|X-X_i|)];\varepsilon s_i)\\
    &\subset ((T^\Lambda_{Q_i,\rho_i})^{-1}[\Sigma_\infty\cap B(X,1)];\varepsilon s_i + \lmax(K)\rho_i\varepsilon).
    \end{split}
\end{equation}
2. The other inclusion we can extract from \eqref{flat measures - sigma is close to sigma infinity} is
\begin{align*}
    \Sigma_\infty\cap B(X,1)&\subset (T^\Lambda_{Q_i,\rho_i}(\Sigma)\cap B(X,1);\varepsilon)\subset (T^\Lambda_{Q_i,\rho_i}(\Sigma)\cap B(X_i,1+|X-X_i|);\varepsilon).
\end{align*}
Applying $(T^\Lambda_{Q_i,\rho_i})^{-1}(\cdot)$ as before, this inclusion gives
\begin{equation}
\label{flat measures - second auxiliary inclusion}
\begin{split}
    (T^\Lambda_{Q_i,\rho_i})^{-1}[\Sigma_\infty\cap B(X,1)]\subset (\Sigma\cap\left\lbrace Q_i+\rho_i\Lambda(Q_i)B(X_i,1+|X-X_i|)\right\rbrace;\lmax(K)\rho_i\varepsilon).
\end{split}
\end{equation}
Proceeding similarly as above, we can make the right hand side look like the intersection of $\Sigma$ with a suitable ellipse. Namely, if $i$ is large enough depending on $K$ and $\Lambda$,
    \begin{align*}
        Q_i+&\rho_i\Lambda(Q_i)(B(X_i,1+|X-X_i|))\\
        &= Q_i+\rho_i\Lambda(Q_i)(X_i+B(0,1+|X-X_i|))\\
        &= Z_i + \rho_i\Lambda(Q_i)B(0,1+|X-X_i|)\\
        &\subset Z_i + \rho_i\Lambda(Z_i)B(0,1+|X-X_i|) + \rho_i(\Lambda(Q_i)-\Lambda(Z_i))B(0,1+|X-X_i|)\\
        &= B_\Lambda(Z_i, r_i) + \rho_i(\Lambda(Q_i)-\Lambda(Z_i))B(0,1+|X-X_i|)\subset (B_\Lambda(Z_i, r_i);\varepsilon r_i).
    \end{align*}
    This and \eqref{flat measures - second auxiliary inclusion} give 
    \begin{equation}
    \label{flat measures - second inclusion between sigmas}
    \begin{split}
        (T^\Lambda_{Q_i,\rho_i})^{-1}[\Sigma_\infty\cap B(X,1)]\subset (\Sigma\cap B_\Lambda(Z_i,r_i);\lmax(K)\rho_i\varepsilon + \varepsilon r_i).
    \end{split}
    \end{equation}
    
    We would now like to use  \eqref{flat measures - first inclusion between sigmas} and \eqref{flat measures - second inclusion between sigmas} along with assumption \eqref{flat measures - fair flatness} to show that $\Sigma_\infty$ must be a plane. Recall the planes $P(Z_i,r_i')$, $P(Z_i,s_i')$ from \eqref{flat measures - plane of ri} and \eqref{flat measures - plane of si}. On one hand, using \eqref{flat measures - close planes}, \eqref{flat measures - plane of si} and \eqref{flat measures - first inclusion between sigmas}, and keeping in mind that $\rho_i-s_i=\rho_i|X-X_i|$, $ s_i\leq \rho_i$ and $r_i\leq 2\rho_i$, we see that
    \begin{equation}
    \label{flat measures - final inclusion 1}
        \begin{split}
            P(Z_i,r_i')\cap B_\Lambda(Z_i,\rho_i)&\subset (P(Z_i,r_i')\cap B_\Lambda(Z_i,s_i);\lmax(K)\rho_i|X-X_i|)\\
            &\subset (P(Z_i,s_i')\cap B_\Lambda(Z_i,s_i);C_K\rho_i|X-X_i|+\sigma_K\delta_K(s_i+2r_i))\\
            &\subset \left(\Sigma\cap B_\Lambda(Z_i,s_i); C_K(\rho_i|X-X_i|+s_i+2r_i) + M_K\delta_Ks_i\right)\\
            &\subset ((T^\Lambda_{q_i,\rho_i})^{-1}\left[\Sigma_\infty\cap B(X,1)\right]; C_K(\rho_i|X-X_i|+s_i+2r_i+\delta_K\rho_i)+ \varepsilon(s_i+\lmax(K)\rho_i))\\
            &\subset ((T^\Lambda_{q_i,\rho_i})^{-1}\left[\Sigma_\infty\cap B(X,1)\right]; C_K\rho_i(|X-X_i|+\delta_K+\varepsilon)).
        \end{split}
    \end{equation}
    Similarly, from \eqref{flat measures - second inclusion between sigmas}, \eqref{flat measures - plane of ri} and \eqref{flat measures - close planes}
    we get
    \begin{equation}
    \label{flat measures - final inclusion 2}
        \begin{split}
            (T^\Lambda_{Q_i,\rho_i})^{-1}[\Sigma_\infty\cap B(X,1)]&\subset(\Sigma\cap B_\Lambda(Z_i,r_i); \varepsilon(\lmax(K)\rho_i+r_i))\\
            &\subset (P(Z_i,r_i')\cap B_\Lambda(Z_i,r_i);\varepsilon(C_K\rho_i + r_i)+ M_K\delta_Kr_i)\\
            &\subset (P(Z_i,r_i')\cap B_\Lambda(Z_i,\rho_i); \varepsilon(C_K\rho_i + r_i)+ C_K\delta_K\rho_i + \lmax(K)\rho_i|X-X_i|)\\
            &\subset (P(Z_i,r_i')\cap B_\Lambda(Z_i,\rho_i); C_K\rho_i(|X-X_i|+\delta_K+\varepsilon)).
        \end{split}
    \end{equation}
    
    We now want to apply $T^\Lambda_{Q_i,\rho_i}(\cdot)$ on \eqref{flat measures - final inclusion 1} and \eqref{flat measures - final inclusion 2}. Notice that 
    \begin{equation}
    \label{flat measures - helper inclusion 5}
    \begin{split}
        T^\Lambda_{Q_i,\rho_i}(B_\Lambda(Z_i,\rho_i))&= \Lambda(Q_i)^{-1}\left(\frac{B_\Lambda(Z_i,\rho_i)-Q_i}{\rho_i}\right)\\
        &= \Lambda(Q_i)^{-1}\left(\frac{Z_i-Q_i}{\rho_i}\right)+\Lambda(Q_i)^{-1}\Lambda(Z_i)B(0,1)\\
        &= X_i+\Lambda(Q_i)^{-1}\Lambda(Z_i)B(0,1).
        \end{split}
    \end{equation}
    We would like to compare this set with $B(X_i,1)$. Notice that by continuity of $\Lambda$ and because $|Z_i-Q_i|\to 0$, if $i$ is large enough depending on $K$ and $\Lambda$, we have
    \begin{equation*}
    \|\Lambda(Q_i)^{-1}\Lambda(Z_i)-\mathrm{Id}\|\leq \varepsilon,\quad \|\Lambda(Z_i)^{-1}\Lambda(Q_i)-\mathrm{Id}\|\leq\varepsilon.  
    \end{equation*}
    
    We claim that this implies
    \begin{equation}
    \label{flat measures - helper inclusion 4}
        \begin{split}
            B(0,1-\varepsilon)&\subset \Lambda(Q_i)^{-1}\Lambda(Z_i)B(0,1)\subset B(0,1+\varepsilon).
        \end{split}
    \end{equation}
    To see this, note that on one hand,
    \begin{equation*}
        \begin{split}
            \Lambda(Q_i)^{-1}\Lambda(Z_i)B(0,1)&\subset (\Lambda(Q_i)^{-1}\Lambda(Z_i)-I)B(0,1)+B(0,1)\\
            &\subset B(0,\varepsilon)+B(0,1)= B(0,1+\varepsilon).
        \end{split}
    \end{equation*}
    On the other hand,
    \begin{equation*}
        \begin{split}
            B(0,1-\varepsilon)&\subset (I-\Lambda(Q_i)^{-1}\Lambda(Z_i))B(0,1-\varepsilon)+\Lambda(Q_i)^{-1}\Lambda(Z_i)B(0,1-\varepsilon)\\
            &\subset B(0,\varepsilon(1-\varepsilon))+\Lambda(Q_i)^{-1}\Lambda(Z_i)B(0,1-\varepsilon)\\
            &\subset \Lambda(Q_i)^{-1}\Lambda(Z_i)\left[(\Lambda(Z_i)^{-1}\Lambda(Q_i)-I)B(0,\varepsilon(1-\varepsilon))+B(0,\varepsilon(1-\varepsilon))\right]\\
            &\quad\quad\quad\quad\quad\quad\quad\quad\quad\quad\quad\quad\quad\quad\quad+\Lambda(Q_i)^{-1}\Lambda(Z_i)B(0,1-\varepsilon)\\
            &\subset \Lambda(Q_i)^{-1}\Lambda(Z_i)\left[B(0,\varepsilon^2)+B(0,\varepsilon(1-\varepsilon))\right]+\Lambda(Q_i)^{-1}\Lambda(Z_i)B(0,1-\varepsilon)\\
            &\subset \Lambda(Q_i)^{-1}\Lambda(Z_i)B(0,1).
        \end{split}
    \end{equation*}
    These inclusions prove \eqref{flat measures - helper inclusion 4}.
    
    Now, combining \eqref{flat measures - helper inclusion 5} and \eqref{flat measures - helper inclusion 4} we obtain
    \begin{equation}
        \label{flat measures - helper inclusion 6}
        \begin{split}
            B(X_i,1-\varepsilon)&\subset T^\Lambda_{Q_i,\rho_i}(B_\Lambda(Z_i,\rho_i))\subset B(X_i,1+\varepsilon).
        \end{split}
    \end{equation}
    Denote by $P_i$ the plane $T^\Lambda_{Q_i,\rho_i}(P(Z_i,r_i'))$, and notice that $X_i\in P_i$.
    Applying $T^\Lambda_{Q_i,\rho_i}$ on \eqref{flat measures - final inclusion 1} and \eqref{flat measures - final inclusion 2}, we obtain
    \begin{equation*}
        P_i\cap T^\Lambda_{Q_i,\rho_i}(B_\Lambda(Z_i,\rho_i))\subset (\Sigma_\infty\cap B(X,1);C_K(|X-X_i|+\delta_K+\varepsilon)),
    \end{equation*}
    \begin{equation*}
        \Sigma_\infty\cap B(X,1)\subset (P_i\cap T^\Lambda_{Q_i,\rho_i}(B_\Lambda(Z_i,\rho_i));C_K(|X-X_i|+\delta_K+\varepsilon)).
    \end{equation*}
    Taking now \eqref{flat measures - helper inclusion 6} into account, the last two inclusions above give, respectively,
    \begin{equation*}
    \label{flat measures - helper inclusion 7}
        \begin{split}
            P_i\cap B(X_i,1)&\subset (P_i\cap B(X_i,1-\varepsilon);\varepsilon)\subset (P_i\cap T^\Lambda_{Q_i,\rho_i}(B_\Lambda(Z_i,\rho_i));\varepsilon)\\
            &\subset (\Sigma_\infty\cap B(X,1);C_K(|X-X_i|+\delta_K+\varepsilon)),
        \end{split}
    \end{equation*}
    and
    \begin{equation*}
    \label{flat measures - helper inclusion 8}
        \begin{split}
            \Sigma_\infty\cap B(X,1)&\subset (P_i\cap T^\Lambda_{Q_i,\rho_i}(B_\Lambda(Z_i,\rho_i));C_K(|X-X_i|+\delta_K+\varepsilon))\\
            &\subset (P_i\cap B(X_i,1+\varepsilon);C_K(|X-X_i|+\delta_K+\varepsilon))\\
            &\subset (P_i\cap B(X_i,1); C_K(|X-X_i|+\delta_K+\varepsilon)).
        \end{split}
    \end{equation*}
   These inclusions show that
    \begin{equation}
        \label{flat measures - sigma infty is close to a plane}
        D[\Sigma_\infty\cap B(X,1);P_i\cap B(X_i,1)]\leq C_K(|X-X_i|+\delta_K+\varepsilon).
    \end{equation}
    
    To conclude, we want to replace $X_i$ with $X$ in this estimate and use it to rule out the case in which $\Sigma_\infty$ is a cone. Let $P_i' = P_i -X_i+X$. Since $X_i\in P_i$, we have $X\in P_i'$. Also, note that
    $$D[P_i'\cap B(X,1);P_i\cap B(X_i,1)]\leq |X-X_i|.$$
    Combining this with \eqref{flat measures - sigma infty is close to a plane}, we get
    \begin{equation}
    \label{flat measures - helper inequality 10}
        D[\Sigma_\infty\cap B(X,1);P_i'\cap B(X,1)]\leq C_K(|X-X_i|+\delta_K+\varepsilon).
    \end{equation}
    By compactness of the space of $n$-planes through $X$ in $\R^{n+1}$, we can assume upon passing to a subsequence, that there is an $n$-plane $P_X$ through $X$ such that $P_i'\to P_X$ with respect to $D$, uniformly on compact sets. By \eqref{flat measures - helper inequality 10}, $P_X$ satisfies
    \begin{equation}
    \label{flat measures - helper inequality 11}
        D[\Sigma_\infty\cap B(X,1);P_X\cap B(X,1)]\leq C_K(\delta_K+\varepsilon).
    \end{equation}
    
    Now, by Theorem \ref{flat measures - kowalski preiss}, if $\Sigma_\infty$ is not an $n$-plane, then $n\geq 3$ and there exist $X_\infty\in\Sigma_\infty$ and a rotation $\mathcal{R}$ of $\R^{n+1}$ such that $\mathcal{R}(\Sigma_\infty - X_\infty)$ is the light cone 
    $$\mathcal{C}=\{(x_1,\cdots, x_{n+1})\in \R^{n+1}: x_4^2=x_1^2+x_2^2+x_3^2\}.$$
    In such scenario, applying \eqref{flat measures - helper inequality 11} with $X=X_\infty$ and denoting by $L$ the plane $\mathcal{R}(P_{X_\infty}-X_\infty)$, we get
    \begin{equation*}
        \begin{split}
            D[\mathcal{C}\cap B(0,1);L\cap B(0,1)]&= D[\Sigma_\infty\cap B(X_\infty,1);P_{X_\infty}\cap B(X_\infty,1)]\leq C_K(\delta_K+\varepsilon).
        \end{split}
    \end{equation*}
    Notice that since $P_{X_\infty}$ contains $X_\infty$, $L$ must contain the origin. Then if $\delta_K<C_K^{-1}/\sqrt{2}$ and $\varepsilon$ is small enough,
    \begin{equation}
    \label{flat measures - helper inequality 12}
    D[\mathcal{C}\cap B(0,1);L\cap B(0,1)]<\frac{1}{\sqrt{2}}.
    \end{equation}
    However, a quick calculation shows that this inequality fails for every plane $L$ through the origin. It follows that $\Sigma_\infty$ must be an $n$-plane. Moreover, since $\Sigma_\infty=\mathrm{spt}(\nu)$ and $\nu$ is a $\Lambda$-pseudo tangent measure of $\mu$, we have $0\in\Sigma_\infty$ by Remark \ref{pseudo tangents - 0 is in the support}. So we can use \eqref{flat measures - sigma is close to sigma infinity} with $X=0$ to get
    \begin{equation*}
        D[\Sigma_\infty\cap B(0,1);T^\Lambda_{Q_i,\rho_i}(\Sigma)\cap B(0,1)]\leq \varepsilon.
    \end{equation*}
    Applying $T^\Lambda_{Q_i,\rho_i}(\cdot)=Q_i+\rho_i\Lambda(Q_i)(\cdot)$, we obtain
    \begin{equation}
    \label{flat measures - helper inequality 13}
        D[\Sigma_\infty^{(i)}\cap B_\Lambda(Q_i,\rho_i);\Sigma\cap B_\Lambda(Q_i,\rho_i)]\leq \lmax(K)\rho_i\varepsilon,
    \end{equation}
    where $\Sigma_\infty^{(i)}=T^\Lambda_{Q_i,\rho_i}(\Sigma_\infty)$. Notice that $\Sigma_\infty^{(i)}$ is an $n$-plane containing $Q_i$. Now recall that $\rho_i=\lmin(K)^{-1}\tau_i$, so if we combine \eqref{flat measures - helper inequality 13} with Corollary \ref{prelim - lambda and normal reif flat corollary}, we get 
    \begin{equation*}
        \begin{split}
            D[\Sigma_\infty^{(i)}\cap B(Q_i,\tau_i);\Sigma\cap B(Q_i,\tau_i)]&\leq 2\lmax(K)\rho_i\varepsilon\\
            &= 2e_\Lambda(K)\tau_i\varepsilon.
        \end{split}
    \end{equation*}
    Combining this with \eqref{flat measures - tau i approximates ell}, we deduce that
    \begin{equation*}
        \begin{split}
            \lim_{\tau\to 0}b\beta_\Sigma(K,\tau)&=\lim_{i\to\infty}b\beta_\Sigma(Q_i,\tau_i)\leq \limsup_{i\to\infty}\frac{1}{\tau_i}D[\Sigma_\infty^{(i)}\cap B(Q_i,\tau_i);\Sigma\cap B(Q_i,\tau_i)]\leq 2e_\Lambda(K)\varepsilon.
        \end{split}
    \end{equation*}
    Since this holds for every $\varepsilon>0$, we conclude that
    $$\lim_{\tau\to 0}b\beta_\Sigma(K,\tau)=0,$$
    completing the proof of Proposition \ref{flat measures - main result}.
    \end{proof}

\section{Proof of Theorem \ref{theorem 2}}\label{asod}

    The key idea of the proof is that the doubling condition \eqref{introduction - lambda holder asod} can be used to obtain information about the density $\Theta_\Lambda(\mu,X)$ introduced in \eqref{introduction - elliptic density definition}. More specifically, the assumptions of Theorem \ref{theorem 2} imply that \eqref{introduction - main density estimate} holds when $\mu$ is replaced with a certain measure which has the same support as $\mu$, and $\alpha$ is replaced with a number that depends on $\alpha$ and $\beta$, making Theorem \ref{theorem 1} applicable. These ideas are contained in the following lemma.

    \begin{lemma}
    \label{asod - technical lemma}
        Let $\Lambda$ and $\mu$ be as in the assumptions of Theorem \ref{theorem 2}. Then 
        \begin{equation}
            \label{asod - theta is finite}
            0<\Theta_\Lambda(\mu,X)<\infty,
        \end{equation}
        for every $X\in\Sigma=\mathrm{spt}(\mu)$. Also, for every compact set $K\subset\mathbb R^{n+1}$ there exists a constant $C_K>0$ depending on $K$ and $\Lambda$, such that
        \begin{equation}
            \label{asod - log theta holder continuous}
            |\log\Theta_\Lambda(X)-\log\Theta_\Lambda(Y)|\leq C_K|X-Y|^{\frac{\gamma}{1+\alpha}},
        \end{equation}
        whenever $X,Y\in\Sigma\cap K$ and $|X-Y|\leq \Delta_K$, where $\Delta_K>0$ is small enough depending on $K$ and $\Lambda$, and $\gamma=\min\{\alpha,\beta\}$. Moreover, the measure 
        $$\de\mu_0(X)=\frac{1}{\Theta_\Lambda(\mu,X)}\de\mu(X)$$
        is a Radon measure with $\mathrm{spt}(\mu_0)=\Sigma$, with the property that for every compact set $K\subset\mathbb R^{n+1}$ there exist $r_K>0$ and $C_K>0$ such that for every $X\in K\cap\Sigma$ and $r\in(0,r_K]$,
        \begin{equation}
            \label{asod - mu0 has density 1}
            \left|\frac{\mu_0(B_\Lambda(X,r))}{\omega_nr^n}-1\right|\leq C_Kr^{\gamma'},
        \end{equation}
        where $\gamma'=\frac{\min\{\alpha,\beta\}}{1+\alpha}$. 
        
    \end{lemma}
   \begin{proof}
       The proof is similar to that of \cite[Proposition 6.1]{DKT01}.
       Let $K$ be as in the statement and let $X\in \Sigma\cap K$. For $r_k=2^{-k}$, $k\geq 0$, let
       $$D_k(X)=\frac{\mu(B_\Lambda(X,r_k))}{\omega_nr_k^n},\quad l_k=\log D_k(X),$$
       and for any $t\in[\frac{1}{2},1]$, let
       $$R_t(X,r)=\frac{\mu(B_\Lambda(X,tr))}{\mu(B_\Lambda(X,r))}-t^n.$$
       Notice that
       \begin{align*}
           \frac{D_{k+1}(X)}{D_k(X)}-1&=\frac{2^n\mu(B_\Lambda(X,r_{k+1}))}{\mu(B_\Lambda(X,r_k))}-1= 2^nR_{1/2}(X,r_k).
       \end{align*}
       By \eqref{introduction - lambda holder asod}, we have
       \begin{equation}
           \label{asod - asod bound on R}
           2^nR_{1/2}(X,r_k)\leq C_K2^{-k\alpha}.
       \end{equation}
        Thus, if $k_0$ is large enough and $k\geq k_0$,
        \begin{equation}
            \label{asod - lk cauchy}
            \begin{split}
                |l_{k+1}-l_k|&=\left|\log\frac{D_{k+1}(X)}{D_k(X)}\right|\leq C_K2^{-k\alpha}.
            \end{split}
        \end{equation}
        This implies that the sequence $\{l_k\}$ is Cauchy, so $l_\infty:=\lim_{k\to\infty}l_k$ exists and is finite, and we have
        \begin{equation}
            \label{asod - Dk converges}
        \lim_{k\to\infty}D_k(X)=e^{l_\infty}.\end{equation}
        It also follows from \eqref{asod - lk cauchy} that if $k_0$ is large enough,
        \begin{equation}
            \label{asod - l infty - lk}
            |l_k-l_\infty|\leq C_K2^{-k\alpha}.
        \end{equation}
        
        We will show that
        \begin{equation}
            \label{asod - e to the l infty equals density}
            \Theta_\Lambda(\mu,X)=e^{l_\infty}.
        \end{equation}
        Let $r\in (0,1)$, and write $r=tr_k$ for some $t\in[\frac{1}{2},1]$ and some $k\geq 0$. Then 
        \begin{equation}
        \label{asod - expansion of density ratio}
            \begin{split}
                \frac{\mu(B_\Lambda(X,r))}{\omega_nr^n}&=\frac{\mu(B_\Lambda(X,tr_k))}{\omega_nt^nr_k^n}\\
                &=\frac{\mu(B_\Lambda(X,tr_k))}{t^n\mu(B_\Lambda(X,r_k))} D_{k}(X)= t^{-n}(R_t(X,r_k)+t^n)D_k(X).
            \end{split}
        \end{equation}
        Letting $r\to 0$, we have $r_k\to 0$, $R_t(X,r_k)\to 0$ by \eqref{introduction - lambda holder asod}, and $D_k(X)\to e^{l_\infty}$ by \eqref{asod - Dk converges}. Thus, \eqref{asod - expansion of density ratio} yields \eqref{asod - e to the l infty equals density}, and in particular
        $$0<\Theta_\Lambda(\mu,X)<\infty.$$
        
      We will now prove \eqref{asod - log theta holder continuous}. Let us denote $\delta=\log(1+t^{-n}R_t(X,r_k))$. By \eqref{introduction - lambda holder asod}, and keeping in mind that $t\geq 1/2$ and $r_k\leq 2r$, if $r$ is small enough depending on $K$ and $\Lambda$, we have
      \begin{equation}
      \label{asod - bound on delta}
          \begin{split}
              |\delta|&\leq \log(1+t^{-n}C_Krk^\alpha)\leq C_Kr^\alpha.
          \end{split}
      \end{equation}
      Notice that if $r$ is small enough, then by \eqref{asod - l infty - lk}, \eqref{asod - expansion of density ratio} and \eqref{asod - bound on delta},
      \begin{equation}
          \begin{split}
              \label{asod - log ratio minu log density}\left|\log\frac{\mu(B_\Lambda(X,r))}{\omega_nr^n}-\log \Theta_\Lambda(\mu,X)\right|&\leq |\delta|+|\log D_k(X)-\log \Theta_\Lambda(\mu,X)|\\
              &= |\delta|+|l_k-l_\infty|\leq C_Kr^\alpha,
          \end{split}
      \end{equation}
      where we have used that $2^{-k\alpha}=r_k^\alpha\leq 2^\alpha r^\alpha$. We will show that \eqref{asod - log theta holder continuous} holds when $|X-Y|$ is small, depending on $K$ and $\Lambda$. Suppose $|X-Y|^{\frac{1}{1+\alpha}}< r_{k_0}$, with $k_0$ as in \eqref{asod - lk cauchy} and \eqref{asod - l infty - lk}. Let $k\geq k_0$ be such that 
      $$r_{k+1}\leq |X-Y|^{\frac{1}{1+\alpha}}<r_k.$$
      
      Choosing $k_0$ large enough depending on $K$ and $\Lambda$, we have
      
      $$|X-Y|\leq r_k^{1+\alpha}\leq \lmin(K)\frac{r_k}{2},$$ 
      where $\lmin(K)$ is as in \eqref{prelim - lmin and lmax of K}. In particular, we can apply Lemma \ref{prelim - lemma on nonconcentric ellipses} to ensure that
      $$B_\Lambda(Y,r_k-\lmin(X)^{-1}|X-Y|-C_Kr_k^{1+\beta})\subset B_\Lambda(X,r_k).$$
      Now, using again that $|X-Y|<r_k^{1+\alpha}$, setting $\gamma=\min\{\alpha,\beta\}$ we obtain
      \begin{equation}
      \label{asod - ellipse containment}
          \begin{split}
              r_k-\lmin(X)^{-1}|X-Y|-C_Kr_k^{1+\beta}&\geq r_k-C_K(r_k^{1+\alpha}+r_k^{1+\beta})\geq r_k(1-C_Kr_k^{\gamma})\\
              &= r_k(1-C_Kr_{k+1}^\gamma)\geq r_k(1-C_K|X-Y|^{\frac{\gamma}{1+\alpha}}).
          \end{split}
      \end{equation}
      Denoting $\rho=r_k(1-C_K|X-Y|^{\frac{\gamma}{1+\alpha}})$, we see that \eqref{asod - ellipse containment} implies $B_\Lambda(Y,\rho)\subset B_\Lambda(X,r_k),$
      and thus
      $$\frac{\mu(B_\Lambda(Y,\rho))}{\omega_n\rho^n}\leq\frac{\mu(B_\Lambda(X,r_k))}{\omega_n\rho^n}=\frac{r_k^n}{\rho^n}D_k(X).$$
      Therefore,
      \begin{equation}
          \begin{split}
              \label{asod - bound for log of ratio}
              \log\left(\frac{\mu(B_\Lambda(Y,\rho))}{\omega_n\rho^n}\right)&\leq n\log \frac{r_k}{\rho}+l_k\leq n\log\frac{r_k}{\rho}+l_\infty+C_K2^{-k\alpha}\\
              &= l_\infty+C_Kr_k^{\alpha}-n\log\frac{\rho}{r_k}\\
              &\leq l_\infty+C_K|X-Y|^{\frac{\alpha}{1+\alpha}}-n\log\left(1-C_K|X-Y|^{\frac{\gamma}{1+\alpha}}\right)\\
              &\leq l_\infty+C_K|X-Y|^{\frac{\alpha}{1+\alpha}}+C_K|X-Y|^{\frac{\gamma}{1+\alpha}}\leq l_\infty+C_K|X-Y|^{\frac{\gamma}{1+\alpha}}.
          \end{split}
      \end{equation}
      
      On the other hand, we know by \eqref{asod - log ratio minu log density} that
      \begin{equation}
          \label{asod - log ratio minus log density plus more}
          \left|\log\frac{\mu(B_\Lambda(Y,\rho))}{\omega_n\rho^n}-\log\Theta_\Lambda(\mu,Y)\right|\leq C_K\rho^\alpha\leq C_K|X-Y|^{\frac{\alpha}{1+\alpha}}.
      \end{equation}
      Thus, combining \eqref{asod - e to the l infty equals density}, \eqref{asod - bound for log of ratio} and \eqref{asod - log ratio minus log density plus more} we obtain
      \begin{equation*}
          \log \Theta_{\Lambda}(\mu,Y)\leq l_\infty+C_K|X-Y|^{\frac{\gamma}{1+\alpha}}=\log\Theta_\Lambda(\mu,X)+C_K|X-Y|^{\frac{\gamma}{1+\alpha}}.
      \end{equation*}
      An analog argument can be used to show that
      $$\log \Theta_{\Lambda}(\mu,X)\leq \log\Theta_\Lambda(\mu,Y)+C_K|X-Y|^{\frac{\gamma}{1+\alpha}},$$
      from which \eqref{asod - log theta holder continuous} follows.

       Now we continue with the measure $\mu_0$ defined in the statement of Lemma \ref{asod - technical lemma}. From \eqref{asod - theta is finite} and \eqref{asod - log theta holder continuous}, it follows that $\Theta_\Lambda(\mu,\cdot)$ is locally bounded above and below by positive constants. This implies that $\mu_0$ is a Radon measure with support $\mathrm{spt}(\mu_0)=\mathrm{spt}(\mu)=\Sigma$. We will show that \eqref{asod - mu0 has density 1} holds. Let $X\in K\cap\Sigma$, and suppose that $0<r\leq \lmax(K)^{-1}r_K$ and $r_K<1$. Then every $Y\in B_\Lambda(X,r)$ satisfies $|X-Y|< r_K$ and $Y\in\Sigma\cap (K;1)$, so applying \eqref{asod - log theta holder continuous} to the compact set $(K;1)$,
      \begin{equation}
      \label{asod - bound on eta}
          \eta:=\sup_{Y\in \Sigma\cap B_\Lambda(X,r)}|\log\Theta_\Lambda(\mu,X)-\log\Theta_\Lambda(\mu,Y)|\leq C_Kr^{\frac{\gamma}{1+\alpha}}.
      \end{equation}
      
      From the definition of $\eta$, it follows that for every $Y\in\Sigma\cap B_\Lambda(X,r)$,
      $$e^{-\eta}\leq\frac{\Theta_\Lambda(\mu,Y)}{\Theta_\Lambda(\mu,X)}\leq e^\eta.$$
      If we integrate this inequality with respect to $\de\mu_0(Y)$ over $B_\Lambda(X,r)$, we get
      $$\mu_0(B_\Lambda(X,r))e^{-\eta}\leq\frac{\mu(B_\Lambda(X,r))}{\Theta_\Lambda(\mu,X)}\leq \mu_0(B_\Lambda(X,r))e^\eta,$$
      or equivalently,
      $$e^{-\eta}\leq\frac{\Theta_\Lambda(\mu,X)\mu_0(B_\Lambda(X,r))}{\mu(B_\Lambda(X,r))}\leq e^\eta.$$
      This implies that
      \begin{equation}
          \begin{split}
              \label{asod - log of ratio of mu0}
              \left|\log\frac{\mu_0(B_\Lambda(X,r))}{\omega_nr^n}\right|&=\left|\log\left(\frac{\Theta_\Lambda(X,r)\mu_0(B_\Lambda(X,r))}{\mu(B_\Lambda(X,r))}\cdot\frac{\mu(B_\Lambda(X,r))}{\Theta_\Lambda(\mu,X)\omega_nr^n}\right)\right|\\
              &\ \\
              &\leq \eta+\left|\log\frac{\mu(B_\Lambda(X,r))}{\omega_nr^n}-\log\Theta_\Lambda(\mu,X)\right|.
          \end{split}
      \end{equation}
      It then follows from \eqref{asod - log ratio minu log density}, \eqref{asod - bound on eta} and \eqref{asod - log of ratio of mu0} that
      $$\left|\log\frac{\mu_0(B_\Lambda(X,r))}{\omega_nr^n}\right|\leq C_K(r^{\frac{\gamma}{1+\alpha}}+r^\alpha)\leq C_Kr^{\frac{\gamma}{1+\alpha}},$$
      or equivalently
      $$e^{-C_Kr^{\gamma'}}\leq\frac{\mu_0(B_\Lambda(X,r))}{\omega_nr^n}\leq e^{C_Kr^{\gamma'}},$$
      where $\gamma'=\frac{\gamma}{1+\alpha}$. Thus, for all $r>0$ small enough depending on $K$ and $\Lambda$,
      \begin{equation}
      \label{asod - estimate X}
          1-C_Kr^{\gamma'}\leq\frac{\mu_0(B_\Lambda(X,r))}{\omega_nr^n}\leq 1+C_Kr^{\gamma'},
      \end{equation}
       completing the proof of the lemma. 
   \end{proof}
   We are now ready to prove Theorem \ref{theorem 2}.
   \begin{proof}[Proof of Theorem \ref{theorem 2}]
       Let $\mu$ be as in the assumptions of the Theorem and $\mu_0$ as in Lemma \ref{asod - technical lemma}. Lemma \ref{asod - technical lemma} implies that $\Sigma=\mathrm{spt}(\mu)=\mathrm{spt}(\mu_0)$, and by  \eqref{asod - mu0 has density 1}, $\mu_0$ satisfies the density condition \eqref{introduction - main density estimate} with $\alpha$ replaced by $\frac{\min\{\alpha,\beta\}}{1+\alpha}$. Thus, by Theorem \ref{theorem 1}, $\Sigma$ is a $C^{1,\gamma}$ submanifold of dimension $n$ of $\mathbb R^{n+1}$, where $\gamma\in (0,1)$ depends on $\alpha$ and $\beta$.   
   \end{proof}

\section{Proof of theorem \ref{theorem 3}: the regular set}
\label{nonflat case}

We define the regular set of $\Sigma=\mathrm{spt}(\mu)$ as
\begin{equation}
    \label{nonflat case - regular set}
    \mathcal{R}=\{X\in\Sigma:\limsup_{r\searrow 0}b\beta_{\Sigma}(X,r)=0\},
\end{equation}
and the singular set as $\mathcal{S}=\Sigma\backslash\mathcal{R}$. We devote this section to the study of $\mathcal{R}$. The techniques used here are based on ideas developed in \cite{PTT08} in the Euclidean setting.

\begin{proposition}[The regular set]
\label{nonflat case - regular set}
    Under the assumptions of Theorem \ref{theorem 3}, the set $\mathcal{R}$ defined in \eqref{nonflat case - regular set} is a $C^{1,\gamma}$ submanifold of dimension $n$ of $\R^{n+1}$. In particular, $\mathcal{R}$ is an open subset of $\Sigma$. 
\end{proposition}
 
\subsection{Technical results and proof of Proposition \ref{nonflat case - regular set}}
 Suppose $\mu$ and $\Lambda$ satisfy the assumptions of Theorem \ref{theorem 3}. By Lemma \ref{asod - technical lemma}, we may assume without loss of generality that $\mu$ satisfies \eqref{asod - mu0 has density 1}. Note for later use that if $X_0\in \Sigma\cap K$, where $K\subset\mathbb R^{n+1}$ is compact, and $r>0$ is small enough depending on $K$ and $\Lambda$, then \eqref{asod - mu0 has density 1} implies
 \begin{equation}
      \begin{split}
          \label{nonflat case - measure 1}
          C_K^{-1}r^n\leq \mu(B(X_0,r))\leq C_Kr^n.
      \end{split}
\end{equation}
For example, the upper bound in \eqref{nonflat case - measure 1} can be obtained by noting that $B(X_0,r)\subset B_\Lambda(X_0,\lmin(K)^{-1}r)$, and applying \eqref{asod - mu0 has density 1} to $B_\Lambda(X_0,\lmin(K)^{-1}r)$. The lower bound can be obtained similarly.

As in \cite{PTT08}, we need a smooth version of the $\beta_2$-numbers of $\mu$, which are in turn an $L^2$ version of the $\beta$-numbers considered in Section \ref{beta numbers}. Let $\varphi\in C_c^\infty(\R^{n+1})$ be a radially non-increasing function such that $\chi_{B(0,2)}\leq\varphi\leq\chi_{B(0,3)}$. For $X_0\in\Sigma=\mathrm{spt}(\mu)$ and $B=B(X_0,r)$, let
 \begin{equation}
     \label{nonflat case - smooth beta}
     \tilde\beta_{2,\mu}(B)=\tilde\beta_{2,\mu}(X_0,r)=\min_P\left(\frac{1}{r^{n+2}}\int\varphi\left(\frac{|X-X_0|}{r}\right)\mathrm{dist}(X,P)^2\de\mu(X)\right)^{1/2},
 \end{equation}
 where the minimum is taken over all $n$-planes $P$ in $\mathbb R^{n+1}$. 
Note that if $P$ is a minimizing $n$-plane for $b\beta_\Sigma(X_0,3r)$ and $r$ is small, then by \eqref{nonflat case - measure 1},
 \begin{equation}
     \begin{split}
     \label{nonflat case - beta less than bilateral beta}
     \tilde\beta_{2,\mu}(X_0,r)&\leq\min_{P'}\left(\frac{1}{r^{n+2}}D[\Sigma\cap B(X_0,3r);P'\cap B(x_0,3r)]^2\mu(B(X_0,3r))\right)^{1/2} \\
     &\leq \frac{C_K}{r}D[\Sigma\cap B(X_0,3r);P\cap B(x_0,3r)]=C_Kb\beta_\Sigma(X_0,3r),     \end{split}
 \end{equation}
 for some constant $C_K>0$ depending only on $K$ and $\Lambda$, where $D$ denotes Hausdorff distance as before.
 
 It is also convenient to observe that the coefficients $\tilde\beta_{2,\mu}$ enjoy some regularity: if $X_0,X_0'\in\Sigma\cap K$, $B=B(X_0,r)$, $B'=B(X_0',r')$, $B'\subset B\subset K$, and $r'\geq cr$, then there exists a constant $C_K>0$ depending on $c$, $K$ and $\Lambda$ such that
 \begin{equation}
     \label{asod nonflat case - doubling of beta}
     \tilde\beta_{2,\mu}(B')\leq C_K\tilde\beta_{2,\mu}(B).
 \end{equation}
 
 In the same spirit as Lemma \ref{prelim - lambda and normal reif flat}, we need to establish a comparison between the quantity on the right-hand side of \eqref{nonflat case - smooth beta} and the corresponding quantity obtained when the term $|X-X_0|/r$ is replaced with an anisotropic rescaling determined by $\Lambda$. Recall the numbers $\lmax(K)$ and $\lmin(K)$ associated with any compact set $K$, introduced in \eqref{prelim - lmin and lmax of K}.

 \begin{lemma}
     \label{nonflat case - round and elliptic beta2}Let $r>0$ and suppose $K\subset\mathbb R^{n+1}$  is compact. Denote $r'=\lmax(K)r$ and $r''=\lmin(K)r$, where $\Lambda$ satisfies the assumptions of Theorem \ref{theorem 3}. If $P$ is any $n$-plane in $\mathbb R^{n+1}$ and $\mu$ is a Radon measure in $\R^{n+1}$ with support $\Sigma$, then for every $X_0,Z\in \Sigma\cap K$,
     \begin{equation}
         \label{nonflat case - round to elliptic}
         \int\varphi\left(\frac{|\Lambda(Z)^{-1}(X-X_0)|}{r}\right)\mathrm{dist}(X,P)^2\de\mu(X)\leq \int\varphi\left(\frac{|X-X_0|}{r'}\right)\mathrm{dist}(X,P)^2\de\mu(X),
     \end{equation}

          \begin{equation}
         \label{nonflat case - elliptic to round}
         \int\varphi\left(\frac{|X-X_0|}{r''}\right)\mathrm{dist}(X,P)^2\de\mu(X)\leq \int\varphi\left(\frac{|\Lambda(Z)^{-1}(X-X_0)|}{r}\right)\mathrm{dist}(X,P)^2\de\mu(X).
     \end{equation}
  \end{lemma}

  \begin{remark}
  \label{nonflat case - remark on lambda inverse}
      The statement remains true if $\Lambda(\cdot)$ is replaced with $\Lambda(\cdot)^{-1}$, as long as $r'$ and $r''$ are adjusted accordingly. More specifically, since the smallest and largest eigenvalues of $\Lambda(\cdot)^{-1}$ are $\lmax(\cdot)^{-1}$ and $\lmin(\cdot)^{-1}$, respectively, the lemma applies with $\Lambda(\cdot)^{-1}$ in place of $\Lambda(\cdot)$ if the scales $r'$ and $r''$ are taken to be $r'=\lmin(K)^{-1}r$ and $r''=\lmax(K)^{-1}r$.
  \end{remark}

  \begin{proof}[Proof of Lemma \ref{nonflat case - round and elliptic beta2}]
      With $K$, $X_0$ and $Z$ as in the assumptions, we have for any $r>0$,
      $$\frac{1}{r}|\Lambda(Z)^{-1}(X-X_0)|\geq\frac{1}{r}|\lmax(K)^{-1}(X-X_0)|=\frac{|X-X_0|}{r'},$$
      so \eqref{nonflat case - round to elliptic} follows because $\varphi$ is radially non-increasing. Equation \eqref{nonflat case - elliptic to round} follows for the same reason, by observing that
      $$\frac{1}{r}|\Lambda(Z)^{-1}(X-X_0)|\leq\frac{1}{r}|\lmin(K)^{-1}(X-X_0)|=\frac{|X-X_0|}{r''}.$$
  \end{proof}

The proof of Theorem \ref{theorem 3} relies on the following two results, which are analogues of Theorem 4.2 and Theorem 4.3 in \cite{PTT08}. It is worth noticing that even though we state both results in codimension $1$, the statements remain true in any codimension. Recall the notion of a $\Lambda$-asymptotically optimally doubling measure (see Definition \ref{pseudo tangents - lambda asod}).

\begin{theorem}\label{nonflat case - theorem 1}
    Let $\mu$ be a $\Lambda$-asymptotically optimally doubling measure of dimension $n$ on $\R^{n+1}$ with support $\Sigma$. Let $K\subset\mathbb R^{n+1}$ be a compact set, and suppose that
    \begin{equation}
        \label{nonflat case - adr}
        C_K^{-1}r^n\leq \mu(B(X,r))\leq C_Kr^n,
    \end{equation}
    for $X\in \Sigma\cap K$, $0<r\leq\mathrm{diam}(K)$. For any $\eta>0$, there exists $\delta>0$ depending only on $\eta$, $n$, $\mu$, $K$ and $\Lambda$ such that if $B$ is a ball contained in $K$ and centered at $\Sigma\cap K$ with $\tilde\beta_{2,\mu}(B)\leq\delta$, then $\tilde\beta_{2,\mu}(B')\leq\eta$ for any ball $B'\subset B$ centered at $\Sigma\cap \frac{1}{2}B$.
\end{theorem}

\begin{theorem}\label{nonflat case - theorem 2}
    Let $\mu$ be a $\Lambda$-asymptotically optimally doubling measure of dimension $n$ on $\R^{n+1}$ with support $\Sigma$. Assume that $0\in\Sigma$. Let $K\subset\mathbb R^{n+1}$ be a compact set such that $B(0,2)\subset K$, and suppose that \eqref{nonflat case - adr} holds for $X\in\Sigma\cap K$, $0<r\leq \mathrm{diam}(K)$. Given $\varepsilon>0$, there exists $\delta\in (0,\varepsilon_0)$, depending only on $\varepsilon$, $n$, $\mu$, $K$ and $\Lambda$ such that if $\tilde\beta_{2,\mu}(B)\leq \delta$ for every ball $B$ contained in $B(0,2)$ and centered at $\Sigma\cap K$, then there exists $R>0$ such that $b\beta_\Sigma(X,r)<\varepsilon$ for all $X\in\Sigma\cap B(0,1)$ and $r\in (0,R)$.
\end{theorem}

We will use these results combined in the form of the following corollary.

\begin{corollary}\label{nonflat case - corollary}
    Let $\mu$ be a $\Lambda$-asymptotically optimally doubling measure in $\R^{n+1}$ with support $\Sigma$. Let $K\subset\mathbb R^{n+1}$ be compact, and suppose that \eqref{nonflat case - adr} holds for $X\in\Sigma\cap K$, $0<r\leq\mathrm{diam}(K)$. Given $\varepsilon>0$, there exists $\delta\in (0,\varepsilon_0)$ depending only on $\varepsilon$, $n$, $\mu$, $K$ and $\Lambda$ such that if $\tilde\beta_{2,\mu}(B(X_0,4R_0))\leq\delta$, where $X_0\in\Sigma$ and $B(X_0,4R_0)\subset K$, then there exists $R>0$ such that $b\beta_\Sigma(X,r)<\varepsilon$ for all $X\in\Sigma\cap B(X_0,R_0)$ and $r\in (0,R)$. In particular, $\Sigma\cap B(X_0,R_0)$ is $\varepsilon$-Reifenberg flat.
\end{corollary}

Before proving Theorem \ref{nonflat case - theorem 1} and Theorem \ref{nonflat case - theorem 2}, we use Corollary \ref{nonflat case - corollary} to derive Proposition \ref{nonflat case - regular set}.

\begin{proof}[Proof of Proposition \ref{nonflat case - regular set}]  Let $\mu$ be as in the assumptions of Theorem \ref{theorem 3}. As before, by Lemma \ref{asod - technical lemma} we may assume without loss of generality that $\mu$ satisfies \eqref{asod - mu0 has density 1}, so that \eqref{nonflat case - measure 1} holds. By Proposition \ref{pseudo tangents - density implies asod}, we know that $\mu$ is $\Lambda$-asymptotically optimally doubling. Let $X_0\in\mathcal{R}$ and $\sigma>0$. By definition of $\mathcal{R}$, there exists $R_0>0$ such that $b\beta_\Sigma(X_0,r)\leq \sigma$ whenever $0<r\leq 12R_0$. Let $K=\overline{B(X_0,4R_0)}$. By \eqref{nonflat case - beta less than bilateral beta}, we have 
\begin{equation}
    \label{nonflat case - small beta 2 tilde}
    \tilde\beta_{2,\mu}(B(X_0,4R_0))\leq C_K\sigma,
\end{equation}
where $C_K>0$ depends only on $K$ and $\Lambda$. Let us assume without loss of generality that $R_0$ is small enough so that \eqref{asod - mu0 has density 1} and \eqref{nonflat case - measure 1} hold for every $X\in \Sigma\cap K$ and $r< 8R_0=\mathrm{diam}(K)$, ensuring that the assumptions of Corollary \ref{nonflat case - corollary} are satisfied.

Given any $\varepsilon>0$, let $\delta\in (0,\varepsilon_0)$ be as in the conclusion of Corollary \ref{nonflat case - corollary}. If $\sigma$ is small enough so that $C_K\sigma<\delta$, then by \eqref{nonflat case - small beta 2 tilde} we have $\tilde\beta_{2,\mu}(B(X_0,4R_0))<\delta$, and Corollary \ref{nonflat case - corollary} implies that $\Sigma\cap B(X_0,R_0)$ is $\varepsilon$-Reifenberg flat. We can assume without loss of generality that $\varepsilon<\delta_{K}$, where $K=\overline{B(X_0,R_0)}$ and $\delta_{K}$ is as in Proposition \ref{flat measures - main result}. Then, by Proposition \ref{flat measures - main result},
$$\lim_{r\searrow 0}b\beta_\Sigma(K,r)=0.$$
This and \eqref{asod - mu0 has density 1} ensure that $\mu\mres B(X_0,R_0)$ satisfies the assumptions of Theorem \ref{theorem 1}. Therefore, $\Sigma\cap B(X_0,R_0)$ is a $C^{1,\gamma}$-submanifold of $\mathbb R^{n+1}$ of dimension $n$ for some $\gamma\in (0,1)$ depending on $\alpha$ and $\beta$. This also implies that $\mathcal{R}$ is an open subset of $\Sigma$, as desired.
\end{proof}

  \subsection{Proof of technical results}
We now turn to the proofs of Theorem \ref{nonflat case - theorem 1} and Theorem \ref{nonflat case - theorem 2}. The main ingredient is the following lemma, where for any ball $B=B(X,r)$ and any positive number $c>0$, we denote $r(B)=r$ and $cB=B(X,cr)$.

\begin{lemma}
\label{nonflat case - technical lemma}
    Let $\mu$ be a $\Lambda$-asymptotically optimally doubling measure of dimension $n$ on $\mathbb R^{n+1}$. Let $K\subset\mathbb R^{n+1}$ be compact, and let $\delta_0$ be any positive constant. Suppose that \eqref{nonflat case - adr} holds for $X\in\Sigma\cap K$, $0<r\leq\mathrm{diam}(K)$. Then there exists some constant $\varepsilon_1$ depending on $\varepsilon_0$ and $C_0$, but not on $\delta_0$, and there exists an integer $N>0$ depending only on $\mu$, $K$, $\Lambda$ and $\delta_0$, such that if $B$ is a ball centered at $\Sigma$ with $2^kB\subset K$ and
    \begin{equation}
        \label{nonflat case - small beta}
        \tilde\beta_{2,\mu}(2^kB)\leq\varepsilon_1, \quad k\in\{1,\ldots, N\},
    \end{equation}
    then 
    $$\tilde\beta_{2,\mu}(B)\leq\delta_0.$$ 
\end{lemma}
\begin{proof}[Proof of Lemma \ref{nonflat case - technical lemma}]
    Suppose, for contradiction, that such an $N$ does not exist. Then there is a sequence of points $\{X_j\}\subset\Sigma\cap K$ and balls $B_j=B(X_j,r_j)$ such that $2^jB_j\subset K$ and 
    \begin{equation}
        \label{nonflat case - small beta j}
        \tilde\beta_{2,\mu}(2^kB_j)\leq\varepsilon_1,\quad k\in\{1,\ldots, j\},
    \end{equation}
    but $\tilde\beta_{2,\mu}(B_j)>\delta_0$. Note that since $K$ is bounded and $2^jB_j\subset K$, we have $r_j\to 0$ as $j\to\infty$. For each $j\geq 1$, let 
    $$\mu_j=\frac{1}{\mu(B_\Lambda(X_j,r_j))}T^\Lambda_{X_j,r_j}[\mu].$$ 
    Upon taking a subsequence, we may assume without loss of generality that $\mu_j\rightharpoonup\nu$, where $\nu$ is a $\Lambda$-pseudo tangent of $\mu$, which we know is $n$-uniform by Proposition \ref{pseudo tangents - pseudo tangents are uniform}.
    
    We will show that
    \begin{equation}
        \label{nonflat case - goal 1}\tilde\beta_{2,\nu}(B(0,2^k\lmax(K)^{-1}))\leq C_K\varepsilon_1,\quad k\geq 1,
    \end{equation}
    and
    \begin{equation}
        \label{nonflat case - goal2}
        \tilde\beta_{2,\nu}(B(0,\lmin(K)^{-1}))\geq C_K^{-1}\delta_0.
    \end{equation}
    To prove \eqref{nonflat case - goal 1}, fix $k\geq 1$. Let $L_j^*$ be a minimizing plane for $\tilde\beta_{2,\mu}(2^kB_j)$, and let 
    $$L_j=\frac{1}{r_j}\Lambda(X_j)^{-1}(L_j^*-X_j).$$  
   Upon taking a subsequence, we may assume that $L_j\to L$ with respect to $D$, uniformly on compact sets, where $L$ is an $n$-plane. Note that this implies that $\mathrm{dist}(\cdot,L_j)\to\mathrm{dist}(\cdot,L)$ uniformly on compact subsets of $\mathbb R^{n+1}$. Combining this with the fact that the function $\varphi$ in the definition of $\tilde\beta_{2,\cdot}$ is continuous, $|\varphi|\leq 1$ and $\mu_j\rightharpoonup\nu$, it follows that
    \begin{equation}
        \label{nonflat case - tool 1}
        \begin{split}
            \left|\int\varphi\right.&\left.\left(\frac{|X|}{2^k\lmax(K)^{-1}}\right)\mathrm{dist}(X,L_j)^2\de\mu_j(X)-\int\varphi\left(\frac{|X|}{2^k\lmax(K)^{-1}}\right)\mathrm{dist}(X,L)^2\de\nu(X)\right|\\
            & \\
        &\leq  \mu_j(B(0,3\cdot 2^k\lmax(K)^{-1}))\|\mathrm{dist}(\cdot,L_j)^2-\mathrm{dist}(\cdot,L)^2\|_{L^\infty(B(0,2^k\lmax(K)^{-1}))}\\
        & \\
        & +\left|\int\varphi\left(\frac{|X|}{2^k\lmax(K)^{-1}}\right)\mathrm{dist}(X,L)^2\de\mu_j(X)-\int\varphi\left(\frac{|X|}{2^k\lmax(K)^{-1}}\right)\mathrm{dist}(X,L)^2\de\nu(X)\right|\\
        &\to 0,
        \end{split}
    \end{equation}
    as $j\to\infty$. On the other hand, by Remark \ref{nonflat case - remark on lambda inverse}, an application of Lemma \ref{nonflat case - round and elliptic beta2}, equation \eqref{nonflat case - elliptic to round} with $\Lambda(\cdot)^{-1}$ in place of $\Lambda(\cdot)$ gives

    \begin{equation}
    \label{nonflat case - estimate 1}
        \begin{split}
            \frac{1}{2^{k(n+2)}}\int&\varphi\left(\frac{|X|}{2^k\lmax(K)^{-1}}\right)\mathrm{dist}(X,L_j)^2\de\mu_j(X)\\
            &\leq\frac{1}{2^{k(n+2)}}\int\varphi\left(\frac{|\Lambda(X_j)X|}{2^k}\right)\mathrm{dist}\left(X,L_j\right)^2\de\mu_j(X)\\
            &\leq\frac{C_K}{2^{k(n+2)}\mu(B_\Lambda(X_j,r_j))}\int\varphi\left(\frac{|X-X_j|}{2^kr_j}\right)\mathrm{dist}\left(\frac{\Lambda(X_j)^{-1}(X-X_j)}{r_j},L_j\right)^2\de\mu(X).
        \end{split}
    \end{equation}
    
    By the definition of $L_j$,
    \begin{equation*}
        \begin{split}
            \mathrm{dist}\left(\frac{\Lambda(X_j)^{-1}(X-X_j)}{r_j},L_j\right)&=\mathrm{dist}\left(\frac{\Lambda(X_j)^{-1}(X-X_j)}{r_j},\frac{\Lambda(X_j)^{-1}(L_j^*-X_j)}{r_j}\right)\\
            &\leq \frac{C_K}{r_j}\mathrm{dist}(X,L_j^*).
        \end{split}
    \end{equation*}
     Combining this with \eqref{nonflat case - estimate 1} and \eqref{nonflat case - small beta j} we obtain

     \begin{equation}
     \label{asod nonflt case - helper &}
         \begin{split}
             \frac{1}{2^{k(n+2)}}\int\varphi\left(\frac{|X|}{2^k\lmax(K)^{-1}}\right)&\mathrm{dist}(X,L_j)^2\de\mu_j(X)\\
             &\leq \frac{C_K}{2^{k(n+2)}r_j^{n+2}}\int\varphi\left(\frac{|X-X_j|}{2^kr_j}\right)\mathrm{dist}(X,L_j^*)^2\de\mu(X)\\
             &= C_K\tilde\beta_{2,\mu}(2^kB_j)\leq C_K\varepsilon_1.
         \end{split}
     \end{equation}
This estimate and \eqref{nonflat case - tool 1} with a choice of $j$ large enough give
\begin{equation*}
    \begin{split}
        \frac{1}{(2^k\lmax(K)^{-1})^{n+2}}\int\varphi\left(\frac{|X|}{2^k\lmax(K)^{-1}}\right)\mathrm{dist}(X,L)^2\de\nu(X)&\leq C_K\varepsilon_1,
    \end{split}
\end{equation*}
from which \eqref{nonflat case - goal 1} follows.

To prove \eqref{nonflat case - goal2}, let $L$ be any $n$-plane. Using Lemma \ref{nonflat case - round and elliptic beta2} applied to $\Lambda(\cdot)^{-1}$, along with the definition of $\mu_j$,
\begin{equation*}
    \begin{split}
        \tilde\beta_{2,\nu}(0,\lmin(K)^{-1})&\geq C_K\int\varphi\left(\frac{|X|}{\lmin(K)^{-1}}\right)\mathrm{dist}(X,L)^2\de\mu_j(X)\\
        &\geq C_K\int\varphi(|\Lambda(X_j)X|)\mathrm{dist}(X,L)^2\de\mu_j(X)\\
        &= \frac{C_K}{\mu(B_\Lambda(X_j,r_j))}\int\varphi\left(\frac{|X-X_j|}{r_j}\right)\mathrm{dist}\left(\frac{\Lambda(X_j)^{-1}(X-X_j)}{r_j},L\right)^2\de\mu(X) \\
        & \\
        &\geq\frac{C_K}{r_j^{n+2}}\int\varphi\left(\frac{|X-X_j|}{r_j}\right)\mathrm{dist}(X,X_j+r_j\Lambda(X_j)L)^2\de\mu(X)\\
        &\geq C_K\tilde\beta_{2,\mu}(B_j)>C_K\delta_0,
    \end{split}
\end{equation*}
by our assumption on $\tilde\beta_{2,\mu}(B_j)$. This proves \eqref{nonflat case - goal2}.

We are now ready to complete the proof of the lemma. We claim that $\varepsilon_1$ is small enough, then \eqref{nonflat case - small beta j} implies that the tangent measure $\tilde\nu$ at $\infty$ of $\nu$  satisfies
\begin{equation}
    \label{asod nonflat case - tg at infinity is flat}
    \min_P\int_{B(0,1)}\mathrm{dist}(X,P)^2\de\tilde\nu(X)\leq\varepsilon_0^2,
\end{equation}
where the minimum is taken over all $n$-planes $P\subset\mathbb R^{n+1}$. To show this, notice first that by arguments similar to those leading up to \eqref{asod nonflt case - helper &} and by definition of $\tilde\nu$, we have 
$$\tilde\beta_{2,\tilde\nu}(0,3)\leq C_K\tilde\beta_{2,\nu}(0,2^k\lmax(K)^{-1}),$$for $k$ large. Also by the estimates leading up to \eqref{asod nonflt case - helper &}, we have
$$\tilde\beta_{2,\nu}(0,2^k\lmax(K)^{-1})\leq C_K\tilde\beta_{2,\mu}(2^kB_j)\leq C_K\varepsilon_1.$$
It follows that if $j$ and $k\in \{1,\ldots,j\}$ are large enough, then $\tilde\beta_{2,\tilde\nu}(0,3)\leq C_K\varepsilon_1$, which gives \eqref{asod nonflat case - tg at infinity is flat} by choosing $\varepsilon_1$ small enough depending on $K$, $\Lambda$ and $\varepsilon_0$, and observing that the left hand side of \eqref{asod nonflat case - tg at infinity is flat} is upper bounded by $\tilde\beta_{2,\tilde\nu}(0,3)$.

To conclude, we combine \eqref{asod nonflat case - tg at infinity is flat} with Theorem \ref{nonflat case - theorem of preiss} to deduce that $\nu$ is flat, which contradicts \eqref{nonflat case - goal2}, completing the proof of the lemma.
\end{proof}
With this lemma in hand, we can prove Theorems \ref{nonflat case - theorem 1} and \ref{nonflat case - theorem 2} essentially in the same way as \cite{PTT08}. 

\begin{proof}[Proof of Theorem \ref{nonflat case - theorem 1}]
    Let $\eta>0$, let $\varepsilon_1$ and $N$ be as in Lemma \ref{nonflat case - technical lemma}, and set $\delta_0=\min\{\varepsilon_1,\eta\}$. Let $\delta>0$ be a small number to be determined, and suppose $B$ is a ball of radius $r(B)$ contained in $K$ and centered at $\Sigma\cap K$ with $\tilde\beta_{2,\mu}(B)\leq\delta$. If $\delta$ is small enough depending on $\varepsilon_1$, $\eta$ and $N$, and $B'$ is any ball contained in $B$,  centered at $\Sigma\cap B$, with radius $r(B')\geq 2^{-N-1}r(B)$, then by \eqref{asod nonflat case - doubling of beta},
\begin{equation}
    \label{asod nonflat case - desired small beta}
    \tilde\beta_{2,\mu}(B')\leq\min\{\varepsilon_1,\eta\}.
\end{equation} 

    Let now $B'$ be any ball centered at $\Sigma\cap\frac{1}{2}B$ with $2^{-N-2}r(B)\leq r(B')<2^{-N-1}r(B)$. Then $2^NB'$ is centered at $\Sigma\cap\frac{1}{2}B$ and $r(2^NB')<r(B)/2$, so $2^NB'\subset B$ and $\tilde\beta_{2,\mu}(2^kB')\leq\varepsilon_1$ for every $k\in\{1,\ldots,N\}$. Therefore, we can apply Lemma \ref{nonflat case - technical lemma} to $B'$ and deduce that  $B'$ satisfies \eqref{asod nonflat case - desired small beta}. From this and the arguments above it follows that if $B'$ is any ball centered on $\Sigma\cap \frac{1}{2}B$ with radius $r(B')\geq 2^{-N-2}r(B)$, then
    $B'$ satisfies \eqref{asod nonflat case - desired small beta}. Iterating this procedure, we deduce that for any $j\geq 2$, if $B'$ is a ball centered at $\Sigma\cap\frac{1}{2}B$ with $r(B')\geq 2^{-N-j}$, then $B'$ satisfies \eqref{asod nonflat case - desired small beta}, which completes the proof.

\end{proof}

\begin{proof}[Proof of Theorem \ref{nonflat case - theorem 2}]
    Suppose, for contradiction, that there exists $\varepsilon_1>0$ such that for each $i\geq i_0$ for some $i_0\geq 1$, and for each ball $B\subset B(0,2)$ centered at $\Sigma\cap K$, we have
    $$\tilde\beta_{2,\mu}(B)\leq 2^{-i}\leq\varepsilon_0,$$
    but there are $X_i\in\Sigma\cap B(0,1)$ and $r_i\searrow 0$ such that $b\beta_\Sigma(X_i,r_i)\geq\varepsilon_1$. Write $r_i=\lmin(K)\tau_i$. Fix $i\geq 1$ momentarily, and let $P$ be a minimizing plane for $b\beta_{\Sigma_i}(0,1)$. Write
   $$P=\frac{1}{\tau_i}\Lambda(X_i)^{-1}(\tilde P-X_i),$$ 
   for some $n$-plane $\tilde P$. Consider $\Sigma_i=\frac{1}{\tau_i}\Lambda(X_i)^{-1}(\Sigma-X_i)$,
    as well as the measures 
    $$\mu_i=\frac{1}{\mu(B_\Lambda(X_i,\tau_i))}T^\Lambda_{X_i,\tau_i}[\mu].$$
    Note that $\Sigma_i=\mathrm{spt}(\mu_i)$. By Corollary \ref{prelim - lambda and normal reif flat corollary},
    
    \begin{equation}
    \label{nonflat case - big bilateral beta}
        \begin{split}    
    b\beta_{\Sigma_i}(0,1)&= D[\Sigma_i\cap B(0,1);P\cap B(0,1)]\\
    &\geq \frac{C_K}{\tau_i}D[\Sigma\cap B_\Lambda(X_i,\tau_i);\tilde P\cap B_\Lambda(X_i,\tau_i)]\\
    &\geq \frac{C_K}{r_i}D[\Sigma\cap B(X_i,r_i);\tilde P\cap B(X_i,r_i)]\geq C_Kb\beta_\Sigma(X_i,r_i)\geq \varepsilon_1.
    \end{split}
    \end{equation}
    
Note that this estimate holds for every $i\geq 1$. We also know that upon passing to a subsequence, we have $\mu_i\rightharpoonup\nu$, where $\nu$ is an $n$-uniform $\Lambda$-pseudo tangent of $\mu$, and $\Sigma_i\to\Sigma_\infty=\mathrm{spt}(\nu)$ with respect to $D$, uniformly on compact sets, as before. This, combined with \eqref{nonflat case - big bilateral beta} implies that
\begin{equation}
    \label{nonflat case - big beta of mu infinity}b\beta_{\Sigma_\infty}(0,1)\geq\varepsilon_1/2.
\end{equation}
On the other hand, similarly as in \eqref{nonflat case - tool 1}, we have for every $r>0$, $\tilde\beta_{2,\mu}(0,r)\to\tilde\beta_{2,\nu}(0,r)$. But our initial assumptions imply that for every $r>0$ there exists $i_r\geq 1$ such that if $i\geq i_r$, then $\tilde\beta_{2,\mu_i}(0,r)\leq 2^{-i}$. It then follows that $\tilde\beta_{2,\nu}(0,r)=0$ for every $r>0$, which implies that $\Sigma_\infty$ is contained in an $n$-plane. This contradicts \eqref{nonflat case - big beta of mu infinity} and completes the proof.
\end{proof}

\section{Proof of Theorem \ref{theorem 3}: the singular set}\label{singular set}

Our arguments involve several different measures along with their singular sets, so to avoid confusion we denote the singular set of any Radon measure $\nu$, as defined in the previous section, by $\mathcal{S}_\nu$. We also let $\mathrm{dim}_\mathcal{H}$ denote Hausdorff dimension, and we continue to denote $\Sigma=\mathrm{spt}(\mu)$, with $\mu$ as in the assumptions of Theorem \ref{theorem 3} or Proposition \ref{singular set - main result} below. To complete the proof of Theorem \ref{theorem 3} it is enough to show the following.

\begin{proposition}[The singular set]\label{singular set - main result}
    Suppose $\mu$ is a $\Lambda$-asymptotically optimally doubling measure of dimension $m$ on $\R^{n+1}$, $1\leq m\leq n+1$. Then either $m\leq 2$ and $\mathcal{S}_\mu=\varnothing$, or $m\geq 3$ and
    \begin{equation}
        \label{singular set - main conclusion}\mathrm{dim}_\mathcal{H}(\mathcal{S}_\mu)\leq m-3.
    \end{equation}
\end{proposition}

The Euclidean analogue of this result was proven by Nimer \cite{Ni18} under the assumption that $\mu$ be \textit{uniformly asymptotically doubling}. In our setting, an anisotropic analogue of such notion would go beyond the scope of this paper, so we only consider the particular case in which $\mu$ is $\Lambda$-asymptotically optimally doubling. It is still worth mentioning that a detailed analysis shows that the proofs are very similar in both cases. We also point out that even though we prove Proposition \ref{singular set - main result} in any codimension, Theorem \ref{theorem 3} only relies on its validity in codimension 1.

The dimension bound of Proposition \ref{singular set - main result} can be motivated by the case where $\mu$ is $n$-uniform in $\R^{n+1}$. In that scenario, it follows from Theorem \ref{flat measures - kowalski preiss} that $\mathcal{S}_\mu$  is either empty or, up to a rotation and translation, an $(n-3)$-dimensional plane. More generally, Nimer showed in \cite{Ni17} that in any codimension, the dimension of the singular set of an $m$-uniform measure $\mu$ is at most $m-3$. Since the measure $\mu$ in the assumptions of Theorem \ref{theorem 3} behaves asymptotically as $r\to 0$ as an $n$-uniform measure in a certain sense, it is natural to expect that the same dimension bound should be preserved.

\subsection{Technical results and proof of Proposition \ref{singular set - main result}}

Our arguments rely crucially on a dimension reduction result, and a statement about the singular set of conical $3$-uniform measures in arbitrary codimension, both proven by Nimer.

\begin{lemma}[Dimension reduction \cite{Ni18}]\label{singular set - dimension reduction}
    Let $\nu$ be an $m$-uniform, conical measure on $\R^{n+1}$, $1\leq m\leq n+1$. Let $X\in\mathrm{spt}(\nu)$, $X\neq 0$, and let $\lambda$ be the tangent measure of $\nu$ at $X$ (normalized so that $\lambda(B(0,1))=1$). Then, up to rotation,
    $$\lambda=c\mathcal{H}^m\mres(\mathbb R\times A),$$
    where $c>0$ is a constant and $A\subset\mathbb R^{n}$ is the support of an $(m-1)$-uniform measure.
\end{lemma}

\begin{theorem}[Regularity of conical $3$-uniform measures \cite{Ni17}]
\label{singular set - singular set of conical 3-uniform measures}
    Let $\nu$ be a conical $3$-uniform measure on $\R^{n+1}$. Then there exists $\gamma>0$ such that $\mathrm{spt}(\nu)\backslash\{0\}$ is a $C^{1,\gamma}$ submanifold of dimension $3$ of $\R^{n+1}$.
\end{theorem}

Recall, for $X\in\Sigma$ and $r>0$, the mappings $T_{X,r}$ and $T^{\Lambda}_{X,r}$ defined in \eqref{nonflat case - round blow up map} and \eqref{pseudo tangents - blow up map}, as well as the corresponding re-scalings of a Radon measure $\nu$,
$$\nu_{X,r}=\frac{1}{\nu(B(X,r))}T_{X,r}[\nu],\quad \nu^\Lambda_{X,r}=\frac{1}{\nu(B_\Lambda(X,r))}T^\Lambda_{X,r}[\nu].$$

We will need the following technical results, whose Euclidean analogues are in \cite{Ni18}.

\begin{lemma}[Connectedness property]\label{singular set - connectedness property}
    Suppose $\mu$ is a $\Lambda$-asymptotically optimally doubling measure of dimension $m$ in $\R^{n+1}$, $1\leq m\leq n+1$. Let $X_k\in\mathrm{spt(\mu)\cap \overline{B(0,1)} }$ be such that $X_k\to X$ as $k\to \infty$. Let also $\tau_k>\sigma_k>0$ be such that $\tau_k,\sigma_k\to 0$ and
    $$\mu_{X_k,\tau_k}^{\Lambda}\rightharpoonup \alpha,\quad \mu_{X_k,\sigma_k}^{\Lambda}\rightharpoonup \beta,$$
    as $k\to\infty$, where $\alpha$ and $\beta$ are non-zero Radon measures. If $\alpha$ is flat, then $\beta$ is flat.
\end{lemma}

\begin{lemma}[$\Lambda$-tangents at a singular point]
\label{singular set - lambda tangents at singular points}
    Suppose $\mu$ is a $\Lambda$-asymptotically optimally doubling measure of dimension $m$ in $\R^{n+1}$, $1\leq m \leq n+1$. Then:
    \begin{enumerate}
        \item If $X\in\mathcal{S}_\mu$, then $\mu$ has no flat $\Lambda$-tangent measures at $X$.
        \item If $\mu$ has a $\Lambda$-tangent measure at $X$ which is not flat, then $X\in\mathcal{S}_\mu$.
    \end{enumerate}
    \end{lemma}

One of the key insights of Nimer \cite{Ni18} in the Euclidean setting is that if a measure is regular enough, then its singularities are ``preserved'' under blow-ups. The following is the corresponding anisotropic version of this statement.

\begin{proposition}[Conservation of singularities]\label{singular set - conservation of singularities}
    Suppose $\mu$ is a $\Lambda$-asymptotically optimally doubling measure of dimension $m$ on $\R^{n+1}$, $1\leq m\leq n+1$. Let $X\in\mathrm{spt}(\mu)$ and $r_k>0$ be such that $r_k\to 0$ and $\mu^{\Lambda}_{X,r_k}\rightharpoonup \nu$ as $k\to\infty$, where $\nu$ is a nonzero Radon measure. If $X_k\in\mathcal{S}_\mu$ and $Y_k=\Lambda(X )^{-1}\left(\frac{X_k-X}{r_k}\right)\to Y$ for some $Y\in \R^{n+1}$ as $k\to\infty$, then $Y\in\mathcal{S}_\nu$.
\end{proposition}
An important corollary of the Euclidean version of this proposition is the following result.

\begin{corollary}[Singularities of $3$-uniform measures \cite{Ni18}]
\label{singular set - corollary}
    Let $\nu$ be a $3$-uniform measure in $\R^{n+1}$. Then for every compact set $K\subset\R^{n+1}$, $\mathcal{S}_\nu\cap K$ is finite. In particular, $\mathrm{dim}_\mathcal{H}(\mathcal{S}_\nu)=0$. 
\end{corollary}

Before we proceed with the proofs of these results we use them to prove Proposition \ref{singular set - main result}.

\begin{proof}[Proof of Proposition \ref{singular set - main result}] First consider the case $m\leq 2$. Suppose that there exists $X\in\mathcal{S}_\mu$. By Lemma \ref{singular set - lambda tangents at singular points}, $\mu$ has a $\Lambda$-tangent measure $\nu$ at $X$ that is not flat. By Lemma \ref{pseudo tangents - pseudo tangents are uniform}, we know that $\nu$ is $2$-uniform, so by Theorem \ref{prelim - uniform measures of dimension 2} we obtain a contradiction. This implies that $\mathcal{S}_\mu=\varnothing$.

Next we handle the case $m\geq 3$ by induction. For the base case, consider $m=3$ and suppose, for contradiction, that there exists $s>0$ such that $\mathcal{H}^s(\mathcal{S}_\mu)>0$. We first find a point $X\in\mathcal{S}_\mu$ and a $\Lambda$-tangent measure $\nu$ of $\mu$ at $X$, such that
\begin{equation}
    \label{singular set - conclusion 1}
    \mathcal{H}^s
(\mathcal{S}_\nu\cap \overline{B(0,1)})>0.\end{equation}

The assumption that $\mathcal{H}^s(\mathcal{S}_\mu)>0$ implies that $\mathcal{H}^s_\infty(\mathcal{S}_\mu)>0$ (by the arguments in the proof of Lemma 2.1 in \cite{EG15}, for example). By sub-additivity of $\mathcal{H}^s_\infty$, there exists a compact set $K\subset\R^{n+1}$ such that $\mathcal{H}^s_\infty(\mathcal{S}_\mu\cap K)>0$. Let $\mathcal{S}_\mu(K)=\mathcal{S}_\mu\cap K$. Then, according to the proof of Theorem 2.7 in \cite{EG15}, for $\mathcal{H}^s$-almost every $X\in \mathcal{S}_\mu(K)$
we have 
$$\theta^{s,*}(\mathcal{H}^s_\infty\mres\mathcal{S}_\mu(K),X)\geq 2^{-s}.$$ Fix any such $X$. Then there exists a sequence of radii $r_k>0$ such that $r_k\to 0$, and $\mathcal{S}_k:=\Lambda(X)^{-1}\left(\frac{\mathcal{S}_\mu(K)-X}{r_k}\right)$ satisfies
\begin{equation}
    \label{singular set - condition 7}
    \mathcal{H}^s_\infty\left(\mathcal{S}_k\cap \overline{B(0,1)}\right)\geq C_K^{-1}2^{-s},
\end{equation}
    for all $k$, where $C_K>0$ depends only on $K$ and $\Lambda$. In particular, $X$ is a limit point of $\mathcal{S}_\mu$, and since $\mathcal{S}_\mu$ is closed by Proposition \eqref{nonflat case - regular set}, it follows that $X\in\mathcal{S}_\mu(K)$, or $0\in\mathcal{S}_k$ for all $k$. Thus, upon passing to a subsequence, we may find a closed set $F\subset\R^{n+1}$ such that 
    \begin{equation}
        \label{singular set - convergence of singular set}
        \mathcal{S}_k\to F,
    \end{equation} with respect to Hausdorff distance, uniformly on compact sets.
    
    Upon passing to a further subsequence, we may also assume $\mu^\Lambda_{X,r_k}\rightharpoonup\nu$, where $\nu$ is $m$-uniform with $\nu(B(0,1))=1$ by Proposition \ref{pseudo tangents - pseudo tangents are uniform}.
    We claim that $F\subset\mathcal{S}_\nu$. In fact, if $Y\in F$, we can find $Y_k\in\mathcal{S}_k$ such that $Y_k\to Y$. By definition of $\mathcal{S}_k$, we have $Y_k=\Lambda(X)^{-1}\left(\frac{X_k-X}{r_k}\right)$ for some $X_k\in\mathcal{S}_\mu$, so Proposition \ref{singular set - conservation of singularities} implies that $Y\in\mathcal{S}_\nu$, proving the claim. This, combined with \eqref{singular set - convergence of singular set}, implies that for any $\varepsilon>0$ and $k$ large enough,

    \begin{equation}
        \label{singular set - condition 4}
        \mathcal{S}_k\cap \overline{B(0,1)}\subset (\mathcal{S}_\nu;\varepsilon)\cap \overline{B(0,1)}\subset (\mathcal{S}_\nu\cap\overline{B(0,1)};\varepsilon).
    \end{equation}

    We use this information to estimate $\mathcal{H}_\infty^s(\mathcal{S}_\nu\cap\overline{B(0,1)})$. Let $\delta>0$, $\mathcal{S}_\nu(B)=\mathcal{S}_\nu\cap\overline{B(0,1)}$, and let $\{E_\ell\}$ be a cover of $\mathcal{S}_\nu(B)$ such that
    $$\mathcal{H}^s_{\infty}(\mathcal{S}_\nu(B))>\sum_\ell\omega_s\left(\frac{\mathrm{diam}(E_\ell)}{2}\right)^s-\delta.$$
    By incorporating small open neighborhoods of the sets $E_\ell$ if necessary we may assume that each $E_\ell$ is open. By Proposition \ref{nonflat case - regular set}, $\mathcal{S}_\nu$ is closed, so $\mathcal{S}_\nu(B)$ is compact, and we may assume that $\mathcal{S}_\nu(B)\subset E_1\cup\cdots\cup E_L$ for some $L<\infty$. Therefore, by \eqref{singular set - condition 4}, if $k$ is large enough we have $\mathcal{S}_k(B):=\mathcal{S}_k\cap \overline{B(0,1)}\subset E_1\cup\cdots\cup E_L$, which implies
    $$\mathcal{H}^s_\infty(\mathcal{S}_k(B))\leq\omega_s\sum_{\ell=1}^L\left(\frac{\mathrm{diam}(E_\ell)}{2}\right)^s<\mathcal{H}^s_\infty(\mathcal{S}_\nu(B))+\delta.$$
    With this and \eqref{singular set - condition 7} we obtain
    $$\mathcal{H}^s_\infty(\mathcal{S}_\nu(B))\geq\limsup_{k\to\infty}\mathcal{H}^s_\infty(\mathcal{S}_k(B))>0,$$
    which implies \eqref{singular set - conclusion 1}. However, we know that $\nu$ is $3$-uniform, so by Corollary \ref{singular set - corollary}, $\mathcal{S}_\nu$ is at most countable, which contradicts \eqref{singular set - conclusion 1}. This implies that $s=0$, and completes the proof of the base case.

    For the inductive step, assume that $m\geq 4$, and suppose that the theorem holds for measures of dimension $m-1$. Let $s>0$ be such that $\mathcal{H}^s(\mathcal{S}_\mu)>0$. By the procedure described above, we can find a point $X\in\mathcal{S}_\mu$ and a $\Lambda$-tangent measure $\nu$ of $\mu$ at $X$, such that \eqref{singular set - conclusion 1} holds.
    Now, by \eqref{singular set - conclusion 1}, we can apply the same procedure again with $\Lambda(X)^{-1}$ replaced by $\mathrm{Id}$, to obtain a tangent measure $\nu'$ of $\nu$ at a point $ X'\in\mathcal{S}_\nu\cap\overline{B(0,1)}$, such that $\mathcal{H}^s(\mathcal{S}_{\nu'}\cap \overline{B(0,1)})>0$. We know that $\nu$ is $m$-uniform, so by Theorem \ref{nonflat case - theorem of preiss}, $\nu'$ is $m$-uniform and conical, and in particular it satisfies the assumptions of Lemma \ref{singular set - dimension reduction}. Moreover, since $\mathcal{H}^s(\mathcal{S}_{\nu'}\cap \overline{B(0,1)})>0$ and $s>0$, we have that for some $\rho>0$ small enough, $\mathcal{H}^s(\mathcal{S}_{\nu'}\cap\overline{B(0,1)}\backslash B(0,\rho))>0$. Thus, we can apply the above procedure once more to obtain a point $X''\in\mathcal{S}_{\nu'}$, $X''\neq 0$, and a tangent measure $\nu''$ of $\nu'$ at $X''$ such that
    \begin{equation}
        \label{singular set - condition 9}
        \mathcal{H}^s(\mathcal{S}_{\nu''}\cap \overline{B(0,1)})>0.
    \end{equation}

    By Lemma \ref{singular set - dimension reduction}, upon a rotation, $\Sigma'':=\mathrm{spt}(\nu'')$ satisfies $\Sigma''=\mathbb R\times A$, where $A\subset\R^{n}$ is the support of an $(m-1)$-uniform measure $\nu_0''$. In particular, we have $\mathcal{S}_{\nu''}\subset\R\times\mathcal{S}_{\nu''_0}$,
    which implies
    \begin{equation}
    \label{singular set - condition 8}\mathrm{dim}_\mathcal{H}(\mathcal{S}_{\nu''})\leq\mathrm{dim}_\mathcal{H}(\mathcal{S}_{\nu''_0})+1.
    \end{equation}
    Since $\nu_0''$ is $(m-1)$-uniform, our induction hypothesis implies that $\mathrm{dim}_\mathcal{H}(\mathcal{S}_{\nu''_0})\leq m-4,$ so \eqref{singular set - condition 8} gives $\mathrm{dim}_\mathcal{H}(\mathcal{S}_{\nu''})\leq m-3$. Combining this with \eqref{singular set - condition 9} we obtain 
    $$s\leq\mathrm{dim}_\mathcal{H}(\mathcal{S}_{\nu''})\leq m-3.$$ Since $\mathcal{H}^s(\mathcal{S}_\mu)>0$ by assumption, it then follows that $\mathrm{dim}_\mathcal{H}(\mathcal{S}_\mu)\leq m-3$, completing the inductive step and the proof of Proposition \ref{singular set - main result}.
\end{proof}

\subsection{Proof of technical results}

Now we turn to the proofs of Lemma \ref{singular set - connectedness property}, Lemma \ref{singular set - lambda tangents at singular points} and Proposition \ref{singular set - conservation of singularities}. Some of the arguments below rely on the following fact.

\begin{lemma}[Scaling before and after blow-up]
\label{singular set - scaling and blow up}
    If $\nu,\nu'$ are nonzero Radon measures on $\R^{n+1}$, $\nu$ is $\Lambda$-asymptotically optimally doubling of dimension $m$, $1\leq m\leq n+1$, $X_k,X\in\mathrm{spt}(\nu)$, $R,R_k,r_k>0$, $R_k\to R$, $r_k\to 0$ and $\nu^\Lambda_{X_k,r_k}\rightharpoonup\nu'$ as $k\to\infty$, then
\begin{equation}
    \label{singular set - scaling pre and post blow up}\nu^\Lambda_{X_k,R_kr_k}\rightharpoonup\nu'_{0,R}.
\end{equation}

\end{lemma}

\begin{proof}
    Let $\psi\in C_c(\R^{n+1})$. It suffices to show that
    \begin{equation}
        \label{singular set - statement 0}
        \int\psi\de\nu^\Lambda_{X_k,R_kr_k}=\frac{1}{\nu^\Lambda_{X_k,r_k}(B(0,R_k))}\int\psi\de(T_{0,R_k}[\nu^\Lambda_{X_k,r_k}])\to\frac{1}{\nu'(B(0,R))}\int\psi\circ T_{0,R}\de\nu'.
    \end{equation}
    
    Note that by the doubling property of $\nu$, $\nu'$ is $m$-uniform, so for any $\rho>0$, $\nu'(\partial B(0,\rho))=0$, which implies $\nu^\Lambda_{X_k,r_k}(B(0,\rho))\to\nu'(B(0,\rho))$. We first show that
    \begin{equation}
        \label{singular set - statement 1}\nu^{\Lambda}_{X_k,r_k}(B(0,R_k))\to\nu'(B(0,R)).
    \end{equation}
    We write
    \begin{equation}
    \label{singular set - statement 1.1}
        \begin{split}
            |\nu^\Lambda_{X_k,r_k}(B(0,R_k))&-\nu'(B(0,R))|\\
            &\leq |\nu^\Lambda_{X_k,r_k}(B(0,R_k))-\nu^\Lambda_{X_k,r_k}(B(0,R))|+|\nu^\Lambda_{X_k,r_k}(B(0,R))-\nu'(B(0,R))|.
        \end{split}
    \end{equation}
    The last term goes to $0$ by the above observation. For the first term, let $\varepsilon\in (0,R)$ and suppose that $k$ is large enough so that $|R_k-R|<\varepsilon$. Then
    \begin{equation}
        \begin{split}
            |\nu^\Lambda_{X_k,r_k}(B(0,R_k))-\nu^\Lambda_{X_k,r_k}(B(0,R))|&\leq \nu^\Lambda_{X_k,r_k}(B(0,R+\varepsilon)\backslash B(0,R-\varepsilon))\\
            &\to \nu'(B(0,R+\varepsilon)\backslash B(0,R-\varepsilon))\leq C\varepsilon R^{m-1},
        \end{split}
    \end{equation}
    as $k\to\infty$. Since this holds for every $\varepsilon\in (0,R)$, \eqref{singular set - statement 1} now follows from \eqref{singular set - statement 1.1}. 
    Next we show that
    \begin{equation}
        \label{singular set - statement 2}
        \int\psi\de(T_{0,R_k}[\nu^\Lambda_{X_k,r_k}])=\int\psi\circ T_{0,R_k}\de\nu^\Lambda_{X_k,r_k}\to\int\psi\circ T_{0,R}\de\nu'.
    \end{equation}
    Note that $\psi\circ T_{0,R_k}\to\psi\circ T_{0,R}$ uniformly. Let $M>0$ be such that $\mathrm{spt}(\psi\circ T_{0,R_k})\subset B(0,M)$ for all $k$. Since $\nu^\Lambda_{X_k,r_k}$ converges weakly, the sequence $\nu^\Lambda_{X_k,r_k}(\overline{B(0,M)})$ is bounded, so
    \begin{equation}
        \begin{split}
            \label{singular set - statement 3}\left|\int\psi\circ T_{0,R_k}\de\nu^\Lambda_{X_k,r_k}-\int\psi\circ T_{0,R}\de\nu^\Lambda_{X_k,r_k}\right|&\leq C\|\psi\circ T_{0,R_k}-\psi\circ T_{0,R}\|_\infty\to 0,
        \end{split}
    \end{equation}
    as $k\to\infty$. On the other hand, since $\nu^\Lambda_{X_k,r_k}\rightharpoonup\nu'$ we have
    $$\int\psi\circ T_{0,R}\de\nu^\Lambda_{X_k,r_k}\to \int\psi\circ T_{0,R}\de\nu'.$$
    Combining this with \eqref{singular set - statement 3} we obtain \eqref{singular set - statement 2}, and from \eqref{singular set - statement 1} and \eqref{singular set - statement 2} we obtain \eqref{singular set - statement 0}.

\end{proof}

\begin{proof}[Proof of Lemma \ref{singular set - connectedness property}]Recall the functional $F$ defined in \eqref{singular set - functional}.  Since $\mu$ is $\Lambda$-asymptotically optimally doubling of dimension $m$, we know that $\alpha$ and $\beta$ are $m$-uniform. Suppose $\alpha$ is flat, but $\beta$ is not. Then $F(\alpha)=0$, but by Corollary \ref{singular set - Preiss reformulation}, for every $R>0$ we can find  $R_0>R$ such that $F(\beta_{0,R_0})>\varepsilon^2_0$, with $\varepsilon_0>0$ as in \eqref{singular set - Preiss reformulation}. In particular, since $\mu_{X_k,\tau_k}^{\Lambda}\rightharpoonup \alpha$, we have $\mu_{X_k,R_0\sigma_k}^{\Lambda}\rightharpoonup \beta_{0,R_0}$ by Lemma \ref{singular set - scaling and blow up}, so by continuity of $F$ with respect to weak convergence, it follows that for some $\kappa\in (0,\varepsilon_0^2)$ and for all $k$ large enough,
    \begin{equation}
        \label{singular set - kappa inequality}F(\mu_{X_k,\tau_k}^{\Lambda})<\kappa< F(\mu_{X_k,R_0\sigma_k}^{\Lambda}).
    \end{equation}
    
    We will first show that 
    \begin{equation}
        \label{singular set - ratio goes to infinity}
        \tau_k/\sigma_k\to\infty,
    \end{equation}
    as $k\to\infty$. Assume, by contradiction, that this is not the case. Since $\tau_k>\sigma_k$ by assumption, upon passing to a subsequence we have $\gamma_k:=\tau_k/\sigma_k\to\gamma$, for some $\gamma\geq 1$. This implies
    \begin{equation}
        \label{singular set - beta blow up 1}
        \begin{split}
            \frac{1}{\mu(B_\Lambda(X_k,\sigma_k))}T^\Lambda_{X_k,\tau_k}[\mu]&=\frac{1}{\mu(B_\Lambda(X_k,\sigma_k))}T^\Lambda_{X_k,\gamma_k\sigma_k}[\mu]\rightharpoonup T_{0,\gamma}[\beta].
        \end{split}
    \end{equation}
    On the other hand, since $\tau_k/\sigma_k\to\gamma$, \eqref{pseudo tangents - asymptotically optimally doubling} implies that $\mu(B(X_k,\tau_k))/\mu(B(X_k,\sigma_k))\to\gamma^n$
    as $k\to\infty$. Therefore,
    \begin{equation*}
        \label{}
        \begin{split}
            \frac{1}{\mu(B_\Lambda(X_k,\sigma_k))}T^\Lambda_{X_k,\tau_k}[\mu] = \frac{\mu(B_\Lambda(X_k,\tau_k))}{\mu(B_\Lambda(X_k,\sigma_k))}\frac{T^\Lambda_{X_k,\tau_k}[\mu]}{\mu(B_\Lambda(X_k,\tau_k))}\rightharpoonup \gamma^n\alpha,
        \end{split}
    \end{equation*}
    which contradicts \eqref{singular set - beta blow up 1} because $\gamma^n\alpha$ is flat but $T_{0,\gamma}[\beta]$ is not. This proves \eqref{singular set - ratio goes to infinity}. In particular, if $k$ is large enough, we have $R_0\sigma_k<\tau_k$.

    Note that for each $k$, the function $r\mapsto F(\mu^\Lambda_{X_k,r})$ is continuous away from $r=0$, by continuity of $F$. Therefore, \eqref{singular set - kappa inequality} implies that there exist $\delta_k\in [R_0\sigma_k,\tau_k]$ such that $F(\mu^\Lambda_{X_k,\delta_k})=\kappa$ and 
    \begin{equation}
        \label{singular set - kappa ineq 1}
        F(\mu^\Lambda_{X_k,r})\leq \kappa,
    \end{equation} for every $r\in [\delta_k,\tau_k]$ (take $\delta_k=\max\{r\in [R_0\sigma_k,\tau_k]: F(\mu^\Lambda_{X_k,r})\geq \kappa\}$). By the doubling property of $\mu$, we may assume upon passing to a subsequence that $\mu_{X_k,\delta_k}^{\Lambda}\rightharpoonup\nu$, where $\nu$ is an $m$-uniform $\Lambda$-pseudo tangent measure of $\mu$. Note that by continuity of $F$, we have
    \begin{equation}
        \label{singular set - F of kappa}F(\nu)=\kappa.
    \end{equation}
    
    We will now show that $\nu$ is flat, which will contradict \eqref{singular set - F of kappa} and complete the proof. Note that since $F(\mu^\Lambda_{X_k,\delta_k})=\kappa>0$ for all $k$ large, the same proof that $\tau_k/\sigma_k\to\infty$ can be applied to show that $\tau_k/\delta_k\to\infty$. Let $R>1$, and suppose $k$ is large enough so that $R\delta_k<\tau_k$. By \eqref{singular set - kappa ineq 1} we have $F(\mu^\Lambda_{X_k,R\delta_k})\leq\kappa$ for all $k$ large. But we also have $\mu^\Lambda_{X_k,R\delta_k}\rightharpoonup T_{0,R}[\nu]$. Therefore, by continuity of $F$, it follows that $F(\nu_{0,R})\leq \kappa$, and letting $R\to\infty$ we get
    $$\limsup_{R\to\infty}F(\nu_{0,R})\leq \kappa<\varepsilon_0^2.$$
    By Corollary \ref{singular set - Preiss reformulation}, this implies that $\nu$ is flat, as desired.
\end{proof}

    \begin{proof}[Proof of Lemma \ref{singular set - lambda tangents at singular points}]
        Let us denote $\Sigma=\mathrm{spt}(\mu)$. For the proof of 1, applying Lemma \ref{singular set - connectedness property} with $X_k=X$ for all $k$, we see that it suffices to show that $\mu$ has at least one $\Lambda$-tangent measure at $X$ that is not flat.
        Since $X\in\mathcal{S}_\mu$, there exists a sequence $r_k>0$, $k\in\mathbb N$, with $r_k\searrow 0$ as $k\to\infty$, such that
\begin{equation}
    \label{nonflat case - big beta}
    b\beta_\Sigma(X,r_k)\geq c,
\end{equation}
for all $k$. Let $\Sigma_k=T^\Lambda_{X,r_k}(\Sigma)$. Then by \eqref{nonflat case - big beta} and an argument as in Step 2 of Section \ref{beta numbers}, we have for $\rho_0=\lmin(X)^{-1}$,
\begin{equation}
    \label{nonflat case - big beta sigma k}
    b\beta_{\Sigma_k}(0,\rho_0)\geq c_1 b\beta_\Sigma(X,r_k)\geq c_2>0,
\end{equation}
where $c_1$ and $c_2$ depend on $X$. Since $0\in \Sigma_k$ for all $k$, we have as in Section \ref{flatness} that upon passing to a subsequence, there exists a closed set $\Sigma_\infty\subset\R^{n+1}$ such that $0\in\Sigma_\infty$ and $\Sigma_k\to\Sigma_\infty$ as $k\to\infty$ with respect to Hausdorff distance, uniformly on compact sets. Note that \eqref{nonflat case - big beta sigma k} implies that
\begin{equation}
    \label{nonflat case - big beta sigma infinity}
    b\beta_{\Sigma_\infty}(0,\rho_0)\geq c_2/2.
\end{equation}

Let now $\mu_k=\mu^\Lambda_{X,r_k}$. Since $\mu$ is $\Lambda$-asymptotically optimally doubling, we may assume by Lemma \ref{pseudo tangents - existence of pseudo tangents} that upon passing to a further subsequence, we have $\mu_k\rightharpoonup\nu$, where $\nu$ is a $\Lambda$-tangent measure of $\mu$. Moreover, since $\Sigma_k\to\Sigma_\infty$ with respect to Hausdorff distance, Lemma \ref{pseudo tangents - characterization of support} implies that $\mathrm{spt}(\nu)=\Sigma_\infty$. But \eqref{nonflat case - big beta sigma infinity} implies that $\Sigma_\infty$ cannot be a plane, so $\nu$ is not flat, as desired.

For the proof of 2, suppose $\mu^\Lambda_{X,r_k}\rightharpoonup\nu$, where $\nu$ is not flat. Then by a similar argument as above we may assume upon passing to a subsequence, that $\Sigma_k=\mu^\Lambda_{X,r_k}$ satisfies $\Sigma_k\to\Sigma_\infty$ for some $\Sigma_\infty\subset\mathbb R^{n+1}$ with respect to Hausdorff distance, and $\Sigma_\infty=\mathrm{spt}(\nu)$. Since $\nu$ is not flat, we have $b\beta_{\Sigma_\infty}(0,1)>0$, so for $k$ large enough and $\rho_1=\lmax(X)^{-1}$,
$$b\beta_{\Sigma}(X,\rho_1r_k)\geq c_2b\beta_{\Sigma_k}(0,1)\geq c_2b\beta_{\Sigma_\infty}(0,1)>0,$$
which implies that $X\in\mathcal{S}_\mu$.
    \end{proof}

\begin{proof}[Proof of Proposition \ref{singular set - conservation of singularities}]  Note first that since $Y_k\to Y$, we have
$$|X_k-X|\leq \lmax(X)r_k|Y_k|\to 0,$$
as $k\to\infty$, so $X_k\to X$. We start by constructing a sequence of radii $\sigma_k>0$ such that $\sigma_k\to 0$, $r_k/\sigma_k\to \infty$, and $\mu^\Lambda_{X_k,\sigma_k}\rightharpoonup \nu^{(\infty)}$ as $k\to\infty$, where $\nu^{(\infty)}$ satisfies 
    \begin{equation}
        \label{singular set - tilde nu big F}
        \limsup_{R\to\infty}F(\nu^{(\infty)}_{0,R})>\varepsilon_0^2/2.
    \end{equation} Since $\mu$ is $\Lambda$-asymptotically optimally doubling, we can construct inductively a family of sequences $\{s_j^{(k)}\}_{i\in\mathbb N}$, where $\{s_j^{(1)}\}$ is a subsequence of $\{r_j\}$, $\{s_j^{(k+1)}\}$ is a subsequence of $\{s_j^{(k)}\}$ for $k\geq 1$, and for every $k$,
    $$\mu^\Lambda_{X_k,s_j^{(k)}}\rightharpoonup \nu^{(k)},$$
    as $j\to\infty$, where $\nu^{(k)}$ is an $m$-uniform $\Lambda$-tangent measure of $\mu$ at $X_k$ with $\nu^{(k)}(B(0,1))=1$ by Proposition \ref{pseudo tangents - pseudo tangents are uniform}. For each $k$, note that $s^{(k)}_j\to 0$ as $j\to\infty$, and since $X_k\in\mathcal{S}_\mu$, Lemma \ref{singular set - lambda tangents at singular points} implies that $\nu^{(k)}$ is not flat. Therefore, by Corollary \ref{singular set - Preiss reformulation}, for each $R>0$ we can find $R'>R$, depending on $k$, such that 
    \begin{equation}
        \label{singular set - limsup is big}
        F(\nu^{(k)}_{0,R'})>\varepsilon_0^2.
    \end{equation}
    Let $R_k$ be any such $R'$. Note that by uniformity of $\nu^{(k)}$, for every $r>0$ we have
    $$\nu^{(k)}_{0,R_k}(B(0,r))=\frac{\nu^{(k)}(B(0,rR_k))}{\nu^{(k)}(B(0,R_k))}=r^m,$$
    so upon passing to a subsequence we may assume that $\nu^{(k)}_{0,R_k}\rightharpoonup\nu^{(\infty)}$ as $k\to\infty$, for some Radon measure $\nu^{(\infty)}$. We will show that $\nu^{(\infty)}$ satisfies \eqref{singular set - tilde nu big F}. Recall the functionals $\mathcal{F}_r$ and $\mathcal{F}$ introduced in \eqref{prelim - distance between measures}.
    Note first that $\mu^\Lambda_{X_k,R_ks^{(k)}_j}\rightharpoonup\nu^{(k)}_{0,R_k}$ as $j\to\infty$, so by Proposition \ref{singular set - metric for weak convergence} and because $s^{(k)}_j\to 0$ as $j\to\infty$, we can find $j=j_k$ large enough so that 
    \begin{equation}
    \label{singular set - bound 1}
        \mathcal{F}(\mu^\Lambda_{X_k,R_ks^{(k)}_{j_k}},\nu^{(k)}_{0,R_k})<2^{-k},\quad \sigma_k:=R_ks^{(k)}_{j_k}\leq r_k^2.
    \end{equation}
     Since $\nu^{(k)}_{0,R_k}\rightharpoonup\nu^{(\infty)}$, \eqref{singular set - bound 1} implies that $\mu^\Lambda_{X_k,\sigma_k}\rightharpoonup\nu^{(\infty)}$, so $\nu^{(\infty)}$ is a $\Lambda$-pseudo tangent of $\mu$ at $X$. Note that $\frac{r_k}{\sigma_k}\geq\frac{1}{r_k}\to\infty$
    as $k\to\infty$.

    We now show that \eqref{singular set - tilde nu big F} holds. Suppose, for contradiction, that $F(\nu^{(\infty)}_{0,R})\leq\varepsilon_0^2/2$ for all $R>1$ large enough. For each such $R$, let $P$ be a minimizing plane in the definition of $F(\nu^{(\infty)}_{0,R})$. Then with $\varphi$ as in the definition of $F$, and keeping in mind that $\nu^{(k)}_{0,RR_k}(B(0,1))=1$ by Proposition \ref{pseudo tangents - pseudo tangents are uniform},
    \begin{equation}
        \label{singular set - bound 2}
        \begin{split}
            F(\nu^{(k)}_{0,RR_k})&\leq\int\varphi(Z)\mathrm{dist}(Z,P)^2\de\nu^{(k)}_{0,RR_k}(Z)\\
            &=\int\varphi(Z)\mathrm{dist}(Z,P)^2\de\nu^{(k)}_{0,RR_k}(Z)-\int\varphi(Z)\mathrm{dist}(Z,P)^2\de\nu^{(\infty)}_{0,R}(Z)+F(\nu^{(\infty)}_{0,R})\\
            &\leq C\mathcal{F}_3(\nu^{(k)}_{0,RR_k},\nu^{(\infty)}_{0,R})+\frac{\varepsilon_0^2}{2}.
        \end{split}
    \end{equation}
    Now, by \eqref{singular set - scaling pre and post blow up} we have $\nu^{(k)}_{0,RR_k}\rightharpoonup\nu^{(\infty)}_{0,R}$ as $k\to\infty$, so by Proposition \ref{singular set - metric for weak convergence} and \eqref{singular set - bound 2} it follows that for all $k$ large enough, $F(\nu^{(k)}_{0,RR_k})<\varepsilon_0^2$.
    This contradicts \eqref{singular set - limsup is big}, proving that \eqref{singular set - tilde nu big F} holds. This implies that $\nu^{(\infty)}$ is not flat.

    Next, let $\nu,Y_k$ and $Y$ be as in the statement of the proposition. We know by Lemma \ref{pseudo tangents - characterization of support} and Proposition \ref{pseudo tangents - pseudo tangents are uniform} that $Y\in\mathrm{spt}(\nu)$ and $\nu$ is $m$-uniform with $\nu(B(0,1))=1$. We show that
    \begin{equation}
        \label{singular set - conclusion 2} \mu^\Lambda_{X_k,r_k}\rightharpoonup \nu_{Y,1}.
    \end{equation}
    Note first that
    $$T^\Lambda_{X_k,r_k}=T_{Y_k,1}\circ\Lambda(X_k)^{-1}\circ\Lambda(X)\circ T^\Lambda_{X,r_k}.$$
    Therefore, if $\varphi\in C_c(\R^{n+1})$ and $A_k=\Lambda(X_k)^{-1}\circ\Lambda(X)$, then
    \begin{equation}
        \label{singular set - expression 1}
        \begin{split}
            \int\varphi\de\mu^\Lambda_{X_k,r_k}&=\frac{1}{\mu(B_\Lambda(X_k,r_k))}\int\varphi\de(T^\Lambda_{X_k,r_k}[\mu])\\
            &=\frac{1}{\mu(B_\Lambda(X_k,r_k))}\int\varphi\circ T_{Y_k,1}\circ A_k\de(T^\Lambda_{X,r_k}[\mu])\\
            &=\frac{\mu(B_\Lambda(X,r_k))}{\mu(B_\Lambda(X_k,r_k))}\int\varphi\circ T_{Y_k,1}\circ A_k\de\mu^\Lambda_{X,r_k}.
        \end{split}
    \end{equation}
    By an argument  as the one leading up to \eqref{pseudo tangents - reference for singular set section}, we have $\mu(B_\Lambda(X,r_k))/\mu(B_\Lambda(X_k,r_k))\to 1$ as $k\to\infty$. On the other hand, since $A_k\to \mathrm{Id}$ and $T_{Y_k,1}\to T_{Y,1}$ uniformly on compact sets as $k\to\infty$, it follows that $\varphi\circ T_{Y_k,1}\circ A_k\to\varphi\circ T_{Y,1}$ uniformly as $k\to\infty$. Let $M>0$ be such that $\mathrm{spt}(\varphi\circ T_{Y_k,1}\circ A_k)\subset B(0,M)$ for all $k$. The fact that $\mu^\Lambda_{X,r_k}\rightharpoonup\nu$ as $k\to\infty$ implies that the sequence $\mu^\Lambda_{X,r_k}(\overline{B(0,M)})$ is bounded. Therefore, for $\varepsilon>0$, if $k$ is large enough,
    \begin{equation}
        \label{}
        \begin{split}
            \left|\int\varphi\right.&\circ\left.  T_{Y_k,1}\circ  A_k\de\mu^\Lambda_{X,r_k}-\int\varphi\de\nu_{Y,1}\right|\\
            &\leq \left|\int\varphi\circ T_{Y_k,1}\circ A_k\de\mu^\Lambda_{X,r_k}-\int\varphi\circ T_{Y,1}\de\mu^\Lambda_{X,r_k}\right|+\left|\int\varphi\circ T_{Y,1}\de\mu^\Lambda_{X,r_k}-\int\varphi\circ T_{Y,1}\de\nu\right|\\
            &\leq \varepsilon \mu^\Lambda_{X,r_k}(\overline{B(0,M)})+\varepsilon\leq C\varepsilon.
        \end{split}
    \end{equation}
    Thus
    $$\int\varphi\circ T_{Y_k,1}\circ A_k\de\mu^\Lambda_{X,r_k}\to\int\varphi\de\nu_{Y,1},$$
    as $k\to\infty$, and it then follows from \eqref{singular set - expression 1} that $\int\varphi\de\mu^\Lambda_{X_k,r_k}\to \int\varphi\de\nu_{Y,1}$ as $k\to\infty$, proving \eqref{singular set - conclusion 2}.

    Finally, we use $\nu^{(\infty)}$ to show that $\nu$ has a tangent measure at $Y$ that is not flat, so that by Lemma \ref{singular set - lambda tangents at singular points}
    \begin{equation}
        \label{singular set - singular point}Y\in\mathcal{S}_\nu.
    \end{equation}
    Let $\rho_k=\sigma_k/r_k$, and note that $\rho_k\to 0$ as shown above. We know that $\nu$ is $m$-uniform, so we may assume upon passing to a subsequence that $\nu_{Y,\rho_k}\rightharpoonup\alpha$, where $\alpha$ is a tangent measure of $\nu$ at $Y$ with $\alpha(B(0,1))=1$. We will prove that $\alpha$ is not flat. By \eqref{singular set - conclusion 2} and since $\rho_k\to 0$, for every $k$ we can find $\ell_k$ such that if $\ell\geq \ell_k$, then
    \begin{equation}
        \label{singular set - bound 6}
        \mathcal{F}_1(\mu^\Lambda_{X_\ell,r_\ell},\nu_{Y,1})<\frac{\rho_k}{k}\nu_{Y,1}(B(0,\rho_k)),\quad \rho_\ell<\rho_k.
    \end{equation}
    We may assume without loss of generality that $\ell_{k+1}<\ell_{k}$ for all $k\geq 1$. Let $\tilde\tau_k=r_{\ell_k}\rho_k$ and $\tilde X_k=X_{\ell_k}$. We claim that
    \begin{equation}
        \label{singular set - bound 5}
        \mathcal{F}_R(\mu^\Lambda_{\tilde X_k,\tilde\tau_k},\nu_{Y,\rho_k})<\frac{2}{k}.
    \end{equation}
    In fact, if $\varphi\in\mathcal{L}(R)$ (see \eqref{prelim - distance between measures}),
    \begin{equation}
        \begin{split}
            \label{singular set - bound 7}\left|\int\right. &\varphi\left.\de\mu^\Lambda_{\tilde X_k,\tilde\tau_k}-\int\varphi\de\nu_{Y,\rho_k}\right|\\
            &\leq \left|\frac{1}{\mu(B_\Lambda(\tilde X_k,\tilde\tau_k)))}\int\varphi\de(T^\Lambda_{\tilde X_k,\tilde\tau_k}[\mu])-\frac{1}{\nu_{Y,1}(B(0,\rho_k))}\int\varphi\de(T_{0,\rho_k}[\nu_{Y,1}])\right|\\
            &\leq \left|\frac{\mu(B_\Lambda(\tilde X_k,r_{\ell_k}))}{\mu(B_\Lambda(\tilde X_k,\tilde\tau_k)))}\int\varphi\de(T_{ 0,\rho_k}[\mu^\Lambda_{\tilde X_k,r_{\ell_k}}])-\frac{1}{\nu_{Y,1}(B(0,\rho_k))}\int\varphi\de(T_{0,\rho_k}[\nu_{Y,1}])\right|\\
            &\leq I_1+I_2,
        \end{split}
    \end{equation}
    where
    $$I_1=\left|\frac{\mu(B_\Lambda(\tilde X_k,r_{\ell_k}))}{\mu(B_\Lambda(\tilde X_k,\tilde\tau_k)))}-\frac{1}{\nu_{Y,1}(B(0,\rho_k))}\right|\int\varphi\de(T_{0,\rho_k}[\mu^\Lambda_{\tilde X_k,r_{\ell_k}}]),$$
    $$I_2=\frac{1}{\nu_{Y,1}(B(0,\rho_k))}\left|\int\varphi\de(T_{0,\rho_k}[\mu^\Lambda_{\tilde X_k,r_{\ell_k}}])-\int\varphi\de(T_{0,\rho_k}[\nu_{Y,1}])\right|.$$
    By the doubling property of $\mu$, we have
    \begin{equation*}
        \label{}
        \begin{split}
 \left|\frac{\mu(B_\Lambda(\tilde X_k,r_{\ell_k}))}{\mu(B_\Lambda(\tilde X_k,\tilde\tau_k)))}-\frac{1}{\nu_{Y,1}(B(0,\rho_k))}\right|=\left|\frac{\mu(B_\Lambda(\tilde X_k,r_{\ell_k}))}{\mu(B_\Lambda(\tilde X_k,\tilde\tau_k)))}-\frac{r_{\ell_k}}{\tilde\tau_k}\right|\to 0,
        \end{split}
    \end{equation*}
    as $k\to\infty$. Moreover, since $\varphi\in\mathcal{L}(R)$, we have $|\varphi\circ T_{0,\rho_k}|\leq CR$ for some $C>0$, and since  $\varphi\circ T_{0,\rho_k}$ is continuous and supported in $B(0,R)$ for $k$ large, we have
    $$\int\varphi\de(T_{0,\rho_k}[\mu^\Lambda_{\tilde X_k,r_{\ell_k}}])=\int\varphi\circ T_{0,\rho_k}\de\mu^\Lambda_{\tilde X_k,r_{\ell_k}}$$
    is bounded, so $I_1\to 0$ as $k\to\infty$. On the other hand, using \eqref{singular set - bound 6} we have for all $k$ large enough,
    \begin{equation}
        \label{}
        \begin{split}
            I_2&\leq \frac{1}{\rho_k\nu_{Y,1}(B(0,\rho_k))}\mathcal{F}_{R\rho_k}(\mu^\Lambda_{\tilde X_k,r_{\ell_k}},\nu_{Y,1})\leq \frac{1}{\rho_k\nu_{Y,1}(B(0,\rho_k))}\mathcal{F}_1(\mu^\Lambda_{X_{\ell_k},r_{\ell_k}},\nu_{Y,1})\leq\frac{1}{k},
        \end{split}
    \end{equation}
    so $I_2\to 0$ as well.

    Thus, \eqref{singular set - bound 7} implies \eqref{singular set - bound 5}, and since $\nu_{Y,\rho_k}\rightharpoonup\alpha$, we deduce that $\mu^\Lambda_{\tilde X_k,\tilde\tau_k}\rightharpoonup\alpha$. But we also know that $\mu^\Lambda_{X_k,\sigma_k}\rightharpoonup\nu^{(\infty)}$, and $\nu^{(\infty)}$ is not flat. Therefore, by Lemma \ref{singular set - connectedness property}, $\alpha$ cannot be flat, which implies \eqref{singular set - singular point} and completes the proof.
\end{proof}

\end{document}